\documentclass[11pt]{article}
\usepackage{amssymb,amsmath,epsfig}
\usepackage[colorlinks=false,breaklinks=true,linkcolor=blue]{hyperref}
\usepackage{graphics,psfrag,graphicx,color}
\usepackage{arydshln}
\usepackage{float,hhline}
\usepackage{comment}
\usepackage{cite}
\numberwithin{equation}{section}
\newcommand{\ds}{\displaystyle}

\def\<{\langle}
\def\>{\rangle}
\newcommand{\bgamma}{{\boldsymbol\gamma}}

\newcommand{\blambda}{{\boldsymbol\lambda}}
\newcommand{\bLambda}{{\boldsymbol\Lambda}}

\newcommand{\bbeta}{{\boldsymbol\eta}}
\newcommand{\bsi}{{\boldsymbol\sigma}}

\newcommand{\bphi}{{\boldsymbol\phi}}

\newcommand{\bvarphi}{{\boldsymbol\varphi}}
\newcommand{\bpsi}{{\boldsymbol\psi}}

\newcommand{\btau}{{\boldsymbol\tau}}
\newcommand{\bzeta}{{\boldsymbol\zeta}}

\newcommand{\bchi}{{\boldsymbol\chi}}

\newcommand{\btheta}{{\boldsymbol\theta}}

\newcommand{\brho}{{\boldsymbol\rho}}

\newcommand{\bv}{{\mathbf{v}}}
\newcommand{\bw}{{\mathbf{w}}}
\newcommand{\f}{\mathbf{f}}

\newcommand{\bi}{\mathbf{i}}

\newcommand{\bp}{\mathbf{p}}
\newcommand{\bq}{\mathbf{q}}
\newcommand{\br}{\mathbf{r}}
\newcommand{\bu}{\mathbf{u}}

\newcommand{\bt}{{\mathbf{t}}}
\newcommand{\bn}{{\mathbf{n}}}
\newcommand{\be}{{\mathbf{e}}}
\def\bs{\mathbf{s}}
\newcommand{\0}{{\mathbf{0}}}

\def\bF{\mathbf{F}}
\def\bG{\mathbf{G}}
\def\bK{\mathbf{K}}
\def\bI{\mathbf{I}}

\def\bV{\mathbf{V}}
\def\bW{\mathbf{W}}

\def\bT{\mathbf{T}}

\def\bP{\mathbf{P}}
\def\bQ{\mathbf{Q}}
\def\bS{\mathbf{S}}

\def\bx{\mathbf{x}}
\newcommand{\bL}{\mathbf{L}}
\newcommand\bH{\mathbf{H}}

\newcommand\bbN{\mathbb{N}}
\newcommand\bbM{\mathbb{M}}

\newcommand\bbP{\mathbb{P}}
\newcommand\bbQ{\mathbb{Q}}

\newcommand\bbS{\mathbb{S}}
\newcommand\bbH{\mathbb{H}}
\newcommand\bbX{\mathbb{X}}

\newcommand\bbL{\mathbb{L}}
\newcommand{\bR}{\mathcal{\mathbf{R}}}

\newcommand{\cA}{\mathcal{A}}
\newcommand{\cB}{\mathcal{B}}

\newcommand{\cE}{\mathcal{E}}

\newcommand{\cT}{\mathcal{T}}
\newcommand{\cJ}{\mathcal{J}}
\newcommand{\cK}{\mathcal{K}}

\newcommand{\cL}{\mathcal{L}}
\newcommand{\cM}{\mathcal{M}}
\newcommand{\cN}{\mathcal{N}}
\newcommand{\cD}{\mathcal{D}}
\newcommand{\cO}{\mathcal{O}}

\newcommand{\cR}{\mathcal{R}}

\def\H{\mathrm{H}}
\def\L{\mathrm{L}}

\def\R{\mathrm{R}}
\def\U{\mathrm{U}}
\def\V{\mathrm{V}}
\def\X{\mathrm{X}}
\def\rd{\mathrm{d}}

\def\rD{\mathrm{D}}
\def\rN{\mathrm{N}}
\def\rP{\mathrm{P}}
\def\W{\mathrm{W}}
\def\esssup{\mathrm{ess\,sup}}
\def\bBDM{\mathbf{BDM}}
\def\bbBDM{\mathbb{BDM}}
\def\rt{\mathrm{t}}

\def\BJS{\mathtt{BJS}}
\def\dc{\mathrm{dc}}
\def\bdiv{\mathbf{div}}

\def\tr{\mathrm{tr}}

\def\div{\mathrm{div}}
\def\dist{\mathrm{dist}\,}

\def\pil{\left<}
\def\pir{\right>}

\def\sk{\mathrm{sk}}

\def\qin{{\quad\hbox{in}\quad}}
\def\qon{{\quad\hbox{on}\quad}}
\def\qan{{\quad\hbox{and}\quad}}

\def\wt{\widetilde}
\def\wh{\widehat}
\def\ov{\overline}
\def\dt{\partial_t}
\def\dts{\partial_t^s}

\newcommand{\CKo}{C_{\rm Ko}}

\newtheorem{thm}{Theorem}[section]
\newtheorem{rem}{Remark}[section]
\newtheorem{lem}[thm]{Lemma}

\newenvironment{proof}{\noindent{\it Proof.}}{\hfill$\square$}

\numberwithin{equation}{section}
\numberwithin{figure}{section}
\numberwithin{table}{section}

\allowdisplaybreaks

\title{An augmented fully-mixed formulation for the quasistatic Navier--Stokes--Biot model}

\author{
{\sc Tongtong Li}\thanks{Department of Mathematics, University of Pittsburgh, Pittsburgh, PA 15260, USA, email: {\tt tol24@pitt.edu}. Supported in part by NSF grants DMS 1818775 and DMS 2111129.}	
\quad
{\sc Sergio Caucao}\thanks{Departamento de Matem\'atica y F\'isica Aplicadas, 
Universidad Cat\'olica de la Sant\'isima Concepci\'on, Casilla 297, Concepci\'on, Chile,
and Grupo de Investigaci\'on en An\'alisis Num\'erico y C\'alculo Cient\'ifico, GIANuC$^2$, Concepci\'on, Chile, email: {\tt scaucao@ucsc.cl}. Supported in part by ANID-Chile through the project {\sc Centro de Modelamiento Mate\-m\'ati\-co} (FB210005) and Fondecyt 11220393.}
\quad
{\sc Ivan Yotov}\thanks{Department of Mathematics, University of Pittsburgh, Pittsburgh, PA 15260, USA, email: {\tt yotov@math.pitt.edu}. Supported in part by NSF grants DMS 1818775 and DMS 2111129.}}

\date{\today}

\begin{document}
	
\maketitle
	
\begin{abstract}
\noindent 
We introduce and analyze a partially augmented fully-mixed formulation and a mixed finite element method for the coupled problem arising in the interaction between a free fluid and a poroelastic medium. The flows in the free fluid and poroelastic regions are governed by the Navier--Stokes and Biot equations, respectively, and the transmission conditions are given by mass conservation, balance of fluid force, conservation of momentum, and the Beavers--Joseph--Saffman condition. We apply dual-mixed formulations in both domains, where the symmetry of the Navier--Stokes and poroelastic stress tensors is imposed in an ultra-weak and weak sense. In turn, since the transmission conditions are essential in the fully mixed formulation, they are imposed weakly by introducing the traces of the structure velocity and the poroelastic medium pressure on the interface as the associated Lagrange multipliers. Furthermore, since the fluid convective term requires the velocity to live in a smaller space than usual, we augment the variational formulation with suitable Galerkin type terms. Existence and uniqueness of a solution are established for the continuous weak formulation, as well as a semidiscrete continuous-in-time formulation with non-matching grids, together with the corresponding stability bounds and error analysis with rates of convergence. Several numerical experiments are presented to verify the theoretical results and illustrate the performance of the method for applications to arterial flow and flow through a filter.
\end{abstract}
		
\maketitle
	
	
\section{Introduction}
	
The interaction between free fluid and flow in adjacent deformable poroelastic medium, referred to as fluid--poroelastic structure interaction (FPSI), is motivated by a variety of applications, such as modeling of blood flow, design of industrial filters, and cleanup of groundwater flow in aquifers, to name a few. The free fluid flow is typically modeled by the Stokes or the Navier--Stokes equations, with the Navier--Stokes equations being more suitable for fast flows. The fluid flow within the poroelastic medium is modeled by the Biot system, which takes into account the effect of the deformation of the medium on the flow and vice versa. The two regions are coupled across the interface through dynamic and kinematic transmission conditions. The FPSI problem exhibits features of coupled Stokes--Darcy flows and fluid--structure interaction (FSI). 

One of the first works on the analysis of the Stokes--Biot problem is \cite{s2005}, where the coupled system is resolved by semigroup methods for a suitable variational formulation. Numerical studies for the coupled Navier--Stokes and Biot system are presented in \cite{bqq2009}, where both monolithic solvers and heterogeneous domain decomposition strategies are considered. In \cite{byz2015} a non-iterative operator splitting scheme for a Navier--Stokes--Biot model with non-mixed Darcy formulation is developed. The approach is extended in \cite{Bukac-JCP} to coupling between fluid, elastic structure, and poroelastic material. Mixed Darcy formulations, where the continuity of flux condition is of essential type, are considered in \cite{bukavc2015}, using the Nitsche's interior penalty method, and in \cite{akyz2018}, using a Lagrange multiplier method. Well-posedness for the fully dynamic Navier--Stokes--Biot system with a non-mixed Darcy formulation is established in \cite{cesm2017}.
A nonlinear Stokes--Biot system for non-Newtonian fluids is analyzed in \cite{aeny2019} by means of a reduced parabolic-type system for the pressure and stress in the poroelastic region and classical results on nonlinear monotone operators in Sobolev space setting. A numerical scheme for the Stokes--Biot model with inf-sup stable Stokes elements for the Biot displacement--pressure pair is developed in \cite{Cesm-Chid}. A Stokes--Biot model with a total pressure formulation is studied in \cite{Stokes-Biot-eye}. Well-posedness for a Stokes--Biot system with a multilayered porous medium using Rothe's method is obtained in \cite{Bociu-etal-2021}. A Lagrange multiplier method for a fully dynamic Navier--Stokes--Biot system with a mixed Darcy formulation is developed in \cite{Wang-Yotov}. Additional works include optimization-based decoupling method \cite{Cesm-etal-optim}, a second order in time split scheme \cite{Kunwar-etal}, dimensionally reduced model for flow through fractures
\cite{Buk-Yot-Zun-fracture}, coupling with transport \cite{fpsi-transport}, and porohyperelastic media \cite{Seboldt-etal-2021}. All of the above-mentioned works utilize displacement formulations for the elasticity equation. In a recent work \cite{fpsi-mixed-elast}, the first mathematical and numerical analysis of a stress--displacement mixed elasticity formulation for the Stokes--Biot model is presented. More recently, a fully mixed formulation of the quasistatic Stokes--Biot model based on dual mixed formulations for Darcy, elasticity, and Stokes is developed in \cite{fpsi-msfmfe}. The resulting three-field dual mixed Stokes formulation and five-field dual mixed Biot formulation lead to the development of a multipoint stress--flux mixed finite element method that can be reduced to a positive definite cell-centered pressure--velocities--traces system. This approach is extended numerically to the Navier--Stokes--Biot system in \cite{fpsi-fvca}.

In this paper we consider the quasistatic Navier--Stokes--Biot model. The model is better suitable than the Stokes--Biot model for fast flows that may occur in many applications, including blood flow, flows through industrial filters, and coupling of surface and subsurface flows. The problem is much harder from mathematical point of view, due to the nonlinear convective term in the Navier--Stokes equations. Only two of the above mentioned works, \cite{cesm2017} and \cite{Wang-Yotov}, deal with the analysis of the weak formulation or the numerical approximation of the Navier--Stokes--Biot model. Both consider the fully dynamic problem and utilize a velocity--pressure Navier--Stokes formulation and a displacement-based elasticity formulation. In this paper, we combine techniques developed in \cite{fpsi-mixed-elast}, \cite{fpsi-msfmfe}, \cite{cot2016}, and \cite{cort2017} to study a fully-mixed formulation of the quasistatic Navier--Stokes--Biot model, which is based on dual mixed formulations for all three components -- Navier--Stokes, Darcy, and elasticity. To deal with the nonlinearity, we consider a pseudostress-based formulation for the Navier--Stokes equations. Such formulations allow for a unified analysis for Newtonian and non-Newtonian flows \cite{cgot2016,cgos2017}. Here, similarly to \cite{cot2016}, we introduce a nonlinear pseudostress tensor combining the fluid stress tensor with the convective term. Together with the fluid velocity, it yields a pseudostress--velocity Navier--Stokes formulation. 
Furthermore, in order to control the fluid variables in their natural norms, i.e., the norms associated with the differential operators in the strong form of the equations, and avoid the need for inf-sup stable finite elements, we augment the mixed formulation with some redundant Galerkin-type terms arising from the equilibrium and constitutive equations. In particular, the fluid stress is in H(div) and the fluid velocity is in H$^1$, resulting in smooth and accurate finite element approximations of both variables. We further note that the computational overhead due to adding the stabilization terms is minimal, since they do not involve additional variables. For the Biot system we employ a five-field dual mixed formulation based on the model developed in \cite{lee2016}, and studied in \cite{fpsi-mixed-elast} and \cite{fpsi-msfmfe} for the Stokes--Biot model. In particular, we use a velocity--pressure Darcy formulation and a weakly symmetric stress--displacement--rotation elasticity formulation. While we focus on weakly symmetric elasticity, which in certain cases allows for stress and rotation elimination and a reduction to an efficient cell-centered displacement system \cite{akny2018-a,msfmfe-Biot,fpsi-msfmfe}, our methodology also applies to the strongly symmetric stress--displacement elasticity formulation and the resulting four-field mixed Biot formulation \cite{Yi-Biot-mixed}. In turn, the transmission conditions consisting of mass conservation, balance of fluid force, conservation of momentum, and the Beavers--Joseph--Saffman slip with friction condition are imposed weakly through the introduction of two Lagrange multipliers: the traces of the structure velocity and the Darcy pressure on the interface. The advantages of the resulting fully-mixed formulation for the Navier--Stokes--Biot model include local mass conservation for the Darcy fluid, local momentum conservation for the poroelastic stress, accurate approximations for the Darcy velocity, the poroelastic stress, and the fluid pseudostress with continuous normal components across element edges or faces, locking-free behavior, and robustness with respect to the physical parameters. We emphasize that accurate and locally conservative stress computations are important in many applications, including flows in fractured subsurface formations and blood flow, which is one of the numerical examples we present in Section~6.

The main contributions of this paper are as follows. Since the proposed augmented fully-mixed formulation is new, we first study its well-posedness. Because the model is quasistatic, it is not possible to utilize the theory of ordinary differential equations for the semi-discrete Galerkin approximation, in contrast to \cite{cesm2017} and \cite{Wang-Yotov} where the fully dynamic problem is considered. Instead, we rewrite the system as a parabolic problem for the poroelastic stress and Darcy pressure and employ the classical semigroup theory for differential equations with monotone operators \cite{Showalter}, combined with a fixed point approach for the solvability of the resolvent system. We then present a semidiscrete continuous-in-time formulation based on employing stable mixed finite element spaces for the Navier--Stokes, Darcy, and elasticity equations with possibly non-matching grids along the interface, together with suitable choices for the Lagrange multiplier finite element spaces. Well-posedness and stability analysis results are established using a similar argument to the continuous case. We then develop error analysis and establish rates of convergence for all variables. We further present a fully discrete finite element method based on the backward Euler time discretization and give a roadmap for its analysis. Finally, we present several numerical experiments to verify the theoretical rates of convergence and illustrate the behavior of the method for modeling blood flow in an arterial bifurcation as well as air flow through a filter.

The rest of the paper is organized as follows. The remainder of this section describes standard notation and functional spaces to be employed throughout the paper. In Section 2 we introduce the mathematical model, whereas in Section 3 we derive the continuous weak formulation and establish some stability properties for the associated operators. Section 4 is devoted to the well-posedness of the continuous weak formulation, where, a suitable fixed point approach is applied to establish existence, uniqueness and stability of the solution. The semidiscrete continuous-in-time approximation is introduced and analyzed in Section 5, including its well-posedness, stability and error analysis. The fully discrete scheme is presented at the end of the section. Numerical experiments are presented in Section~6, followed by conclusions in Section~7. 

We end this section by introducing some definitions and fixing some notation. 
Let $\bbM$, $\bbS$ and $\bbN$ denote the sets of $n\times n$ matrices, 
$n\times n$ symmetric matrices and $n\times n$ skew-symmetric matrices, respectively. 
For a bounded domain $\cO \subset \R^n$, $n\in\{2,3\}$, standard notation is adopted for  Lebesgue spaces  $\mathrm L^p(\cO)$, Hilbert spaces $\mathrm H^k(\cO)$, and Sobolev spaces $\mathrm W^{k,p}(\cO)$. By $\mathbf{Z}$ and $\mathbb{Z}$ we denote the corresponding vectorial and tensorial counterparts of a generic scalar functional space $\mathrm Z$. 
The $\L^2(\cO)$ inner product for scalar, vector, or tensor valued functions is denoted by $(\cdot,\cdot)_{\cO}$. For a section of the boundary $\Gamma$, the $\L^2(\Gamma)$ inner product or duality pairing is denoted by $\pil\cdot,\cdot\pir_\Gamma$. For a Banach space $\V$, we denote its dual space by $\V'$. For an operator $\cA:\V \to \U'$, its adjoint operator is denoted by $\cA':\U \to \V'$. For any vector fields $\bv=(v_i)_{i=1,\ldots,n}$ and $\bw=(w_i)_{i=1,\ldots,n}$, we set the gradient, symmetric part of the gradient, divergence, and tensor product operators, as
\begin{equation*}
\nabla\bv:=\left(\frac{\partial v_i}{\partial x_j}\right)_{i,j=1,\ldots,n},\quad 
\be(\bv) := \frac{1}{2} \big( \nabla\bv + (\nabla\bv)^\rt \big),\quad
\div (\bv):=\sum_{j=1}^n \frac{\partial v_j}{\partial x_j},\qan 
\bv\otimes\bw:=(v_i w_j)_{i,j=1,\ldots,n}.
\end{equation*}
Furthermore, for any tensor field $\btau:=(\tau_{ij})_{i,j=1,\ldots,n}$ and 
$\bzeta:=(\zeta_{ij})_{i,j=1,\ldots,n}$, we define the transpose,
the trace, the tensor inner product, and the deviatoric tensor, respectively, as	
\begin{equation}\label{trace-dev-etc}
\btau^\rt := (\tau_{ji})_{i,j=1,\ldots,n},\quad \tr(\btau):=\sum_{i=1}^n \tau_{ii},\quad \btau:\bzeta:=\sum_{i,j=1}^n \tau_{ij}\zeta_{ij},\qan \btau^\rd:=\btau-\frac{1}{n}\tr(\btau)\bI\,,
\end{equation}
where $\bI$ is the identity matrix in $\R^{n\times n}$.
In addition, we recall the Hilbert space
\begin{equation*}
\bH(\div;\cO):=\Big\{ \bv\in\bL^2(\cO) :\quad \div(\bv)\in \L^2(\cO) \Big\},
\end{equation*}
equipped with the norm $\|\bv\|^2_{\bH(\div;\cO)} :=
\|\bv\|^2_{\bL^2(\cO)} + \|\div(\bv)\|^2_{\L^2(\cO)}$.
The space of matrix
valued functions whose rows belong to $\bH(\div;\cO)$ is denoted
by $\bbH(\bdiv;\cO)$ and endowed with the norm
$\|\btau\|^2_{\bbH(\bdiv;\cO)} := \|\btau\|^2_{\bbL^2(\cO)} +
\|\bdiv(\btau)\|^2_{\bL^2(\cO)}$.  Finally, given a separable Banach
space $\V$ endowed with the norm $\| \cdot \|_{\V}$, we introduce
the Bochner spaces $\L^2(0,T;\V)$, $\H^s(0,T;\V)$, with integer $s \ge 1$,
$\L^{\infty}(0,T;\V)$, and $\W^{1,\infty}(0,T;\V)$, endowed with the norms
\begin{equation*}
\begin{array}{c}
\ds \|f\|^{2}_{\L^{2}(0,T;\V)} \,:=\, \int^T_0 \|f(t)\|^{2}_{\V} \,dt\,,\quad
\|f\|^2_{\H^s(0,T;\V)} \,:=\, \int^T_0 \sum^{s}_{i=0} \|\partial^{i}_t f(t)\|^2_{\V}\,dt\,, \\[3ex]
\ds \|f\|_{\L^\infty(0,T;\V)} \,:=\, \mathop{\esssup}\limits_{t\in [0,T]} \|f(t)\|_{\V}\,,\quad
\|f\|_{\W^{1,\infty}(0,T;\V)} \,:=\, \mathop{\esssup}\limits_{t\in [0,T]} \left\{ \|f(t)\|_{\V} + \|\partial_t f(t)\|_{\V}  \right\}\,.
\end{array}
\end{equation*}

\section{The model problem}

Let $\Omega\subset \R^n$, $n\in\{2,3\}$ be a Lipschitz domain with polytopal boundary, which is 
subdivided into two non-overlapping and possibly non-connected regions: 
a fluid region $\Omega_f$ and a poroelastic region $\Omega_p$.
Let $\Gamma_{fp} = \partial\Omega_f\cap\partial\Omega_p$ denote the (nonempty) 
interface between these regions and let $\Gamma_f = \partial\Omega_f\setminus\Gamma_{fp}$ 
and $\Gamma_p = \partial\Omega_p\setminus \Gamma_{fp}$ denote 
the external parts of the boundary $\partial\Omega$.
We denote by $\bn_f$ and $\bn_p$ the unit normal vectors which point outward 
from $\partial\Omega_f$ and $\partial\Omega_p$, respectively, 
noting that $\bn_f = - \bn_p$ on $\Gamma_{fp}$.
Let $(\bu_\star,p_\star)$ be the velocity--pressure pair in $\Omega_\star$ 
with $\star\in\{f,p\}$, and let $\bbeta_p$ be the displacement in $\Omega_p$.
Let $\mu>0$ be the fluid viscosity, let $\rho$ be the density, let $\f_\star$ be 
the body force terms, which do not depend on time, and let $q_p$ be external source or sink term. The flow in $\Omega_f$ is governed by the Navier--Stokes equations:
\begin{subequations}\label{eq:Navier-Stokes-1}
\begin{gather}
\ds \rho\,(\nabla\bu_f)\,\bu_f - \bdiv(\bsi_f) = \f_f,\quad
\div(\bu_f) = 0 \qin \Omega_f\times (0,T] \,, \label{eq:Navier-Stokes-1-a}\\[1ex]
\ds (\bsi_f - \rho\,(\bu_f\otimes\bu_f))\,\bn_f=\0  \qon \Gamma_f^\rN\times (0,T], \quad
\bu_f = \0 \qon \Gamma_f^\rD\times (0,T] \,, \label{eq:Navier-Stokes-1-b}
\end{gather}
\end{subequations}
where $\bsi_f := -\,p_f\,\bI + 2\,\mu\,\be(\bu_f)$ denotes the stress tensor 
and $\Gamma_f=\Gamma_f^\rN \cup \Gamma_f^\rD$.
While the standard Navier--Stokes equations are presented above to describe 
the behavior of the fluid in $\Omega_f$, in this work we make use of 
an equivalent version of \eqref{eq:Navier-Stokes-1} based on the introduction 
of a pseudostress tensor combining the stress tensor $\bsi_f$ with the convective term.
More precisely, analogously to \cite{cort2017,cgot2016,cot2016,cgos2017}, and \cite{gov2020}, 
we introduce the nonlinear-pseudostress tensor
\begin{equation}\label{eq:nonlinear-stress-Tf}
\bT_f := \bsi_f - \rho\,(\bu_f\otimes\bu_f) =  -\,p_f\,\bI + 2\,\mu\,\be(\bu_f) - \rho\,(\bu_f\otimes\bu_f) \qin \Omega_f\times (0,T] \,.
\end{equation}
In this way, applying the matrix trace to the tensor $\bT_f$, and utilizing 
the incompressibility condition $\div(\bu_f) = 0$ in $\Omega_f\times (0,T]$, one arrives at 
\begin{equation}\label{eq:pressure-pf}
p_f = -\frac{1}{n}\,\left( \tr(\bT_f)  + \rho\,\tr(\bu_f\otimes\bu_f) \right) \qin \Omega_f\times (0,T] \,.
\end{equation}
Hence, replacing back \eqref{eq:pressure-pf} into \eqref{eq:nonlinear-stress-Tf}, and using the definition of the deviatoric operator \eqref{trace-dev-etc}, we obtain $\bT_f^\rd = 2\,\mu\,\be(\bu_f) - \rho\,(\bu_f\otimes\bu_f)^\rd$. Therefore
\eqref{eq:Navier-Stokes-1} can be rewritten, equivalently, 
as the set of equations with unknowns $\bT_f$ and $\bu_f$, given by
\begin{subequations}\label{eq:Navier-Stokes-2}
\begin{gather}
\ds \frac{1}{2\,\mu}\,\bT_f^\rd = \nabla\bu_f - \bgamma_f(\bu_f) - \frac{\rho}{2\,\mu}\,(\bu_f\otimes\bu_f)^\rd,\quad
-\,\bdiv(\bT_f) = \f_f,\quad \bT_f = \bT^\rt_f \qin \Omega_f\times (0,T] \,, \label{eq:Navier-Stokes-2-a} \\[1ex]
\ds \bT_f \bn_f = \0 \qon \Gamma_f^\rN \times (0,T], \quad 
\bu_f = \0 \qon \Gamma_f^\rD\times (0,T] \,, \label{eq:Navier-Stokes-2-c}
\end{gather}
\end{subequations}
where $\bgamma_f(\bu_f) := \dfrac{1}{2}\,\big( \nabla\bu_f - (\nabla\bu_f)^\rt \big)$ 
is the vorticity (or the skew-symmetric part of the velocity gradient tensor $\nabla \bu_f$). 
Notice that, as suggested by \eqref{eq:pressure-pf}, $p_f$ is eliminated from the present formulation and can be computed afterwards in terms of $\bT_f$ and $\bu_f$.
In addition, the fluid stress $\bsi_f$ can be recovered from \eqref{eq:nonlinear-stress-Tf}. For simplicity we assume that $|\Gamma^\rN_f| > 0$, which will allow us to control $\bT_f$ by $\bT_f^\rd$, cf. \eqref{eq:tau-H0div-Xf-inequality}.
The case $|\Gamma^\rN_f| = 0$ can be handled as in \cite{gmor2014,gos2011,gov2020} by introducing an additional variable corresponding to the mean value of $\tr(\bT_f)$. We further note that it is also possible to consider the boundary condition $\bsi_f\bn_f = \0$ on $\Gamma_f^\rN$, leading to the Robin-type boundary condition $\bT_f \bn_f + \rho \, (\bu_f\otimes\bu_f)\bn_f = \0$ on $\Gamma_f^\rN$. In this case the space for $\bT_f$ is unrestricted on $\Gamma_f^\rN$ and the third and fourth terms in \eqref{eq:continuous-weak-formulation-1j} below become $\pil\bT_f\bn_f,\bv_f\pir_{\Gamma_{fp}\cup \Gamma^\rN_f} + \rho\,\pil\bu_f\cdot\bn_f,\bu_f\cdot\bv_f\pir_{\Gamma_{fp}\cup \Gamma^\rN_f}$, which can be handled in the same way. In addition, the control of $\bT_f$ by $\bT_f^\rd$ can be achieved similarly to the case $|\Gamma^\rN_f| = 0$.

In turn, let $\bsi_e$ and $\bsi_p$ be the elastic and poroelastic stress tensors, respectively:
\begin{equation}\label{def:bsie-bsip}
A(\bsi_e) \,= \be(\bbeta_p) \qan
\bsi_p \,:=\, \bsi_e - \alpha_p\,p_p\,\bI \qin \Omega_p\times (0,T] \,,
\end{equation}
where $0 < \alpha_p \leq 1$ is the Biot--Willis constant, and $A: \bbS \rightarrow \bbM$ 
is the symmetric and positive definite compliance tensor,
satisfying, for some $0 < a_{\min} \le a_{\max} < \infty$,
\begin{equation}\label{eq:A-bounds}
\forall\, \btau\in\R^{n\times n}, \quad a_{\min}
\, \btau : \btau \, \leq \, A(\btau):\btau \, \leq \, a_{\max} \,
\btau : \btau \quad \forall\, \bx\in\Omega_p.
\end{equation}
In the isotropic case $A$ has the form, for all symmetric tensors $\btau$,
\begin{equation}\label{def:operator-A}
A(\btau) := \frac{1}{2\,\mu_p}\,\left(\btau - \frac{\lambda_p}{2\,\mu_p + n\,\lambda_p}\,\tr(\btau)\,\bI\right),\quad\mbox{with}\quad
A^{-1}(\btau) = 2\,\mu_p\,\btau + \lambda_p\,\tr(\btau)\,\bI,
\end{equation}
where
$\ds 0<\lambda_{\min}\leq \lambda_p(\bx)\leq \lambda_{\max}$ and 
$\ds 0<\mu_{\min} \leq\mu_p(\bx) \leq\mu_{\max}$ are the Lam\'e parameters.
In this case, 
$\bsi_e \,:=\, \lambda_p\,\div(\bbeta_p)\,\bI + 2\,\mu_p\,\be(\bbeta_p)$,
$\ds a_{\min}=\frac{1}{2 \mu_{\max} + n \, \lambda_{\max}}$, and 
$\ds a_{\max}=\frac{1}{2 \mu_{\min}}$. As in \cite{lee2016}, we extend the definition of $A$ on $\bbM$ such that it is a positive constant multiple of the identity map on $\bbN$. 
The poroelasticity region $\Omega_p$ is governed by 
the quasistatic Biot system \cite{b1941}:
\begin{subequations}\label{eq:Biot-model}
\begin{gather}
\ds -\,\bdiv(\bsi_p) = \f_p,\quad 
\mu\,\bK^{-1}\,\bu_p + \nabla\,p_p = \0 \qin \Omega_p\times(0,T], \label{eq:Biot-model-a} \\
\ds \frac{\partial}{\partial t}\left( s_0\,p_p + \alpha_p\,\div(\bbeta_p) \right) + \div(\bu_p) = q_p \qin \Omega_p\times(0,T], \label{eq:Biot-model-b} \\[1ex]
\ds \bu_p\cdot\bn_p = 0 \qon \Gamma^\rN_p\times (0,T],\quad 
  p_p = 0 \qon \Gamma^\rD_p\times (0,T], \label{eq:Biot-model-c} \\[1ex]
    \ds \bsi_p\bn_p = \0 \qon \tilde\Gamma^\rN_p \times (0,T], \quad
    \bbeta_p = \0 \qon \tilde\Gamma^\rD_p\times (0,T], \label{eq:Biot-model-d}
\end{gather}
\end{subequations}
where $\Gamma_p = \Gamma^\rN_p\cup \Gamma^\rD_p = \tilde\Gamma^\rN_p\cup \tilde\Gamma^\rD_p$,
$s_0 > 0$ is a constant storage coefficient 
and $\bK$ the symmetric and uniformly positive definite rock permeability tensor, 
satisfying, for some constants $0< k_{\min}\leq k_{\max}$,
\begin{equation}\label{eq:K-bounds}
\forall\, \bw\in\R^n \quad k_{\min}\,\bw \cdot \bw \leq (\bK\bw)\cdot \bw \leq k_{\max}\,\bw\cdot \bw \quad \forall\, \bx\in\Omega_p.
\end{equation}
We consider a range $0 < s_{0,\min} \le s_0 \le s_{0,\max}$ for the storage coefficient. Since locking in poroelasticity may occur for small values of $s_0$, in the analysis we explicitly track the dependence of the constants on $s_{0,\min}$ and note that they may depend on $s_{0,\max}$. To avoid the issue with restricting the mean value 
of the pressure, we assume that $|\Gamma^\rD_p| > 0$.
We also assume that $\Gamma^\rD_p$ and $\tilde\Gamma^\rD_p$ are not adjacent 
to the interface $\Gamma_{fp}$, i.e., $\exists\,s>0$ such that 
$\dist(\Gamma^\rD_p,\Gamma_{fp}) \geq s >0$ and $\dist(\tilde\Gamma^\rD_p,\Gamma_{fp}) \geq s >0$.
This assumption is used to simplify the characterization 
of the normal trace spaces on $\Gamma_{fp}$.

Next, we introduce the transmission conditions on the interface $\Gamma_{fp}$:
\begin{subequations}\label{eq:interface-conditions-1}
\begin{gather}
\ds \bu_f\cdot\bn_f + \left(\frac{\partial\bbeta_p}{\partial t} + \bu_p\right)\cdot\bn_p = 0, \quad 
\bsi_f\bn_f + \bsi_p\bn_p = \0 \qon \Gamma_{fp}\times (0,T], \label{eq:interface-conditions-1-a}\\
  \ds \bsi_f\bn_f + \mu \,\alpha_{\BJS}\sum^{n-1}_{j=1}\,\sqrt{\bK^{-1}_j}
  \left(\left(\bu_f - \frac{\partial\bbeta_p}{\partial t}\right)\cdot\bt_{f,j}\right)\,\bt_{f,j} = -\,p_p\bn_f \qon \Gamma_{fp}\times (0,T], \label{eq:interface-conditions-1-b}
\end{gather}
\end{subequations}
where $\bt_{f,j}$, $1\leq j\leq n-1$, is an orthogonal system of unit tangent 
vectors on $\Gamma_{fp}$, $\bK_j = (\bK\,\bt_{f,j})\cdot\bt_{f,j}$, and $\alpha_{\BJS} \geq 0$ is an experimentally determined friction coefficient.
The equations in \eqref{eq:interface-conditions-1-a} correspond
to mass conservation and conservation of momentum on $\Gamma_{fp}$, respectively, 
whereas \eqref{eq:interface-conditions-1-b} can be decomposed into its normal 
and tangential components, as follows:
\begin{equation*}\label{eq:interface-conditions-2}
(\bsi_f\bn_f)\cdot\bn_f = -\,p_p,\quad
(\bsi_f\bn_f)\cdot\bt_{f,j} = -\mu\,\alpha_{\BJS}\,\sqrt{\bK^{-1}_j}\left(\bu_f - \frac{\partial\bbeta_p}{\partial t}\right)\cdot\bt_{f,j} \qon \Gamma_{fp}\times (0,T],
\end{equation*}
representing balance of force and the Beavers--Joseph--Saffman (BJS) slip with friction condition, respectively.
The second equation in \eqref{eq:interface-conditions-1-a} and 
\eqref{eq:interface-conditions-1-b} can be rewritten in terms of tensor $\bT_f$ as follows:
\begin{subequations}\label{eq:interface-conditions-3}
\begin{gather}
\ds \bT_f\bn_f + \rho (\bu_f\otimes\bu_f)\bn_f + \bsi_p\bn_p = \0 \qon \Gamma_{fp}\times (0,T], \label{eq:interface-conditions-3a} \\[1ex]
\ds \bT_f\bn_f + \rho (\bu_f\otimes\bu_f)\bn_f + \mu\,\alpha_{BJS}\sum^{n-1}_{j=1} \sqrt{\bK^{-1}_j}\left(\left(\bu_f - \frac{\partial\bbeta_p}{\partial t}\right)\cdot\bt_{f,j}\right) \bt_{f,j} = - p_p\bn_f \,\,\,\mbox{ on }\,\, \Gamma_{fp}\times (0,T] . \label{eq:interface-conditions-3b} 
\end{gather}
\end{subequations}

Finally, the above system of equations is complemented by the initial condition
$p_p(\bx,0) = p_{p,0}(\bx)$ in $\Omega_p$.
In Lemma~\ref{lem:initial-condition} below we will construct compatible initial data for the rest of the variables from $p_{p,0}$ in a way that all equations in the system \eqref{eq:Navier-Stokes-2}--\eqref{eq:interface-conditions-3}, except for the unsteady conservation of mass equation in \eqref{eq:Biot-model-b}, hold at $t=0$.

\section{The weak formulation}

In this section we proceed analogously to \cite[Section~3]{aeny2019} 
(see also \cite{cgos2017,gmor2014}) and derive a weak formulation of 
the coupled problem given by \eqref{eq:Navier-Stokes-2}, \eqref{eq:Biot-model}, 
\eqref{eq:interface-conditions-1} and \eqref{eq:interface-conditions-3}.

\subsection{Preliminaries}\label{sec:preliminaries}

We first introduce further notation and definitions. Given $\star\in\big\{ f, p \big\}$, we set
\begin{equation*}
(p,w)_{\Omega_\star} := \int_{\Omega_\star} p\,w,\quad 
(\bu,\bv)_{\Omega_\star} := \int_{\Omega_\star} \bu\cdot\bv \qan
(\bT,\bR)_{\Omega_\star} := \int_{\Omega_\star} \bT:\bR.
\end{equation*}
In addition, similarly to \cite{cgot2016,cgos2017}, in the sequel we will employ 
the following Hilbert spaces to deal with the nonlinear pseudostress tensor and 
velocity of the Navier--Stokes equation, respectively:
\begin{equation*}
\bbX_f := \Big\{ \bR_f\in \bbH(\bdiv; \Omega_f) :\quad \bR_f\bn_f=\0 \qon \Gamma_f^\rN  \Big\},\quad 
\bV_f := \Big\{ \bv_f\in \bH^1(\Omega_f) :\quad \bv_f = \0 \qon \Gamma_f^\rD \Big\},
\end{equation*}
endowed with the corresponding norms
\begin{equation*}
\|\bR_f\|_{\bbX_f} := \|\bR_f\|_{\bbH(\bdiv;\Omega_f)},\quad
\|\bv_f\|_{\bV_f} := \|\bv_f\|_{\bH^1(\Omega_f)}\,.
\end{equation*}
For the unknowns in the Biot region we introduce the following Hilbert spaces:
\begin{equation*}
\begin{array}{c}
\bbX_p := \Big\{\btau_p \in \bbH(\bdiv;\Omega_p) : \ \btau_p\bn_p = \0 \ \ \hbox{on} \ \ \tilde\Gamma^\rN_p \Big\},\quad \bV_s:=\bL^2(\Omega_p),\quad
\bbQ_p := \Big\{ \bchi_p\in \bbL^2(\Omega_p) : \ \bchi^\rt_p = -\,\bchi_p \Big\}, \\[2ex]
\bV_p := \Big\{ \bv_p\in \bH(\div;\Omega_p) : \ \bv_p\cdot\bn_p = 0 \ \ \hbox{on} \ \ \Gamma^\rN_p \Big\},\quad
\W_p := \L^2(\Omega_p),
\end{array}
\end{equation*}
endowed with the standard norms
\begin{equation*}
\begin{array}{c}
\ds \|\btau_p\|_{\bbX_p} := \|\btau_p\|_{\bbH(\bdiv;\Omega_p)},\quad
\|\bv_s\|_{\bV_s} := \|\bv_s\|_{\bL^2(\Omega_p)},\quad
\|\bchi_p\|_{\bbQ_p} := \|\bchi_p\|_{\bbL^2(\Omega_p)}, \\ [2ex]
\ds \|\bv_p\|_{\bV_p} := \|\bv_p\|_{\bH(\div;\Omega_p)},\quad
\|w_p\|_{\W_p} := \|w_p\|_{\L^2(\Omega_p)}.
\end{array}
\end{equation*}

Finally, we need to introduce the spaces of traces $\Lambda_p := (\bV_p \cdot \bn_p|_{\Gamma_{fp}})'$ and $\bLambda_s := (\bbX_p \, \bn_p|_{\Gamma_{fp}})'$. According to the normal trace theorem, since $\bv_p\in \bV_p\subset \bH(\div;\Omega_p)$, then
$\bv_p\cdot\bn_p\in \H^{-1/2}(\partial \Omega_p)$.  It is shown in
\cite{galvis2007} that, if $\bv_p\cdot\bn_p = 0$ on
$\partial\,\Omega_p\setminus \Gamma_{fp}$, then $\bv_p\cdot\bn_p \in
\H^{-1/2}(\Gamma_{fp})$. This argument has been modified in \cite{akyz2018}
for the case $\bv_p\cdot\bn_p = 0$ on $\Gamma^\rN_p$ and 
$\dist(\Gamma^\rD_p,\Gamma_{fp}) \geq s > 0$. In particular, it holds that
\begin{equation}\label{eq:trace-inequality-1}
\pil \bv_p \cdot \bn_p, \xi \pir_{\Gamma_{fp}} 
\,\leq\, C\,\| \bv_p \|_{\bH(\div; \Omega_p)} \| \xi \|_{\H^{1/2}(\Gamma_{fp})} \quad \forall\,\bv_p \in \bV_p, \, \xi \in \H^{1/2}(\Gamma_{fp}).
\end{equation}
Similarly,
\begin{equation}\label{trace-sigma}
  \langle \btau_p \, \bn_p,\bphi \rangle_{\Gamma_{fp}}
  \le C \|\btau_p\|_{\bbH(\bdiv;\Omega_p)}\|\bphi\|_{\bH^{1/2}(\Gamma_{fp})}
  \quad \forall \, \btau_p \in \bbX_p, \,
  \bphi \in \bH^{1/2}(\Gamma_{fp}).
\end{equation}
Therefore we can take 
\begin{equation}\label{eq:trace-spaces}
\Lambda_p := \H^{1/2}(\Gamma_{fp}) \qan \bLambda_s := \bH^{1/2}(\Gamma_{fp}) 
\end{equation}
endowed with the norms
$\|\xi\|_{\Lambda_p} := \|\xi\|_{\H^{1/2}(\Gamma_{fp})}$ and 
$\|\bphi\|_{\bLambda_s} := \|\bphi\|_{\bH^{1/2}(\Gamma_{fp})}$.

\subsection{Lagrange multiplier weak formulation}

We now proceed with the derivation of the Lagrange multiplier weak
formulation for the coupling of the Navier--Stokes and Biot problems.
To this end, and inspired by \cite{aeny2019}, we begin by introducing 
the structure velocity $\bu_s := \partial_t\,\bbeta_p\in \bV_s$ and two Lagrange multipliers 
that represent the traces of the structure velocity and the Darcy pressure on the interface, respectively:
\begin{equation*}
\btheta := \bu_s|_{\Gamma_{fp}}\in \bLambda_s \qan
\lambda := p_p|_{\Gamma_{fp}}\in \Lambda_p,
\end{equation*}
where we use the notation $\ds \partial_t :=\frac{\partial}{\partial t}$.
In order to impose the symmetry of $\bsi_p$ in a weak sense, we introduce 
the rotation operator $\ds \brho_p:=\frac{1}{2}\big( \nabla\bbeta_p - (\nabla\bbeta_p)^\rt \big)$. In the weak formulation we will use its time derivative, that is, 
the structure rotation velocity
\begin{equation*}
\bgamma_p := \partial_t \brho_p
= \frac{1}{2}\big(\nabla\bu_s - (\nabla\bu_s)^\rt \big) \in \bbQ_p\,.
\end{equation*}
From the definition of 
the elastic and poroelastic stress tensors $\bsi_e$, $\bsi_p$ (cf. \eqref{def:bsie-bsip}) 
and recalling that $\bsi_e$ is connected to the displacement $\bbeta_p$ through 
the relation $A(\bsi_e)=\be(\bbeta_p)$, we deduce the identities
\begin{equation}\label{eq:div-bbeta_p}
\div(\bbeta_p) \,=\, \tr (\be(\bbeta_p)) \,=\, \tr(A(\bsi_e)) \,=\, \tr(A(\bsi_p + \alpha_p\,p_p\,\bI)) \,,
\end{equation}
\begin{equation}\label{eq:partial-t-A-identity}
\qan \partial_t\,A(\bsi_p + \alpha_p\,p_p\,\bI) 
\,=\, \nabla\,\bu_s - \bgamma_p\,.
\end{equation}
Then we test the first equation in \eqref{eq:Navier-Stokes-2-a}, the second equation of \eqref{eq:Biot-model-a}, and \eqref{eq:partial-t-A-identity} with arbitrary $\bR_f\in \bbX_f, \bv_p\in \bV_p$, and $\btau_p \in \bbX_p$, respectively, integrate by parts, and utilize the fact that $\bT^\rd_f:\bR_f = \bT^\rd_f:\bR^\rd_f$. We further test \eqref{eq:Biot-model-b} with $w_p\in \W_p$ employing \eqref{eq:div-bbeta_p} and impose the remaining equations weakly, as well as the symmetry of $\bsi_p$ and the transmission conditions in the first 
equation of \eqref{eq:interface-conditions-1-a} and \eqref{eq:interface-conditions-3} 
to obtain the following variational problem. Given
$\f_f\in \bL^2(\Omega_f)$, $\f_p\in \bL^2(\Omega_p)$, and 
$q_p:[0,T] \to \L^2(\Omega_p)$, find $(\bT_f, \bu_f, \bsi_p, \bu_s, \bgamma_p, \bu_p, p_p, \lambda, \btheta):[0,T] \rightarrow \bbX_{f} \times \bV_f \times \bbX_p \times \bV_s \times \bbQ_p \times \bV_p \times \W_p \times \Lambda_p \times \bLambda_s$, such that for all
$\bR_f\in \bbX_{f}, \bv_f \in \bV_f, \btau_p\in \bbX_p, \bv_s\in \bV_s, \bchi_p\in \bbQ_p, \bv_p\in \bV_p, w_p\in \W_p, \xi\in \Lambda_p, \bphi\in \bLambda_s$, and for a.e. $t \in (0,T)$,
\begin{subequations}\label{eq:continuous-weak-formulation-1}
  \begin{align}
&\ds \frac{1}{2\mu} (\bT^\rd_f,\bR^\rd_f)_{\Omega_f} + (\bu_f,\bdiv(\bR_f))_{\Omega_f}
+ (\bgamma_f(\bu_f),\bR_f)_{\Omega_f} 
- \pil\bR_f\bn_f,\bu_f \pir_{\Gamma_{fp}}
+ \frac{\rho}{2\mu}((\bu_f\otimes\bu_f)^\rd,\bR_f)_{\Omega_f} =   0, \label{eq:continuous-weak-formulation-1a}\\[1ex]
& 
\ds - (\bv_f,\bdiv(\bT_f))_{\Omega_f}
- (\bT_f,\bgamma_f(\bv_f))_{\Omega_f}
+ \pil\bT_f\bn_f,\bv_f\pir_{\Gamma_{fp}}
+ \rho\,\pil\bu_f\cdot\bn_f,\bu_f\cdot\bv_f\pir_{\Gamma_{fp}}
\nonumber \\
& \ds \qquad
+ \mu\,\alpha_{BJS}\,\sum_{j=1}^{n-1}\pil\sqrt{\bK_j^{-1}}
\left( \bu_f - \btheta \right)\cdot\bt_{f,j},\bv_f\cdot\bt_{f,j} \pir_{\Gamma_{fp}}
+ \pil\bv_f\cdot\bn_f,\lambda\pir_{\Gamma_{fp}} 
= \ds (\f_f,\bv_f)_{\Omega_f},
\label{eq:continuous-weak-formulation-1j} \\[1ex]
&(\partial_t\,A(\bsi_p+\alpha_p\,p_p\,\bI), \btau_p)_{\Omega_p} 
+ (\bu_s, \bdiv(\btau_p))_{\Omega_p} 
+ (\bgamma_p,\btau_p)_{\Omega_p} 
- \pil\btau_p\bn_p,\btheta\pir_{\Gamma_{fp}} = 0, \label{eq:continuous-weak-formulation-1d} \\[1ex]
&-(\bv_s,\bdiv(\bsi_p))_{\Omega_p}  =  (\f_p, \bv_s)_{\Omega_p}, \label{eq:continuous-weak-formulation-1e}\\[1ex]
&-(\bsi_p,\bchi_p)_{\Omega_p} = 0, \label{eq:continuous-weak-formulation-1f}\\[1ex]
&\ds \mu (\bK^{-1}\bu_p,\bv_p)_{\Omega_p} - (p_p,\div(\bv_p))_{\Omega_p} + \pil\bv_p\cdot\bn_p,\lambda\pir_{\Gamma_{fp}}  =  0, \label{eq:continuous-weak-formulation-1g}\\[1ex] 
&\ds s_0\,(\partial_t\,p_p, w_p)_{\Omega_p} + (\partial_t\,A(\bsi_p+\alpha_p\,p_p\,\bI), \alpha_p\,w_p\,\bI)_{\Omega_p}+ (w_p,\div(\bu_p))_{\Omega_p} = (q_p,w_p)_{\Omega_p},  \label{eq:continuous-weak-formulation-1h}\\[1ex]
&\ds -\pil\bu_f\cdot\bn_f + \left(\btheta + \bu_p\right)\cdot\bn_p,\xi\pir_{\Gamma_{fp}} = 0, \label{eq:continuous-weak-formulation-1i}\\
&\ds \pil\bsi_p\bn_p,\bphi\pir_{\Gamma_{fp}}
- \mu\,\alpha_{BJS}\,\sum_{j=1}^{n-1} \pil\sqrt{\bK_j^{-1}}\left( \bu_f - \btheta \right)\cdot\bt_{f,j},\bphi\cdot\bt_{f,j} \pir_{\Gamma_{fp}}
+ \pil\bphi\cdot\bn_p,\lambda\pir_{\Gamma_{fp}} = 0 . \label{eq:continuous-weak-formulation-1k} 
\end{align}
\end{subequations}
Note that \eqref{eq:continuous-weak-formulation-1a}--\eqref{eq:continuous-weak-formulation-1j} 
correspond to the Navier--Stokes equations, 
\eqref{eq:continuous-weak-formulation-1d}--\eqref{eq:continuous-weak-formulation-1f} 
are the elasticity equations, and \eqref{eq:continuous-weak-formulation-1g}--\eqref{eq:continuous-weak-formulation-1h} 
are the Darcy equations, whereas 
\eqref{eq:continuous-weak-formulation-1i}--\eqref{eq:continuous-weak-formulation-1k}, together with the interface terms in \eqref{eq:continuous-weak-formulation-1j},
enforce weakly the interface conditions. 
We will discuss the construction of initial conditions for the problem \eqref{eq:continuous-weak-formulation-1} later on in Lemma~\ref{lem:initial-condition}.

\begin{rem}\label{rem: time-diff}
The time differentiated equation \eqref{eq:continuous-weak-formulation-1d} allows us to
eliminate the displacement variable $\bbeta_p$ and obtain a
formulation that uses only $\bu_s$. As part of the analysis we will
construct suitable initial data such that, by integrating
\eqref{eq:continuous-weak-formulation-1d} in time, we can recover the original equation
\begin{equation}\label{non-diff-eq}
(A(\bsi_p+\alpha  p_p\bI),\btau_p)_{\Omega_p}
+(\bbeta_p, \bdiv(\btau_p))_{\Omega_p}
+(\brho_p,\btau_p)_{\Omega_p}
-\langle \btau_p \bn_p, \bpsi \rangle_{\Gamma_{fp}} = 0\,,
\end{equation}
where $\bpsi:= \bbeta_p|_{\Gamma_{fp}}$.
\end{rem}

We observe that, similarly to \cite[eq. (3.5)]{aggr2018} 
(see also \cite{gov2020} for an alternative approach) and since 
$\big\{ \bgamma_f(\bv_f) \,: \ \bv_f\in \bH^1(\Omega_f) \big\}$ is a proper-subspace 
of the skew-symmetric tensor space, the term $(\bT_f,\bgamma_f(\bv_f))_{\Omega_f}$ in 
\eqref{eq:continuous-weak-formulation-1j} 
imposes the symmetry of $\bT_f$ in an ultra-weak sense.
Notice also that the terms $((\bu_f\otimes\bu_f)^\rd,\bR_f)_{\Omega_f}$ in \eqref{eq:continuous-weak-formulation-1a} and
$\langle \bu_f\cdot\bn_f, \bu_f\cdot\bv_f \rangle_{\Gamma_{fp}}$ in 
\eqref{eq:continuous-weak-formulation-1j} require $\bu_f$ 
to live in a smaller space than $\bL^2(\Omega_f)$.
In fact, by applying the Cauchy--Schwarz and H\"older inequalities, 
the continuous injections $\bi_c$ of $\bH^1(\Omega_f)$ into $\bL^4(\Omega_f)$ 
and $\bi_\Gamma$ of $\bH^{1/2}(\partial \Omega_f)$ into $\bL^4(\partial \Omega_f)$, 
and the continuous trace operator $\gamma_0: \bH^1(\Omega_f)\to \bL^2(\partial \Omega_f)$, 
there hold
\begin{equation}\label{eq:injection-H1-into-L4}
  \big|((\bu_f\otimes\bw_f)^\rd,\bR_f)_{\Omega_f} \big| \leq \|\bi_c\|^2 \|\bw_f\|_{\bH^1(\Omega_f)} \|\bu_f\|_{\bH^1(\Omega_f)} \|\bR_f\|_{\bbL^2(\Omega_f)}
    \end{equation}
and
\begin{equation}\label{eq:injection-H1-into-L4-interface}
\big|\langle \bw_f\cdot\bn_f, \bu_f\cdot\bv_f \rangle_{\Gamma_{fp}}\big|
\leq\ \|\bi_{\Gamma}\|^2 \|\gamma_0\| \|\bw_f\|_{\bH^1(\Omega_f)} \|\bu_f\|_{\bH^1(\Omega_f)} \|\bv_f\|_{\bH^1(\Omega_f)},
\end{equation}
for all $\bu_f, \bv_f, \bw_f\in \bH^1(\Omega_f)$ and $\bR_f \in \bbL^2(\Omega_f)$.
Accordingly, we look for $\bu_f$ in $\bV_f$. We also have 
\begin{equation}\label{Rfnfvf}
  \big|\<\bR_f\bn_f,\bv_f\>_{\Gamma_{fp}}\big| \leq C \|\bR_f\|_{\bbH(\bdiv;\Omega_f)}\|\bv_f\|_{\bH^1(\Omega_f)} \quad
  \forall \, \bR_f \in \bbX_f, \ \bv_f \in \bV_f.
\end{equation}
In the case of $\Gamma_f^\rN$ adjacent to $\Gamma_{fp}$, \eqref{Rfnfvf} follows similarly to \eqref{trace-sigma}. In the case of $\Gamma_f^\rD$ adjacent to $\Gamma_{fp}$, it follows from
$\<\bR_f\bn_f,\bv_f\>_{\Gamma_{fp}} = \<\bR_f\bn_f,\tilde\bv_f\>_{\partial\Omega_f}$, where
$\tilde\bv_f \in \bH^{1/2}(\partial\Omega_f)$ is the extension by zero of $\bv_f|_{\Gamma_{fp}}$.
In addition, it holds that
\begin{align}
  &  \big|\<\bv_f\cdot\bt_{f,j},\bphi\cdot\bt_{f,j}\>_{\Gamma_{fp}} \big|\le \|\gamma_0\|\|\bv_f\|_{\bH^1(\Omega_f)} \|\bphi\|_{\bL^2(\Gamma_{fp})} \quad \forall \, \bv_f \in \bV_f, \ \bphi \in \bLambda_s, \label{bjs-cont} \\
  & \big|\<\bv_f\cdot\bn_f,\xi\>_{\Gamma_{fp}} \big|\le \|\gamma_0\|\|\bv_f\|_{\bH^1(\Omega_f)} \|\xi\|_{\L^2(\Gamma_{fp})} \quad \forall \, \bv_f \in \bV_f, \ \xi \in \Lambda_p. \label{bgamma-cont}
  \end{align}
Finally, in order to obtain control on $\bT_f$ in the $\bbH(\bdiv;\Omega_f)$-norm and on $\bu_f$ in the $\bH^1(\Omega_f)$-norm, we augment the system with the following redundant Galerkin-type terms:
\begin{subequations}\label{eq:redundant-galerkin-terms}
\begin{gather}
\ds \kappa_1\,\big(\bdiv(\bT_f) + \f_f,\bdiv(\bR_f)\big)_{\Omega_f} 
= 0 \quad \forall\,\bR_f\in \bbX_{f}\,, \label{eq:redundant-galerkin-terms-a} \\[1ex]
\ds \kappa_2\left(\be(\bu_f) - \frac{\rho}{2\,\mu}\,(\bu_f\otimes\bu_f)^\rd - \frac{1}{2\,\mu}\,\bT^\rd_f, \be(\bv_f)\right)_{\Omega_f} 
= 0 \quad \forall\,\bv_f\in \bV_f\,, \label{eq:redundant-galerkin-terms-b}
\end{gather}
\end{subequations}
where $\kappa_1$ and $\kappa_2$ are positive parameters to be specified later.
Notice that the above terms are consistent expressions arising from the equilibrium and constitutive equations. It is easy to see that each solution of the original system is also a solution of the augmented one, and hence by solving the latter we find all solutions of the former. We emphasize that without the augmented terms, it is not possible to control $\bT_f$ in the $\bbH(\bdiv;\Omega_f)$-norm and $\bu_f$ in the $\bH^1(\Omega_f)$-norm, so they are needed to obtain a well-posed formulation with the current choice of functional spaces.

There are many different ways of ordering the equations in \eqref{eq:continuous-weak-formulation-1}. For the sake of the subsequent analysis, we proceed as in \cite{aeny2019} and \cite{gov2020}, and adopt one leading to an evolution problem in a mixed form, by grouping the spaces, unknowns and test functions as follows:
\begin{equation*}
\begin{array}{c}
\ds \bQ := \bbX_{p}\times \W_p\times \bV_p\times \bbX_f \times \bV_f\times \bLambda_s,\quad
\bS := \Lambda_p \times \bV_s\times \bbQ_p \,, \\[1.5ex]
\ds \bp := (\bsi_p,  p_p, \bu_p, \bT_f, \bu_f, \btheta )\in \bQ,\quad 
\br := (\lambda, \bu_s, \bgamma_p )\in \bS \,, \\[1ex]
\ds \bq := (\btau_p,  w_p, \bv_p, \bR_f, \bv_f, \bphi )\in \bQ,\quad 
\bs := (\xi, \bv_s, \bchi_p )\in \bS \,,
\end{array}
\end{equation*}
where the spaces $\bQ$ and $\bS$ are respectively endowed with the norms
\begin{align*}
\|\bq\|_{\bQ}^2 & = \|\btau_p\|_{\bbX_p}^2 + \|w_p\|_{\W_p}^2 + \|\bv_p\|_{\bV_p}^2 + \|\bR_f\|_{\bbX_f}^2 + \|\bv_f\|_{\bV_f}^2 + \|\bphi\|_{\bLambda_s}^2, \\[1ex]
\|\bs\|_{\bS}^2 & = \|\xi\|_{\Lambda_p}^2 + \|\bv_s\|_{\bV_s}^2 + \|\bchi_p\|_{\bbQ_p}^2.
\end{align*}
Furthermore, given $\bw_f\in \bV_f$, we set the bilinear forms
\begin{subequations}\label{bilinear-forms}
\begin{align}
\ds & a_e(\bsi_p, p_p; \btau_p, w_p) :=(A(\bsi_p+\alpha_p p_p \bI), \btau_p + \alpha_p w_p \bI)_{\Omega_p},\quad
a_p(\bu_p,\bv_p) := \mu\,(\bK^{-1}\bu_p,\bv_p)_{\Omega_p}\,, \label{eq:bilinear-form-ae-ap} \\[0.5ex]
\ds & a_f(\bT_f,\bu_f;\bR_f,\bv_f) := \frac{1}{2\,\mu}\,(\bT^\rd_f,\bR^\rd_f)_{\Omega_f} + \kappa_1\,(\bdiv(\bT_f),\bdiv(\bR_f))_{\Omega_f} 
+ \kappa_2\,(\be(\bu_f), \be(\bv_f) )_{\Omega_f} \nonumber \\
\ds & \quad -\,\, \frac{\kappa_2}{2\,\mu}\left(\bT^\rd_f, \be(\bv_f)\right)_{\Omega_f}
+ (\bu_f,\bdiv(\bR_f))_{\Omega_f} - (\bv_f,\bdiv(\bT_f))_{\Omega_f} \nonumber \\
\ds & \quad +\,\, (\bgamma_f(\bu_f),\bR_f)_{\Omega_f} - (\bT_f,\bgamma_f(\bv_f))_{\Omega_f}
+ \pil\bT_f\bn_f,\bv_f\pir_{\Gamma_{fp}} - \pil\bR_f\bn_f,\bu_f\pir_{\Gamma_{fp}}\,, \label{eq:bilinear-form-af} \\[1ex]
\ds & \kappa_{\bw_f}(\bT_f,\bu_f;\bR_f,\bv_f) 
:= \frac{\rho}{2\,\mu} ((\bu_f\otimes\bw_f)^\rd, \bR_f - \kappa_2\,\be(\bv_f))_{\Omega_f} 
+ \rho \pil\bw_f\cdot\bn_f,\bu_f\cdot\bv_f\pir_{\Gamma_{fp}}\,, \label{eq:bilinear-form-k-wf} \\[1ex]
\ds & b_{\bn_p}(\btau_p,\bphi) := \,\pil\btau_p\bn_p,\bphi\pir_{\Gamma_{fp}},\quad
b_\sk(\bchi_p,\btau_p) := (\bchi_p,\btau_p)_{\Omega_p}, \label{eq:bilinear-form-bnp-bsk} \\[1ex] 
\ds & b_p(w_p, \bv_p) := -\,(w_p,\div(\bv_p))_{\Omega_p},\quad
b_s(\bv_s,\btau_p) := (\bv_s,\bdiv(\btau_p))_{\Omega_p}\,, \label{eq:bilinear-form-bp-bs}
\end{align}
\end{subequations}
and the interface terms
\begin{subequations}\label{int-bilinear-forms}
\begin{align}
\ds & a_{\BJS}(\bu_f,\btheta;\bv_f,\bphi) := \mu\,\alpha_{BJS}\,\sum^{n-1}_{j=1} \pil\sqrt{\bK_j^{-1}}(\bu_f - \btheta)\cdot\bt_{f,j},(\bv_f - \bphi)\cdot\bt_{f,j}\pir_{\Gamma_{fp}}\,, \label{eq:bilinear-form-aBJS} \\
\ds & b_\Gamma(\bv_p,\bv_f,\bphi;\xi) := \pil\bv_f\cdot\bn_f + (\bphi + \bv_p)\cdot\bn_p, \xi\pir_{\Gamma_{fp}}\,. \label{eq:bilinear-form-bGamma}
\end{align}
\end{subequations}
Hence, the Lagrange variational formulation for the system \eqref{eq:continuous-weak-formulation-1} and \eqref{eq:redundant-galerkin-terms}, results in
\begin{equation}\label{eq:continuous-weak-formulation-2}
\begin{array}{l}
\ds s_0 (\partial_t p_p,w_p)_{\Omega_p} 
+ a_e(\partial_t \bsi_p, \partial_t p_p; \btau_p, w_p) 
+ a_p(\bu_p,\bv_p) + a_f(\bT_f,\bu_f;\bR_f,\bv_f)
+ \kappa_{\bu_f}(\bT_f,\bu_f;\bR_f,\bv_f) \\[2ex]
\ds\quad +\,\, a_{\BJS}(\bu_f,\btheta;\bv_f,\bphi) 
+ b_p(p_p,\bv_p) - b_p(w_p,\bu_p)
+ b_{\bn_p}(\bsi_p,\bphi) - b_{\bn_p}(\btau_p,\btheta) \\[2ex]
\ds\quad +\,\, b_s(\bu_s,\btau_p) + b_\sk(\bgamma_p,\btau_p) + b_\Gamma(\bv_p,\bv_f,\bphi;\lambda) 
\,=\, (q_p,w_p)_{\Omega_p} + (\f_f,\bv_f - \kappa_1\,\bdiv(\bR_f))_{\Omega_f} \,, \\[2ex]
\ds -\,b_s(\bv_s,\bsi_p) - b_\sk(\bchi_p,\bsi_p) -  b_\Gamma(\bu_p,\bu_f,\btheta;\xi) 
\, = \,(\f_p,\bv_s)_{\Omega_p}\,.
\end{array}
\end{equation}
We can write \eqref{eq:continuous-weak-formulation-2} in an operator notation as a degenerate evolution problem in a mixed form:
\begin{equation}\label{eq:continuous-weak-formulation-3}
\begin{array}{rcll}
\ds \frac{\partial}{\partial t}\,\cE(\bp(t)) + (\cA + \cK_{\bu_f(t)})(\bp(t)) + \cB'(\br(t)) & = & \bF(t) & \mbox{ in }\, \bQ', \\[2ex]
\ds - \cB(\bp(t)) & = & \bG & \mbox{ in }\, \bS',
\end{array}
\end{equation}
where, given $\bw_f\in \bV_f$, the operators $\cE : \bQ \to \bQ'$, $\cA : \bQ\to \bQ'$, 
$\cK_{\bw_f} : \bQ\to \bQ'$, $\cB : \bQ\to \bS'$, and the functionals 
$\bF\in \bQ'$, $\bG\in \bS'$ are defined as follows:
  \begin{subequations}\label{operators}
\begin{align}
  & \cE(\bp)(\bq) := s_0 (p_p,w_p)_{\Omega_p} 
+ a_e(\bsi_p, p_p; \btau_p, w_p), \label{defn-E}\\
  & \cA(\bp)(\bq) :=  a_p(\bu_p,\bv_p) + a_f(\bT_f,\bu_f;\bR_f,\bv_f) + a_{\BJS}(\bu_f,\btheta;\bv_f,\bphi) \nonumber \\
& \qquad\qquad\quad  + b_p(p_p,\bv_p) - b_p(w_p,\bu_p) 
  + b_{\bn_p}(\bsi_p,\bphi) - b_{\bn_p}(\btau_p,\btheta), \\
  & \cK_{\bw_f}(\bp)(\bq)  := \kappa_{\bw_f}(\bT_f,\bu_f;\bR_f,\bv_f), \\
  & \cB(\bq)(\bs) := b_s(\bv_s,\btau_p) + b_\sk(\bchi_p,\btau_p)
  +  b_\Gamma(\bv_p,\bv_f,\bphi;\xi), \\
  & \bF(\bq) := (q_p,w_p)_{\Omega_p} + (\f_f,\bv_f - \kappa_1\,\bdiv(\bR_f))_{\Omega_f}, \\
  & \bG(\bs) := (\f_p,\bv_s)_{\Omega_p}.
\end{align}
\end{subequations}

\subsection{Stability properties}

Let us now discuss the continuity properties of the operators and functionals in \eqref{operators}.

\begin{lem}\label{lem:cont}
  The operators $\cE$, $\cA$, and $\cB$ are linear and continuous:
\begin{equation}\label{eq:continuity-bounds}
\big|\cE(\bp)(\bq)\big| \,\leq\, C_{\cE}\,\|\bp\|_{\bQ} \|\bq\|_{\bQ},\quad
\big|\cA(\bp)(\bq)\big| \,\leq\, C_{\cA}\,\|\bp\|_{\bQ} \|\bq\|_{\bQ},\quad
\big|\cB(\bq)(\bs)\big| \,\leq\, C_{\cB}\,\|\bq\|_{\bQ} \|\bs\|_{\bS},
\end{equation}
where the constant $C_{\cE}>0$ depends on $s_0, \alpha_p$, and $a_{\max}$,
whereas $C_{\cA}$ and $C_{\cB}$ are positive constants depending on $\mu$, $\bK$, 
$\rho$, $\alpha_{\BJS}$, $\kappa_1$ and $\kappa_2$. The operator
$\cK_{\bw_f}$ is linear and continuous:
\begin{equation}\label{eq:continuity-of-Kwf}
\big| \cK_{\bw_f}(\bp)(\bq) \big|
\,\leq\, C_{\cK}\,\|\bw_f\|_{\bV_f} \|\bp\|_{\bQ} \|\bq\|_{\bQ},
\end{equation}
where 
\begin{equation}\label{eq:C-K-constant}
C_{\cK} := \rho\,\bigg( \frac{1+\kappa_2}{2\,\mu}\|\bi_c\|^2 + \|\bi_{\Gamma}\|^2 \|\gamma_0\| \bigg).
\end{equation}
The linear functionals $\bF\in \bQ'$ and $\bG\in \bS'$ are continuous:
\begin{equation}\label{cont-F-G}  
  \big|\bF(\bq)\big| \le C_{\bF} \|\bq\|_{\bQ}, \quad \big|\bG(\bs)\big| \le C_{\bG} \|\bs\|_{\bS},
\end{equation}
with $C_{\bF} = \left(\|q_p\|_{\L^2(\Omega_p)}^2 + (1+\kappa_1^2)\|\f_f\|_{\bL^2(\Omega_f)}^2 \right)^{1/2}$ and $C_{\bG} = \|\f_p\|_{\bL^2(\Omega_p)}$.
\end{lem}

\begin{proof}
We first note that
\begin{equation}\label{eq:bound-for-evf-gammafvf}
\|\be(\bv_f)\|_{\bbL^2(\Omega_f)} 
\,\leq\, \|\bv_f\|_{\bV_f} \qan
\|\bgamma_f(\bv_f)\|_{\bbL^2(\Omega_f)} 
\,\leq\, \|\bv_f\|_{\bV_f} \quad \forall\,\bv_f \in \bH^1(\Omega_f).
\end{equation}
We recall that the operators and functionals are defined in \eqref{operators}, with the associated bilinear forms defined in \eqref{bilinear-forms} and \eqref{int-bilinear-forms}. 
The continuity of $\cE$ follows from \eqref{eq:A-bounds}. The continuity of $\cA$ follows from \eqref{eq:K-bounds}, \eqref{eq:bound-for-evf-gammafvf}, \eqref{Rfnfvf}, \eqref{bjs-cont}, and \eqref{trace-sigma}. The continuity of $\cB$ follows from \eqref{bgamma-cont} and \eqref{eq:trace-inequality-1}. For the continuity of $\cK_{\bw_f}$ for a given $\bw_f\in \bV_f$, using \eqref{eq:injection-H1-into-L4}--\eqref{eq:injection-H1-into-L4-interface} and \eqref{eq:bound-for-evf-gammafvf}, we deduce that
\begin{equation*}
\big| \cK_{\bw_f}(\bp)(\bq) \big|
\,\leq\, C_{\cK}\,\|\bw_f\|_{\bV_f} \|\bu_f\|_{\bV_f} \|(\bR_f,\bv_f)\|
\,\leq\, C_{\cK}\,\|\bw_f\|_{\bV_f} \|\bp\|_{\bQ} \|\bq\|_{\bQ}\,,
\end{equation*}
with $C_{\cK}$ defined in \eqref{eq:C-K-constant}, where
$$
\|(\bR_f, \bv_f)\|^2 := \|\bR_f\|_{\bbX_f}^2 + \|\bv_f\|_{\bV_f}^2.
$$
Finally, the continuity of $\bF$ and $\bG$ \eqref{cont-F-G} follows easily from their definitions.
\end{proof}

In the sequel, we make use of the Korn inequality: there exists a positive constant 
$\CKo$ such that
\begin{equation}\label{eq:bound-for-e(vf)}
\CKo\,\|\bv_f\|^2_{\bV_f} 
\,\leq\, \|\be(\bv_f)\|^2_{\bbL^2(\Omega_f)} \quad \forall \, \bv_f \in \bH^1(\Omega_f),
\end{equation}
as well as the following well-known estimates: there exist positive constants $c_1(\Omega_f)$ and $c_2(\Omega_f)$, such that (see, \cite[Proposition~IV.3.1]{Brezzi-Fortin} and \cite[Lemma~2.5]{Gatica}, respectively)
\begin{equation}\label{eq:tau-d-H0div-inequality}
c_1(\Omega_f)\,\|\bR_{f,0}\|^2_{\bbL^2(\Omega_f)} 
\,\leq\, \|\bR^\rd_f\|^2_{\bbL^2(\Omega_f)} + \|\bdiv(\bR_f)\|^2_{\bL^2(\Omega_f)} 
\quad \forall\,\bR_f = \bR_{f,0} + \ell\,\bI\in \bbH(\bdiv;\Omega_f)
\end{equation}
and
\begin{equation}\label{eq:tau-H0div-Xf-inequality}
c_2(\Omega_f)\,\|\bR_f\|^2_{\bbX_f} 
\,\leq\, \|\bR_{f,0}\|^2_{\bbX_f} 
\quad \forall\,\bR_f = \bR_{f,0} + \ell\,\bI\in \bbX_f\,,
\end{equation}
where $\bR_{f,0}\in \bbH_0(\bdiv;\Omega_f) := \Big\{ \bR_{f}\in \bbH(\bdiv;\Omega_f) :\quad (\tr(\bR_f),1)_{\Omega_f} = 0 \Big\}$ and $\ell\in \R$. 
We emphasize that \eqref{eq:tau-H0div-Xf-inequality}
holds since each $\bR_f \in \bbX_f$ satisfies the boundary condition
$\ds \bR_f\bn_f = \0$ on $\Gamma^\rN_f$ with $|\Gamma^\rN_f| > 0$. 

Next, we present a lemma that establishes positivity bounds for 
the operators $\cA+\cK_{\bw_f}$ and $\cE$. For any $r > 0$, let $\bW_r$ be the closed ball defined by
\begin{equation}\label{eq:W_r-definition}
\bW_r := \Big\{ \bw_f \in \bV_f: \quad \|\bw_f\|_{\bV_f} \leq r \Big\},
\end{equation}

\begin{lem}\label{lem:monotone}
Assume that $\ds \kappa_1\in (0,+\infty)$ and $\kappa_2 \in (0, 4\mu)$, 
and let $\bw_f \in \bW_r$ with $r\in (0,r_0)$ and
\begin{equation}\label{eq:r0-definition}
r_0:= \frac{\alpha_f}{2\,C_{\cK}},
\end{equation}
where $C_{\cK}$ is defined in \eqref{eq:C-K-constant} and $\alpha_f$ is defined in \eqref{constants} below. Then, $\cE$ and $\cA+\cK_{\bw_f}$ are monotone.
Moreover, there exists a constant $\alpha_{\cA \cK}>0$
depending on $\mu$, $\bK, \alpha_{\BJS}, \CKo, c_1(\Omega_f)$, and $c_2(\Omega_f)$, such that
\begin{equation}\label{eq:positive-bound-E}
\ds \cE(\bq)(\bq)\,=\,
s_0 \|w_p\|^2_{\L^2(\Omega_p)} + \|A^{1/2}(\btau_p + \alpha_p\,w_p\,\bI)\|^2_{\bbL^2(\Omega_p)} \quad \forall\,\bq\in \bQ \,,
\end{equation}
and
\begin{equation}\label{eq:positive-bound-A+Kwf}
(\cA + \cK_{\bw_f})(\bq)(\bq)
\,\geq\, \alpha_{\cA \cK} \left( \|\bv_p\|^2_{\bL^2(\Omega_p)} 
+ \|(\bR_f,\bv_f)\|^2 + |\bv_f - \bphi|^2_{\BJS} \right)
\quad \forall\,\bq\in \bQ\,.
\end{equation}
\end{lem}
\begin{proof}
First, \eqref{eq:positive-bound-E} follows in a straightforward way from the definition of the operator $\cE$ (cf. \eqref{defn-E}, \eqref{eq:bilinear-form-ae-ap}).
In addition, using \eqref{eq:positive-bound-E} and the fact that $\cE$ is linear, the monotonicity property is obtained.
In turn, from the definition of $a_f$ (cf. \eqref{eq:bilinear-form-af}), using 
Young's inequality and \eqref{eq:bound-for-evf-gammafvf}, 
and simple algebraic computations, we find that
\begin{equation*}
a_f(\bR_f,\bv_f;\bR_f,\bv_f) 
\,\geq\, \frac{1}{2\,\mu}\left( 1 - \frac{\kappa_2}{4\,\mu} \right) \|\bR_f^\rd\|^2_{\bL^2(\Omega_f)}
+ \kappa_1\,\|\bdiv(\bR_f)\|^2_{\bL^2(\Omega_f)}
+ \frac{\kappa_2}{2}\,\|\be(\bv_f)\|^2_{\bbL^2(\Omega_f)} 
\end{equation*}
for all $(\bR_f,\bv_f)\in \bbX_f\times \bV_f$.
Then, assuming the stipulated ranges on $\kappa_1$ and $\kappa_2$, and applying 
inequalities \eqref{eq:tau-d-H0div-inequality} and \eqref{eq:tau-H0div-Xf-inequality}, 
we can define the positive constants
\begin{equation} \label{constants}
\alpha_0 := \min\left\{ \frac{1}{2\,\mu}\left( 1 - \frac{\kappa_2}{4\,\mu} \right), \frac{\kappa_1}{2} \right\}\,,\quad
\alpha_1 := c_2(\Omega_f)\min\left\{ c_1(\Omega_f)\,\alpha_0, \frac{\kappa_1}{2} \right\},
\quad \alpha_f := \min \big\{ \alpha_1, \frac{\kappa_2}{2}\,\CKo \big\},
\end{equation}
which, together with the Korn inequality \eqref{eq:bound-for-e(vf)}, allows us to conclude
\begin{equation}\label{eq:coercivity-af}
a_f(\bR_f,\bv_f;\bR_f,\bv_f) 
\,\geq\, \alpha_1\,\|\bR_f\|^2_{\bbX_f}
+ \frac{\kappa_2}{2}\,\CKo\,\|\bv_f\|^2_{\bV_f}
\,\geq\, \alpha_f\,\|(\bR_f,\bv_f)\|^2.
\end{equation}
Next, combining \eqref{eq:coercivity-af} with \eqref{eq:continuity-of-Kwf} and 
the assumption $\|\bw_f\|_{\bV_f} \leq r$, with $r\in (0,r_0)$ defined by \eqref{eq:r0-definition}, 
we deduce that 
\begin{equation}\label{eq:monotone-01}
a_f(\bR_f,\bv_f;\bR_f,\bv_f) + \kappa_{\bw_f}(\bR_f,\bv_f;\bR_f,\bv_f)
\,\geq \left( \alpha_f \,-\, C_{\cK}\,\|\bw_f\|_{\bV_f}\right) \|(\bR_f,\bv_f)\|^2
\geq\, \frac{\alpha_f}{2}\,\|(\bR_f,\bv_f)\|^2\,. 
\end{equation}

Finally, from the definition of the bilinear forms $a_p$ and $a_{\BJS}$ (cf. \eqref{eq:bilinear-form-ae-ap}, \eqref{eq:bilinear-form-aBJS}), the estimate \eqref{eq:K-bounds} and simple computations, we obtain
\begin{equation}\label{eq:monotone-02}
\begin{array}{c}
\ds a_p(\bv_p,\bv_p)
\,\geq\, \mu \, k_{\max}^{-1} \|\bv_p\|^2_{\bL^2(\Omega_p)} \,,\qan  \\[2ex]
\ds a_{\BJS}(\bv_f,\bphi;\bv_f,\bphi)
\,=\, \mu\,\alpha_{BJS}\sum^{n-1}_{j=1} \pil\sqrt{\bK_j^{-1}}(\bv_f-\bphi)\cdot\bt_{f,j},(\bv_f-\bphi)\cdot\bt_{f,j}\pir_{\Gamma_{fp}}
\,\geq\, c_{\BJS}\,|\bv_f - \bphi|^2_{\BJS}\,,
\end{array}
\end{equation}
where, $|\bv_f - \bphi|^2_{\BJS} 
\,:=\, \sum^{n-1}_{j=1} \|(\bv_f - \bphi)\cdot\bt_{f,j}\|^2_{\L^2(\Gamma_{fp})}$ 
for all $(\bv_f,\bphi) \in \bV_f\times \bLambda_s$, 
and $c_{\BJS}$ is a positive constant that only depends on $\mu$, $\alpha_{\BJS}$ and $\bK$.
The monotonicity of $\cA + \cK_{\bw_f}$ and \eqref{eq:positive-bound-A+Kwf} follow from the fact that the forms 
$a_f, \kappa_{\bw_f}, a_p$, and $a_{\BJS}$ are linear, and the 
estimates \eqref{eq:monotone-01} and \eqref{eq:monotone-02}.
\end{proof}

\begin{rem}\label{rem:kappa}
  In the computations we choose a value of $\kappa_2$ in the middle of its admissible range $(0,4\mu)$: $\kappa_2 = 2\mu$, which results in all constants defined in \eqref{constants} being bounded strictly away from zero. We further set $\kappa_1 = \frac{1}{2\mu}$, which maximizes $\alpha_0$ and gives $\alpha_1 = \mathcal{O}(\frac{1}{\mu})$, providing strong control on $\|\bT_f\|_{\bbX_f}$ in the regime of small viscosity.
\end{rem}

Next, we provide inf-sup conditions for some operators involved 
in \eqref{eq:continuous-weak-formulation-2}, which will be used later on 
to derive stability bounds for the solution of \eqref{eq:continuous-weak-formulation-2}.
\begin{lem}\label{lem:inf-sup}
There exist constants $\beta_1, \beta_2>0$ such that for all $(\bv_s,\bchi_p,\bphi) \in \bV_s\times \bbQ_p\times \bLambda_s$,
\begin{equation}\label{eq:continuous-inf-sup-1}
\beta_1\,\left( \|\bv_s\|_{\bV_s} + \|\bchi_p\|_{\bbQ_p} + \|\bphi\|_{\bLambda_s} \right)
\,\leq\, 
\underset{\0 \neq \btau_p\in \bbX_p}{\sup}
\frac{b_s(\btau_p,\bv_s) + b_{sk}(\btau_p,\bchi_p) + b_{\bn_p}(\btau_p, \bphi)}{\|\btau_p\|_{\bbX_p}}\,,
\end{equation}
and for all $(w_p, \xi)\in \W_p\times \Lambda_p$,
\begin{equation}\label{eq:continuous-inf-sup-2}
\beta_2\,\left( \|w_p\|_{\W_p} + \|\xi\|_{\Lambda_p} \right)
\,\leq\, 
\underset{\0 \neq \bv_p \in \bV_p}{\sup}
\frac{b_p(\bv_p,w_p) + b_{\Gamma}(\bv_p,\0,\0;\xi)}{\|\bv_p\|_{\bV_p}} \,. 
\end{equation}
\end{lem}
\begin{proof}
The proof of \eqref{eq:continuous-inf-sup-1} follows from similar arguments to \cite[eq. (2.59), Section 2.4.3.2]{Gatica} for the elasticity problem with mixed boundary conditions,
whereas \eqref{eq:continuous-inf-sup-2} follows from a slight modification of \cite[Lemmas 3.1 and 3.2]{ervin2009} to account for $|\Gamma_p^\rD| > 0$. 
\end{proof}

\section{Well-posedness of the model}\label{sec:well-posedness-model}

In this section, we establish the well-posedness of \eqref{eq:continuous-weak-formulation-3} (equivalently \eqref{eq:continuous-weak-formulation-2}).

\subsection{Preliminaries}\label{section: well posedness prelim}

We begin by recalling the following key result to establish the existence of a solution to \eqref{eq:continuous-weak-formulation-3} (see \cite[Theorem~IV.6.1(b)]{Showalter} for details). In what follows, an operator $A$ from a real vector space $E$ to its algebraic dual $E'$ is symmetric and monotone if, respectively,
\begin{equation*}
A(x)(y) = A(y)(x) \qan
\big(A(x) - A(y)\big)(x - y) \geq 0 \quad \forall\,x, y\in E\,.
\end{equation*}
In addition,  $Rg(A)$ denotes the range of $A$.
\begin{thm}\label{thm:solvability-parabolic-problem}
Let the linear, symmetric and monotone operator $\cN$ be given from the real vector space $E$ to its algebraic dual $E'$, and let $E'_b$ be the Hilbert space which is the dual of $E$ with the seminorm
\begin{equation*}
|x|_b = \big(\cN(x)(x)\big)^{1/2} \quad x\in E.
\end{equation*}
Let $\cM\subset E\times E'_b$ be a relation with domain $\cD = \Big\{ x\in E \,:\, \cM(x) \neq \emptyset \Big\}$.
	
Assume $\cM$ is monotone and $Rg(\cN + \cM) = E'_b$.
Then, for each $u_0\in \cD$ and for each $f\in \W^{1,1}(0,T;E'_b)$, there is a solution $u$ of
\begin{equation}\label{eq: Showalter}
\frac{d}{dt}\big(\cN(u(t))\big) + \cM\big(u(t)\big) \ni f(t) \quad a.e. \, \ 0 < t < T,
\end{equation}
with
\begin{equation*}
\cN(u)\in \W^{1,\infty}(0,T;E'_b),\quad u(t)\in \cD,\quad \mbox{ for all }\, 0\leq t\leq T,\qan \cN(u(0)) = \cN(u_0)\,.
\end{equation*}
\end{thm}

\begin{rem}
The problem \eqref{eq:continuous-weak-formulation-3} is a degenerate 
evolution problem in a mixed form, which fits the structure of the problem 
\eqref{eq: Showalter} in Theorem \ref{thm:solvability-parabolic-problem}. 
However, $f$ is restricted to the space $\W^{1,1}(0,T;E'_b)$ arising from $\cN$. 
If we would like $u(t)$ in the theorem to represent all the variables 
in our case, we will have to restrict the data as $\f_f=\0$ and $\f_p=\0$. 
To avoid this restriction, we will reformulate the problem as 
a parabolic problem for $u=(\bsi_p,p_p)$ as in \cite{aeny2019}.
\end{rem}

Let $E: = \bbX_p \times \W_p$ and let $\cN: E \to E'$ be defined as, (cf. \eqref{defn-E}),
\begin{equation}\label{defn-N}
\cN(\bsi_p, p_p)(\btau_p, w_p) := s_0 (p_p,w_p)_{\Omega_p} + a_e(\bsi_p, p_p; \btau_p, w_p).
\end{equation}
From the definition of $a_e$ (cf. \eqref{eq:bilinear-form-ae-ap})
and the bounds on the operator $A$ (cf. \eqref{eq:A-bounds}), as well as
the fact that $s_0>0$, it follows that the norm induced by $\cN$ 
is equivalent to the $\L^2$ norm
$\big(\|\btau_p\|^2_{\bbL^2(\Omega_p)} + \|w_p\|^2_{\W_p}\big)^{1/2}$,
which implies that $E'_b = \bbL^2(\Omega_p)\times \L^2(\Omega_p) \subset \bbX'_p\times \W'_p$.
Now, let us set $\bQ_2' := \bbL^2(\Omega_p)\times \L^2(\Omega_p)\times \{\0\}\times \bL^2(\Omega_f)\times \bL^2(\Omega_f)\times \{\0\} \subset \bQ'$. 
Next, similarly to \cite[Section~4.1]{aeny2019}, we consider the domain associated with the resolvent system of \eqref{eq:continuous-weak-formulation-3} (cf. \eqref{eq:continuous-weak-formulation-2}). For $r\in (0,r_0)$ with $r_0$ given in \eqref{eq:r0-definition}, define

\medskip

$\cD :=\, \ds \Big\{  (\bsi_p, p_p)\in \bbX_p\times \W_p :$ \ \  for given  $(\f_f, \f_p)\in \bL^2(\Omega_f)\times \bL^2(\Omega_p)$, \\
\indent there exist $((\bu_p, \bT_f, \bu_f, \btheta),(\lambda,\bu_s, \bgamma_p))\in (\bV_p \times \bbX_f \times \bV_f \times \bLambda_s) \times \bS$ with $\bu_f \in \bW_r$, such that
\begin{align}
& \ds s_0\,(p_p,w_p)_{\Omega_p} + a_e( \bsi_p,  p_p; \btau_p, w_p) + a_p(\bu_p,\bv_p) + a_f(\bT_f,\bu_f;\bR_f,\bv_f)
+ \kappa_{\bu_f}(\bT_f,\bu_f;\bR_f,\bv_f) \nonumber \\[1ex]
& \quad\ds +\, a_{\BJS}(\bu_f,\btheta;\bv_f,\bphi) 
+ b_p(p_p,\bv_p) - b_p(w_p,\bu_p)
+ b_{\bn_p}(\bsi_p,\bphi) - b_{\bn_p}(\btau_p,\btheta) \label{eq:domain-D} \\[1ex]
& \quad\ds +\, b_s(\bu_s,\btau_p)+
b_\sk(\bgamma_p,\btau_p) +  b_\Gamma(\bv_p,\bv_f,\bphi;\lambda) 
= (\wh{\f}_p, \btau_p)_{\Omega_p} + (\wh{q}_p, w_p)_{\Omega_p} 
+ (\f_f,\bv_f - \kappa_1\,\bdiv(\bR_f))_{\Omega_f}, \nonumber \\[1ex]
& \ds -\,b_s(\bv_s,\bsi_p) - b_\sk(\bchi_p,\bsi_p) -  b_\Gamma(\bu_p,\bu_f,\btheta;\xi) 
= (\f_p,\bv_s)_{\Omega_p} \,, \nonumber
\end{align}
\indent for all $(\bq, \bs)\in \bQ\times \bS$ and for some $(\wh{\f}_p,\wh{q}_p)\in E_b'$ satisfying 
\begin{equation}\label{eq:domain-D-hat}
\ds \|\wh{\f}_p\|_{\bbL^2(\Omega_p)} + \|\wh{q}_p\|_{\L^2(\Omega_p)}
\,\leq\, \wh{C}_{ep}\,\left( \|\f_f\|_{\bL^2(\Omega_f)} + \|\f_p\|_{\bL^2(\Omega_p)}\right)
\end{equation}
\indent with $\wh{C}_{ep}$ a fixed positive constant \hspace{-7pt} $\Big\}$.

\medskip
\noindent The constant $\wh{C}_{ep}$ is determined in the construction of the initial data, which is required to be in the domain $\cD$, cf. \eqref{eq:initial-data-bound} and \eqref{eq:initial-hat-data-bound-again} below.

Note that the resolvent system \eqref{eq:domain-D} can be written in an operator form as
\begin{equation}\label{eq:resolvent-weak-formulation}
\arraycolsep=1.7pt
\begin{array}{rcll}
\ds (\cE + \cA + \cK_{\bu_f})(\bp) + \cB'(\br) & = & \wh{\bF} & \mbox{ in }\, \bQ_2', \\ [2ex]
\ds - \cB(\bp) & = & \bG & \mbox{ in }\, \bS',
\end{array}
\end{equation}
where $\wh{\bF} \in \bQ'_2\subset \bQ'$ is the functional on the right hand side of \eqref{eq:domain-D}, that is, 
\begin{equation}\label{def:hat-functions}
\wh{\bF}(\bq) = (\wh{\f}_p, \btau_p)_{\Omega_p} 
+ (\wh{q}_p, w_p)_{\Omega_p} 
+ (\f_f, \bv_f - \kappa_1\,\bdiv(\bR_f))_{\Omega_f} \quad \forall\,\bq\in \bQ \,,
\end{equation}
which, thanks to \eqref{eq:domain-D-hat}, is bounded by
\begin{equation}\label{eq:acotamiento-F-hat}
\big|\wh{\bF}(\bq)\big| 
\,\leq\, \left( (1 + \kappa_1^2 + 2\wh{C}_{ep}^2)\|\f_f\|_{\bL^2(\Omega_f)}^2
  + 2\wh{C}_{ep}^2 \|\f_p\|_{\bL^2(\Omega_p)}^2 \right)^{1/2} \|\bq\|_{\bQ}.
\end{equation}
Note that there may be more than one $(\wh{\f}_p, \wh{q}_p) \in E_b'$ that generate the same $(\bsi_p, p_p) \in \cD$. In view of this, we introduce the multivalued operator $\cM(\cdot)$ with domain $\cD$ defined by
\begin{equation}\label{eq:operator-M}
  \cM(\bsi_p,p_p) \,:=\, \Big\{ (\wh{\f}_p,\wh{q}_p)-\cN(\bsi_p, p_p) : \ (\bsi_p,p_p) \,\mbox{ satisfy }\, \eqref{eq:domain-D} \,\mbox{ for }\, (\wh{\f}_p,\wh{q}_p)\in E'_b
  \mbox{ satisfying \eqref{eq:domain-D-hat}}\Big\}\,,
\end{equation}
where $\cN$ is the operator defined in \eqref{defn-N}. 
We observe that the relation $\cM \subset E \times E_b'$ is associated with the domain $\cD$ in the sense that $[\bv,\f]\in \cM$ if $\bv\in \cD$ and $\f\in \cM(\bv)$.

Next, we establish a connection between  \eqref{eq:continuous-weak-formulation-2} and the following parabolic problem: Given 
$$(h_{\bsi_p}, h_{p_p})\in \W^{1,1}(0,T;\bbL^2(\Omega_p))\times \W^{1,1}(0,T;\linebreak\L^2(\Omega_p)),$$ 
find $(\bsi_p,p_p):[0,T]\to\cD$ satisfying
\begin{equation}\label{eq:parabolic-problem}
\frac{d}{dt} \, \cN \left(\begin{array}{c}
\bsi_p(t) \\ p_p(t)
\end{array}\right)
+
\cM\left(\begin{array}{c}
\bsi_p(t) \\ p_p(t)
\end{array}\right) 
\ni \left(\begin{array}{c}
h_{\bsi_p}(t) \\ h_{p_p}(t)
\end{array}\right)\quad a.e. \,\, t\in (0,T)\,.
\end{equation}

\begin{lem}\label{lem:well-posedness-2}
If $(\bsi_p, p_p):[0,T]\to\cD$ solves \eqref{eq:parabolic-problem} for 
$(h_{\bsi_p}, h_{p_p})=(\0, q_p)$ with $q_p\in \W^{1,1}(0,T;\L^2(\Omega_p))$,
then the associated solution to \eqref{eq:domain-D} also solves \eqref{eq:continuous-weak-formulation-2}.
\end{lem}
\begin{proof}
Let $(\bsi_p(t), p_p(t))\in \cD$ solve \eqref{eq:parabolic-problem} 
for $(h_{\bsi_p}, h_{p_p})=(\0, q_p)$. Note that the resolvent system \eqref{eq:domain-D} 
from the definition of the domain $\cD$ directly implies \eqref{eq:continuous-weak-formulation-2} 
when is tested with $\bq=(\0, 0, \bv_p, \bR_f, \bv_f, \linebreak \bphi)$ and 
$\bs=(\xi, \bv_s, \bchi_p)$. Thus it remains to show \eqref{eq:continuous-weak-formulation-2} 
with $\bq=(\btau_p, w_p, \0, \0, \0, \0)$ and $\bs=\0$.

Since $(\bsi_p(t), p_p(t))$ solves \eqref{eq:parabolic-problem} for $(h_{\bsi_p}, h_{p_p})=(\0, q_p)$, there exists $(\wh{\f}_p,\wh{q}_p)\in \bbL^2(\Omega_p)\times \L^2(\Omega_p)$ such that $(\wh{\f}_p, \wh{q}_p)-\cN(\bsi_p,p_p)\in \cM(\bsi_p,p_p)$ satisfies
\begin{equation*}\label{eq:parabolic-problem-specific}
\frac{d}{dt} \, \cN\left(\begin{array}{c}
\bsi_p \\ p_p
\end{array}\right)
+
\left(\begin{array}{c}
\wh{\f}_p \\ \wh{q}_p 
\end{array}\right) 
-\cN \left(\begin{array}{c}
\bsi_p \\ p_p
\end{array}\right) 
\,=\, \left(\begin{array}{c}
\0 \\ q_p
\end{array}\right)\,, 
\end{equation*}
which implies that for all $(\btau_p,w_p)\in \bbX_p\times \W_p$, there holds
\begin{equation}\label{eq:aux-parabolic-problem}
\ds \frac{d}{dt} \, \cN \left(\begin{array}{c}
\bsi_p \\ p_p
\end{array}\right)
\left(\begin{array}{c}
\btau_p \\ w_p
\end{array}\right)
+
\left(\left(\begin{array}{c}
\wh{\f}_p\\ \wh{q}_p 
\end{array}\right) - \cN \left(\begin{array}{c}
\bsi_p \\  p_p
\end{array}\right)\right)
\left(\begin{array}{c}
\btau_p \\ w_p
\end{array}\right)
\,=\, (q_p,w_p)_{\Omega_p}\,.
\end{equation}
In turn, using the definition of $\cN$ (cf. \eqref{defn-N}) and testing the first equation of \eqref{eq:domain-D} with $\bq = (\btau_p,w_p,\0,\0,\newline \0,\0)\in \bQ$, we deduce that
\begin{align*}
\ds \left(\left(\begin{array}{c}
\wh{\f}_p \\ \wh{q}_p 
\end{array}\right) 
-\cN \left(\begin{array}{c}
\bsi_p \\ p_p
\end{array}\right)\right)
\left(\begin{array}{c}
\btau_p \\ w_p
\end{array}\right)
&= (\wh{\f}_p, \btau_p)_{\Omega_p} + (\wh{q}_p,w_p)_{\Omega_p} - a_e(\bsi_p,p_p;\btau_p,w_p) - (s_0\,p_p,w_p)_{\Omega_p} \\[1ex]
&=\, -\,b_p(\bu_p,w_p) - b_{\bn_p}(\btau_p,\btheta) + b_s(\btau_p,\bu_s) + b_\sk(\bgamma_p, \btau_p),
\end{align*}
which, combined with \eqref{eq:aux-parabolic-problem}, yields
\begin{equation*}
\begin{array}{l}
\ds 
a_e(\partial_t \, \bsi_p, \partial_t\,p_p; \btau_p, w_p) + (s_0\,\partial_t\,p_p,w_p)_{\Omega_p} \\ [2ex]
\ds\quad -\,\,b_p(\bu_p,w_p) - b_{\bn_p}(\btau_p,\btheta) + b_s(\btau_p,\bu_s) + b_\sk(\bgamma_p, \btau_p) \,=\, (q_p,w_p)_{\Omega_p} \quad \forall\,(\btau_p,w_p)\in \bbX_p\times \W_p.
\end{array}
\end{equation*}
Therefore the first equation of \eqref{eq:continuous-weak-formulation-2} tested with $\bq=(\btau_p, w_p, \0, \0, \0, \0)$ holds, completing the proof.
\end{proof}

\subsection{Existence and uniqueness of a solution of the reduced parabolic problem}

We will utilize Theorem \ref{thm:solvability-parabolic-problem} to show that the problem \eqref{eq:parabolic-problem} has a solution, which will be used later on to prove the well-posedness of problem \eqref{eq:continuous-weak-formulation-3}. We proceed as follows.
\medskip

\noindent \textbf{Step 1.} 
Introduce a fixed-point operator $\cJ$ associated to problem \eqref{eq:resolvent-weak-formulation} and derive a continuity bound.

\noindent \textbf{Step 2.} 
Prove that $\cJ$ is a contraction mapping and conclude that the domain $\cD$ (cf. \eqref{eq:domain-D}) is nonempty.

\noindent \textbf{Step 3.} 
Show the solvability of the parabolic problem \eqref{eq:parabolic-problem}.

\subsubsection{Step 1: A fixed-point approach}

We begin the solvability analysis of \eqref{eq:resolvent-weak-formulation} 
or equivalently that the domain $\cD$ (cf. \eqref{eq:domain-D}) is nonempty 
by defining the operator $\cJ: \bV_f \rightarrow \bV_f$ as
\begin{equation}\label{eq:def-T}
\cJ(\bw_f):= \bu_f \quad \forall \, \bw_f \in \bV_f,
\end{equation}
where $\bp:=(\bsi_p, p_p, \bu_p, \bT_f, \bu_f, \btheta) \in \bQ$ is the first component of the unique solution (to be confirmed below) of the problem: Find $(\bp, \br)\in \bQ\times \bS$, such that
\begin{equation}\label{eq:NS-Biot-formulation-3}
\begin{array}{rcl}
\ds (\cE + \cA + \cK_{\bw_f})(\bp) + \cB'(\br) & = & \wh{\bF} \qin \bQ_2'\,, \\[1ex]
\ds - \cB(\bp) & = & \bG \qin \bS'\,.
\end{array}
\end{equation}
Thus, $(\bp, \br)\in \bQ\times \bS$ is a solution of (\ref{eq:domain-D}) if and only if $\bu_f\in \bV_f$ is a fixed-point of $\cJ$, that is,
\begin{equation}\label{eq:fixed-point-problem}
\cJ(\bu_f)=\bu_f.
\end{equation}
In what follows we focus on proving that $\cJ$ possesses a unique fixed-point. We remark in advance that the definition of $\cJ$ will make sense only in the ball $\bW_r$, cf. \eqref{eq:W_r-definition}.

The solvability of \eqref{eq:NS-Biot-formulation-3} is established using a suitable regularization, adding some extra terms to \eqref{eq:NS-Biot-formulation-3} multiplied by an arbitrary $\epsilon>0$ that provide coercivity for all unknowns, which yields a well-posed problem. Then, taking $\epsilon \to 0$, we recover \eqref{eq:NS-Biot-formulation-3}. More precisely, let
$R_{\bsi_p}: \bbX_p \to \bbX_p'$, $R_{p_p}: \W_p \to \W_p'$, $R_{\bu_p}: \bV_p \to \bV_p'$, $L_{\bu_s}: \bV_s \to \bV_s'$, and $L_{\bgamma_p}: \bbQ_p \to \bbQ_p'$ be defined by:
\begin{align*}
&\ds R_{\bsi_p}(\bsi_p)(\btau_p) = r_{\bsi_p}(\bsi_p, \btau_p) 
:= (\bsi_p, \btau_p)_{\Omega_p} + (\bdiv(\bsi_p), \bdiv(\btau_p))_{\Omega_p}, \\[1.5ex]
&\ds R_{p_p}(p_p)(w_p) = r_{p_p}(p_p,w_p) 
:= (p_p, w_p)_{\Omega_p},\quad
R_{\bu_p}(\bu_p)(\bv_p) = r_{\bu_p}(\bu_p, \bv_p) 
:= (\div(\bu_p), \div(\bv_p))_{\Omega_p}, \\[1.5ex]
&\ds L_{\bu_s}(\bu_s)(\bv_s) = l_{\bu_s}(\bu_s, \bv_s)
:= (\bu_s, \bv_s)_{\Omega_p}, \quad 
L_{\bgamma_p}(\bgamma_p)(\bchi_p) = l_{\bgamma_p}(\bgamma_p, \bchi_p) 
:= (\bgamma_p, \bchi_p)_{\Omega_p}\,. 
\end{align*}
The following operator properties follow immediately from the above definitions.
\begin{lem}\label{lem:R-operators1}
The operators $R_{\bsi_p}$,  $R_{p_p}$, $L_{\bu_s}$, and $L_{\bgamma_p}$ are bounded, continuous, and coercive.
In addition, $R_{\bu_p}$ is bounded, continuous, and monotone.
\end{lem}

On the other hand, recalling from \eqref{eq:trace-spaces} the trace spaces $\Lambda_p$ and $\bLambda_s$, we define $L_\lambda: \Lambda_p \to \Lambda_p'$ as
\begin{equation*}
L_\lambda(\lambda)(\xi) = l_{\lambda}(\lambda,\xi):=(\nabla \psi(\lambda), \nabla \psi(\xi))_{\Omega_p},
\end{equation*}
where $\psi(\lambda)\in \H^1(\Omega_p)$ is the weak solution of the auxiliary problem
\begin{gather*}
- \div(\nabla \psi(\lambda))  = 0 \qin \Omega_p, \\
\psi(\lambda)  = \lambda \qon \Gamma_{fp}, \quad
\nabla \psi(\lambda) \cdot \bn_p  = 0 \qon \Gamma_p^\rN, \quad \psi(\lambda) = 0 \qon \Gamma_p^\rD.
\end{gather*}
It is shown in \cite{aeny2019} using elliptic regularity and the trace inequality that there exist positive constants $c_1$ and $c_2$ such that
\begin{equation}\label{equiv-psi}
c_1\,\|\psi(\lambda)\|_{\H^1(\Omega_p)} 
\,\leq\, \|\lambda\|_{\Lambda_p} 
\,\leq\, c_2\,\|\psi(\lambda)\|_{\H^1(\Omega_p)}.
\end{equation}
Similarly, we define $R_{\btheta}: \bLambda_s \to \bLambda_s'$ as 
\begin{align*}
R_\btheta(\btheta)(\bphi) = r_{\btheta}(\btheta, \bphi)
:= (\nabla \bvarphi(\btheta), \nabla \bvarphi(\bphi))_{\Omega_p},
\end{align*}
where $\bvarphi(\btheta)\in \bH^1(\Omega_p)$ is the weak solution of
\begin{gather*}
- \bdiv\,(\nabla\bvarphi(\btheta))  = \0 \qin \Omega_p, \\
\bvarphi(\btheta)  = \btheta \qon \Gamma_{fp}, \quad
\nabla \bvarphi(\btheta) \cdot \bn_p  = 0 \qon \tilde\Gamma_p^\rN, \quad
\bvarphi(\btheta)  = \0 \qon \tilde\Gamma_{p}^\rD. 
\end{gather*}
Similarly to \eqref{equiv-psi}, there exist positive constants $\tilde c_1$ and $\tilde c_2$ such that
\begin{equation}\label{equiv-varphi}
\tilde c_1\,\|\bvarphi(\btheta)\|_{\bH^1(\Omega_p)} 
\,\leq\, \|\btheta\|_{\bLambda_s} 
\,\leq\, \tilde c_2\,\|\bvarphi(\btheta)\|_{\bH^1(\Omega_p)}.
\end{equation}

\begin{lem}\label{lem:R-operators2}
The operators $L_\lambda$ and $R_\btheta$ are bounded, continuous, coercive and monotone. 
\end{lem}
\begin{proof}
The assertion follows from the definition of the operators $L_\lambda, R_\btheta$ and the equivalence of norms statements \eqref{equiv-psi} and \eqref{equiv-varphi}.
In particular, there exist positive constants $c_\Gamma$, $C_\Gamma$, $\tilde c_\Gamma$, and $\tilde C_\Gamma$ such that
\begin{align*}
& L_{\lambda}(\lambda)(\xi) \,\leq\, C_\Gamma\,\|\lambda\|_{\Lambda_p} \|\xi\|_{\Lambda_p}\,,\quad L_{\lambda}(\lambda)(\lambda) \,\geq\, c_\Gamma\,\|\lambda\|^2_{\Lambda_p}, \\[1ex]
& R_{\btheta}(\btheta)(\bphi) \,\leq\, \tilde C_\Gamma\,\|\btheta\|_{\bLambda_s} \|\bphi\|_{\bLambda_s}\,,\quad R_{\btheta}(\btheta)(\bphi) \,\geq\, \tilde c_\Gamma\,\|\btheta\|^2_{\bLambda_s},
\end{align*}
for all $\lambda, \xi\in \Lambda_p$ and for all $\btheta, \bphi\in \bLambda_s$.
\end{proof}

According to the above, we define the operators 
$\cR: \bQ \to \bQ'$ and $\cL: \bS \to \bS'$ as
\begin{equation*}
\begin{array}{l}
\cR(\bp)(\bq) := R_{\bsi_p}(\bsi_p)(\btau_p) + R_{p_p}(p_p)(w_p) + R_{\bu_p}(\bu_p)(\bv_p) + R_{\btheta}(\btheta)(\bphi)\,, \\[2ex]
\cL(\br)(\bs) := L_\lambda(\lambda)(\xi) + L_{\bu_s}(\bu_s)(\bv_s) + L_{\bgamma_p}(\bgamma_p)(\bchi_p)\,.
\end{array}
\end{equation*}

\begin{thm}\label{thm:well-posedness-1}
Let $r\in (0, r_0)$, with $r_0$ given by \eqref{eq:r0-definition}, and let $\f_f\in\bL^2(\Omega_f)$ and $\f_p\in \bL^2(\Omega_p)$. Assume that the conditions in Lemma \ref{lem:monotone} are satisfied. Then for each $\bw_f\in \bW_r$ and for each $(\wh{\f}_p, \wh{q}_p)$ 
satisfying \eqref{eq:domain-D-hat}, there exists a unique solution of 
the resolvent system \eqref{eq:NS-Biot-formulation-3}. Moreover, there exists 
a constant $C_{\cJ}>0$, independent of $s_{0,\min}, \bw_f$, and the data $\f_f$ and $\f_p$ such that 
\begin{equation}\label{eq:C_T-bound}
\|\cJ(\bw_f)\|_{\bV_f} \,\leq\, \|(\bp, \br )\|_{\bQ \times \bS} 
\,\leq\, C_{\cJ}\,\left( \|\f_f\|_{\bL^2(\Omega_f)} + 
\|\f_p\|_{\bL^2(\Omega_p)} \right).
\end{equation}
\end{thm}
\begin{proof}
Given $\bw_f\in \bW_r$ with $r\in (0,r_0)$ (cf. \eqref{eq:r0-definition}), for each $0 < \epsilon \le 1$, consider a regularization of \eqref{eq:NS-Biot-formulation-3}: Find $\bp_\epsilon = (\bsi_{p,\epsilon}, p_{p,\epsilon}, \bu_{p,\epsilon}, \bT_{f, \epsilon}, \bu_{f,\epsilon},  \btheta_\epsilon) \in \bQ$ and $\br_\epsilon=(\lambda_\epsilon, \bu_{s,\epsilon}, \bgamma_{p,\epsilon}) \in \bS$,
such that
\begin{equation}\label{eq:NS-Biot-formulation-4}
\begin{array}{rll}
\ds (\epsilon \cR + \cE + \cA + \cK_{\bw_f})(\bp_\epsilon) + \cB'(\br_\epsilon) & = & \wh{\bF} \qin \bQ_2'\,, \\[2ex]
-\,\cB(\bp_\epsilon) + \epsilon \cL(\br_\epsilon) & = & \bG \qin \bS'\,.
\end{array}
\end{equation}
Let $\Psi : \bQ \times \bS \to \bQ' \times \bS'$ be the operator induced by \eqref{eq:NS-Biot-formulation-4}:
\begin{equation*}
\Psi 
\left(\begin{array}{c}
\bq\\ 
\bs
\end{array}\right)
= \left(\begin{array}{cc}
\epsilon \cR + \cE + \cA + \cK_{\bw_f} & \ \cB' \\ 
- \cB  & \  \epsilon \cL 
\end{array}\right) 
\left(\begin{array}{c}
\bq \\
\bs
\end{array}\right).
\end{equation*}
The continuity bounds in Lemmas \ref{lem:cont}, \ref{lem:R-operators1}, and \ref{lem:R-operators2} imply that $\Psi$ is bounded and continuous. In turn, we note that
\begin{equation*}
\Psi
\left(\begin{array}{c}
\bp \\ \br
\end{array}\right)
\left(\begin{array}{c}
\bq \\ \bs
\end{array}\right)
= \left(\epsilon \cR + \cE + \cA + \cK_{\bw_f}\right)(\bp)(\bq)
+ \cB'(\br)(\bq) - \cB(\bp)(\bs) + \epsilon \cL(\br)(\bs)\,.
\end{equation*}
The positivity bounds in Lemmas \ref{lem:monotone}, \ref{lem:R-operators1}, and \ref{lem:R-operators2}, imply
\begin{align}\label{eq:O-coercivity}
&\Psi 
\left(\begin{array}{c}
\bq\\ 
\bs
\end{array}\right)
\left(\begin{array}{c}
\bq\\ 
\bs
\end{array}\right)
= \left(\epsilon \cR + \cE + \cA + \cK_{\bw_f}\right)(\bq)(\bq) + \epsilon \cL(\bs)(\bs) \nonumber \\[1ex]
& =\epsilon r_{\bsi_p}(\btau_p, \btau_p)
+ \epsilon r_{p_p}(w_p, w_p)
+ \epsilon r_{\bu_p}(\bv_p, \bv_p)
+ \epsilon r_{\btheta}(\bphi, \bphi)
+ (s_0 w_p, w_p)
+ a_e(\btau_p, w_p; \btau_p, w_p)
 \nonumber \\[1ex]
& \quad
+ a_p(\bv_p, \bv_p) 
+ a_f(\bR_f, \bv_f; \bR_f, \bv_f)
+ \kappa_{\bw_f}(\bR_f, \bv_f; \bR_f, \bv_f)
+ a_{\BJS}(\bv_f, \bphi;\bv_f, \bphi)  \nonumber \\[1ex]
& \quad
+ \epsilon l_\lambda(\xi, \xi)
+ \epsilon l_{\bu_s}(\bv_s, \bv_s)
+ \epsilon l_{\bgamma_p}(\bchi_p, \bchi_p) \\[1ex]
&\geq C\,\left( \epsilon \|\btau_p\|^2_{\bbX_p}
+ \epsilon \|w_p\|^2_{\W_p}
+ \epsilon \|\div(\bv_p)\|^2_{\L^2(\Omega_p)}
+ \epsilon \|\bphi\|^2_{\bLambda_s}
+ s_0 \|w_p\|^2_{\W_p} + \|A^{1/2}(\btau_p + \alpha_p w_p \bI)\|^2_{\bbL^2(\Omega_p)} \right.
\nonumber \\[1ex]
&\quad
\left.
+ \|\bv_p\|^2_{\bL^2(\Omega_p)}
+ \|\bR_f\|^2_{\bbX_f} + \|\bv_f\|^2_{\bV_f} 
+ |\bv_f - \bphi|^2_{\BJS}
+ \epsilon \|\xi\|^2_{\Lambda_p}
+ \epsilon \|\bv_s\|^2_{\bV_s} 
+ \epsilon \|\bchi_p\|^2_{\bbQ_p}
\right),  \nonumber
\end{align}
which implies that $\Psi$ is coercive. 
Thus, the Lax--Milgram theorem implies the existence of a unique solution $(\bp_\epsilon,\br_\epsilon)\in \bQ\times \bS$ of \eqref{eq:NS-Biot-formulation-4}. 

On the other hand, using \eqref{eq:NS-Biot-formulation-4} and \eqref{eq:O-coercivity}, 
and the Cauchy--Schwarz inequality, we deduce that
\begin{align}\label{eq:epsilon-bound-1}
& \epsilon \|\bsi_{p,\epsilon}\|^2_{\bbX_p}
+ \epsilon \|p_{p,\epsilon}\|^2_{\W_p}
+ \epsilon \|\div(\bu_{p,\epsilon})\|^2_{\L^2(\Omega_p)}
+ \epsilon \|\btheta_{\epsilon}\|^2_{\bLambda_s}
+ s_0 \|p_{p,\epsilon}\|^2_{\W_p} 
+ \|A^{1/2}(\bsi_{p,\epsilon} + \alpha_p p_{p,\epsilon} \bI) \|^2_{\bbL^2(\Omega_p)}
\nonumber \\[1ex]
& \quad
+ \|\bu_{p,\epsilon}\|^2_{\bL^2(\Omega_p)}
+ \|\bT_{f,\epsilon}\|^2_{\bbX_f} + \|\bu_{f,\epsilon}\|^2_{\bV_f} 
+ |\bu_{f,\epsilon} - \btheta_{\epsilon}|^2_{\BJS}
+ \epsilon \|\lambda_{\epsilon}\|^2_{\Lambda_p}
+ \epsilon \|\bu_{s,\epsilon}\|^2_{\bV_s} 
+ \epsilon \|\bgamma_{p,\epsilon}\|^2_{\bbQ_p} \\[1ex]
& \leq C\,\left(
\|\f_f\|_{\bL^2(\Omega_f)}\|(\bT_{f,\epsilon},\bu_{f,\epsilon})\| 
+ \|\wh{\f}_p\|_{\bL^2(\Omega_p)} \|\bsi_{p,\epsilon}\|_{\L^2(\Omega_p)} 
+ \|\wh{q}_{p}\|_{\L^2(\Omega_p)} \|p_{p,\epsilon}\|_{\W_p} 
+ \|\f_p\|_{\bL^2(\Omega_p)} \|\bu_{s,\epsilon}\|_{\bL^2(\Omega_p)} \right)\,. \nonumber
\end{align} 
In addition, the inf-sup conditions \eqref{eq:continuous-inf-sup-1} 
and \eqref{eq:continuous-inf-sup-2} in Lemma \ref{lem:inf-sup}, in combination with 
the first equation of \eqref{eq:NS-Biot-formulation-4}, yield
\begin{equation}\label{eq:epsilon-bound-2}
\begin{array}{c}
\|\bu_{s,\epsilon}\|_{\bV_s}
+ \|\bgamma_{p,\epsilon}\|_{\bbQ_p} 
+ \|\btheta_{\epsilon}\|_{\bLambda_s} 
\,\leq\, C\,\left( \|A^{1/2}(\bsi_{p,\epsilon} + \alpha_p p_{p,\epsilon} \bI)\|_{\bbL^2(\Omega_p)}
+ \epsilon \|\bsi_{p,\epsilon}\|_{\bbX_p}
+ \|\wh{\f}_p\|_{\bL^2(\Omega_p)} \right),  \\[2ex]
\|p_{p,\epsilon}\|_{\W_p} 
+ \|\lambda_\epsilon\|_{\Lambda_p}
\,\leq\, C\,\left( \|\bu_{p,\epsilon}\|_{\bL^2(\Omega_p)}
+ \epsilon \|\div(\bu_{p, \epsilon})\|_{\bL^2(\Omega_p)} \right)\,.
\end{array}
\end{equation}
In turn, taking $\bv_s = \bdiv(\bsi_{p,\epsilon})$
and $w_p =\div( \bu_{p,\epsilon})$ in \eqref{eq:NS-Biot-formulation-4}, 
we deduce that
\begin{equation}\label{eq:epsilon-bound-final}
\begin{array}{c}
\|\bdiv(\bsi_{p, \epsilon})\|_{\bL^2(\Omega_p)} 
\,\leq\, \epsilon \|\bu_{s,\epsilon}\|_{\bV_s}
+ \|\f_p\|_{\bL^2(\Omega_p)}, \\[2ex]
\|\div(\bu_{p,\epsilon})\|_{\L^2(\Omega_p)} 
\,\leq\, C\,\left( \|A^{1/2}(\bsi_{p,\epsilon} + \alpha_p p_{p,\epsilon} \bI) \|_{\bbL^2(\Omega_p)}
+ (s_0 + \epsilon)\|p_{p,\epsilon}\|_{\W_p}
+ \|\wh{q}_{p}\|_{\L^2(\Omega_p)} \right). 
\end{array}
\end{equation}
Next, combining \eqref{eq:epsilon-bound-1} with \eqref{eq:epsilon-bound-2} and \eqref{eq:epsilon-bound-final}, using Young's inequality, 
some algebraic computations, and the estimate
\begin{equation}\label{eq:estimate-sigmap-A}
\|\bsi_{p,\epsilon}\|_{\bbL^2(\Omega_p)} \,\leq\, C\,\left( \|A^{1/2}(\bsi_{p,\epsilon} + \alpha_p\,p_{p,\epsilon}\,\bI)\|_{\bbL^2(\Omega_p)} + \|p_{p,\epsilon}\|_{\W_p} \right),
\end{equation}
which follows from \eqref{eq:A-bounds} and the triangle inequality, we deduce that
\begin{align}\label{eq:epsilon-bound-3}
& \|\bsi_{p,\epsilon}\|^2_{\bbX_p}
+ \|p_{p,\epsilon}\|^2_{\W_p}
+ \|\bu_{p,\epsilon}\|^2_{\bV_p}
+ \|\bT_{f,\epsilon}\|^2_{\bbX_f} + \|\bu_{f,\epsilon}\|^2_{\bV_f} 
+ |\bu_{f,\epsilon} - \btheta_{\epsilon}|^2_{\BJS}
+ \|\btheta_{\epsilon}\|^2_{\bLambda_s} \nonumber \\[1ex]
& \quad +\, \|\lambda_\epsilon\|_{\Lambda_p}
+ \|\bu_{s,\epsilon}\|^2_{\bV_s}
+ \|\bgamma_{p,\epsilon}\|^2_{\bbQ_p}
\,\leq\, C\,\left(
\|\f_f\|^2_{\bL^2(\Omega_f)}  
+ \|\f_p\|^2_{\bL^2(\Omega_p)}
+ \|\wh{\f}_p\|^2_{\bL^2(\Omega_p)} 
+ \|\wh{q}_{p}\|^2_{\L^2(\Omega_p)} 
\right)\,, 
\end{align}
with $C>0$ independent of $s_{0,\min}$ and $\epsilon$. 
Thus, from \eqref{eq:epsilon-bound-3} and the assumption \eqref{eq:domain-D-hat}, 
we deduce that the solution of \eqref{eq:NS-Biot-formulation-4} is bounded independently of $\epsilon$.
More precisely, there exists $\wt{C}_\cJ>0$ 
independent of $s_{0,\min}, \epsilon$ and $\bw_f$, such that
\begin{equation}\label{eq:C_T-bound-2}
\|(\bp_\epsilon, \br_\epsilon )\|_{\bQ \times \bS} 
\,\leq\, \wt{C}_{\cJ}\,\left( \|\f_f\|_{\bL^2(\Omega_f)} + \|\f_p\|_{\bL^2(\Omega_p)} \right).
\end{equation}

Now, we take $\epsilon \to 0$ in \eqref{eq:NS-Biot-formulation-4}. Similarly to \cite[Theorem 3.2]{s2010}, and since $\bQ$ and $\bS$ are reflexive Banach spaces, 
we can extract weakly convergent subsequences $\{\bp_{\epsilon,n}\}_{n=1}^\infty$ 
and $\{\br_{\epsilon,n}\}_{n=1}^\infty$ such that $\bp_{\epsilon,n} \rightharpoonup \bp$ 
in $\bQ$, $\br_{\epsilon,n} \rightharpoonup \br$ in $\bS$, which combined with the fact
that $\cE$, $\cA$, $\cK_{\bw_f}$, $\cB$, $\wh{\bF}$ and $\bG$ are continuous implies that 
$(\bp, \br)$ is a solution to \eqref{eq:NS-Biot-formulation-3}. 
Moreover, proceeding analogously to \eqref{eq:C_T-bound-2} but now considering $\epsilon=0$ in \eqref{eq:epsilon-bound-1}--\eqref{eq:epsilon-bound-final}, 
we are able to derive \eqref{eq:C_T-bound}, with $C_\cJ>0$ 
independent of $s_{0,\min}$ and $\bw_f$.

Finally, we prove that the solution of \eqref{eq:NS-Biot-formulation-3} is unique. 
Since \eqref{eq:NS-Biot-formulation-3} is linear, it is sufficient to prove that the problem with zero data has only the zero solution.
Taking $(\wh{\bF}, \bG) = (\0, \0)$ in \eqref{eq:NS-Biot-formulation-3}, testing it with the solution $(\bp, \br)$, and using Lemma \ref{lem:monotone}, yields
\begin{equation*}
s_0\,\|p_{p}\|^2_{\W_p} 
+ \|A^{1/2}(\bsi_{p} + \alpha_p p_{p} \bI) \|^2_{\bbL^2(\Omega_p)}
+ C_{\cA\cK}\,\left( \|\bu_p\|^2_{\bL^2(\Omega_p)}
+ \|(\bT_f,\bu_f)\|^2
+ |\bu_f - \btheta|^2_{\BJS} \right) \,\leq\, \0,
\end{equation*}
so it follows that $A^{1/2}(\bsi_p + \alpha_p p_{p} \bI) = \0$, $\bu_p = \0$, $\bT_f = \0$, and $\bu_f = \0$. Next, combining the inf-sup conditions \eqref{eq:continuous-inf-sup-1} 
and \eqref{eq:continuous-inf-sup-2} in Lemma \ref{lem:inf-sup} with the first equation of \eqref{eq:NS-Biot-formulation-3}, we deduce that $p_p=0$, $\bsi_p = \0$, $\btheta=\0$, $\lambda=0$, $\bu_s=\0$, and $\bgamma_p=\0$, concluding the proof.
\end{proof}

\subsubsection{Step 2: The domain \texorpdfstring{$\cD$}{Lg} is nonempty}

In this section we proceed  analogously to \cite{cgos2017} and employ 
the Banach fixed-point theorem to show that $\cD$ (cf. \eqref{eq:domain-D}) is nonempty.
\begin{lem}
Let $r\in (0, r_0)$, with $r_0$ given by \eqref{eq:r0-definition} and assume that the conditions in Lemma \ref{lem:monotone} are satisfied. Then, for all $\bw_f$, $\wt{\bw}_f \in \bW_r$ there holds
\begin{equation}\label{eq:contraction-mapping}
\|\cJ(\bw_f) - \cJ(\wt{\bw}_f)\|_{\bV_f} 
\leq \frac{C_{\cJ}}{r_0}\,\left( \|\f_f\|_{\bL^2(\Omega_f)} 
+ \|\f_p\|_{\bL^2(\Omega_p)} \right) \|\bw_f - \wt{\bw}_f\|_{\bV_f}\,,
\end{equation}
where $C_{\cJ}$ is the constant from \eqref{eq:C_T-bound}.
\end{lem}
\begin{proof}
Given $\bw_f$, $\wt{\bw}_f \in \bW_r$, we let $\bu_f:=\cJ(\bw_f)$ and $\wt{\bu}_f:=\cJ(\wt{\bw}_f)$. According to the definition of $\cJ$, (cf. \eqref{eq:def-T}--\eqref{eq:NS-Biot-formulation-3}) it follows that  
\begin{equation*}
\begin{array}{rll}
(\cE + \cA + \cK_{\bw_f})(\bp - \wt{\bp})(\bq) + \cB'(\br - \wt{\br})(\bq) & =\,  -\cK_{\bw_f - \wt{\bw}_f}(\wt{\bp})(\bq) & \forall\,\bq\in \bQ\,, \\[2ex]
- \cB(\bp - \wt{\bp})(\bs) & =\, 0 & \forall\,\bs\in \bS\,.
\end{array}
\end{equation*}
Taking $\bq = \bp - \wt{\bp}$ and $\bs = \br - \wt{\br}$ in the foregoing equations, we obtain
\begin{equation*}
(\cE + \cA + \cK_{\bw_f})(\bp - \wt{\bp})(\bp - \wt{\bp})=-\cK_{\bw_f-\wt{\bw}_f}(\wt{\bp})(\bp - \wt{\bp}),
\end{equation*}
which together with the continuity of $\cK_{\bw_f}$ with $\bw_f \in \bW_r$ 
(cf. \eqref{eq:continuity-of-Kwf}) and the positivity bounds of $\cE$ and $\cA+\cK_{\bw_f}$ 
(cf. \eqref{eq:positive-bound-E}, \eqref{eq:positive-bound-A+Kwf}, \eqref{eq:monotone-01})
in Lemma \ref{lem:monotone}, implies that
\begin{equation*}
\frac{\alpha_f}{2}\,\|(\bT_f - \wt{\bT}_f, \bu_f - \wt{\bu}_f)\|^2 
\,\leq\, C_{\cK}\,\|\wt{\bu}_f\|_{\bV_f} \|\bw_f - \wt{\bw}_f\|_{\bV_f} \|(\bT_f - \wt{\bT}_f, \bu_f - \wt{\bu}_f)\|\,.
\end{equation*}
Therefore, using the bound of $\|\wt{\bu}_f\|_{\bV_f}$, cf. \eqref{eq:C_T-bound}, we get 
\begin{equation*}
\|\bu_f - \wt{\bu}_f\|_{\bV_f} 
\,\leq\, C_{\cJ}\frac{2\,C_{\cK}}{\alpha_f}\,\left( \|\f_f\|_{\bL^2(\Omega_f)} 
+ \|\f_p\|_{\bL^2(\Omega_p)} \right)\,\|\bw_f - \wt{\bw}_f\|_{\bV_f} \,,
\end{equation*}
which, combined with the definition of $r_0$, cf. \eqref{eq:r0-definition}, 
yields \eqref{eq:contraction-mapping}, concluding the proof.
\end{proof}

We are now in position to establish the main result of this section.

\begin{thm}\label{thm:domain-is-nonempty}
Let $r\in (0, r_0)$, with $r_0$ given by \eqref{eq:r0-definition}, 
and assume that the conditions 
in Lemma \ref{lem:monotone} hold. Furthermore, assume that the data satisfy
\begin{equation}\label{eq:Ct-less-than-r}
C_{\cJ}\left( \|\f_f\|_{\bL^2(\Omega_f)} 
+ \|\f_p\|_{\bL^2(\Omega_p)} \right) \leq r \,.
\end{equation}
Then, for each $(\wh{\f}_p, \wh{q}_p) \in E_b'$ 
satisfying \eqref{eq:domain-D-hat}, the resolvent problem \eqref{eq:resolvent-weak-formulation} has a unique solution 
$(\bp, \br)\in \bQ \times \bS$ with $\bu_f \in \bW_r$, and there holds
\begin{equation}\label{eq:solution-bound-by-data}
\|(\bp, \br)\|_{\bQ \times \bS} 
\,\leq\, C_{\cJ}\,\left( \|\f_f\|_{\bL^2(\Omega_f)} + \|\f_p\|_{\bL^2(\Omega_p)} \right).
\end{equation}
\end{thm}
\begin{proof}
Let us fix an arbitrary $(\wh{\f}_p, \wh{q}_p) \in E_b'$ 
satisfying \eqref{eq:domain-D-hat}. We note that \eqref{eq:C_T-bound} and \eqref{eq:Ct-less-than-r} imply that 
$\cJ: \bW_r \rightarrow \bW_r$. Combining bound \eqref{eq:contraction-mapping} and 
assumption \eqref{eq:Ct-less-than-r}, we have
\begin{equation*}
\|\cJ(\bw_f) - \cJ(\wt{\bw}_f)\|_{\bV_f} 
\,\leq\, \frac{r}{r_0}\,\|\bw_f - \wt{\bw}_f\|_{\bV_f}\,,
\end{equation*}
which implies that $\cJ$ is a contraction mapping. Therefore by the classical 
Banach fixed-point theorem, we conclude that $\cJ$ has a unique 
fixed-point $\bu_f \in \bW_r$, or equivalently, \eqref{eq:resolvent-weak-formulation} has a unique solution, hence the domain $\cD$ (cf. \eqref{eq:domain-D}) is nonempty. 
In addition, \eqref{eq:solution-bound-by-data} follows directly from \eqref{eq:C_T-bound}.
\end{proof}

\subsubsection{Step 3: Solvability of the parabolic problem}

In this section we establish the existence of a solution to \eqref{eq:parabolic-problem}
as a direct application of Theorem \ref{thm:solvability-parabolic-problem}. 
We begin by showing that $\cM$ defined by \eqref{eq:operator-M} is a monotone operator.
\begin{lem}\label{lem:M-is-monotone}
Let $r \in (0, r_0)$ with $r_0$ defined by \eqref{eq:r0-definition}. 
Assume that the parameters $\kappa_1, \kappa_2$ satisfy the conditions 
in Lemma \ref{lem:monotone}. Furthermore, assume that the data satisfy \eqref{eq:Ct-less-than-r}.
Then, the operator $\cM$ defined by \eqref{eq:operator-M} is monotone.
\end{lem}
\begin{proof}
For each $(\bsi_p^i, p_p^i)\in \cD$, $i\in \{1,2\}$, let $(\wh{\f}_p^i, \wh q_p^i) \in E_b'$ be such that $(\wh{\f}_p^i, \wh q_p^i)-\cN(\bsi_p^i, p_p^i)\in \cM(\bsi_p^i, p_p^i)$, i.e., 
\eqref{eq:domain-D} holds. Then we have
\begin{equation*}
\begin{array}{l}
\ds \big( (\wh{\f}_p^i, \wh{q}_p^i)-\cN(\bsi_p^i, p_p^i)\big)(\btau_p, w_p) \\[2ex]
\ds\quad \ =\, (\wh{\f}_p^i , \btau_p )_{\Omega_p}
+ (\wh{q}_p^i, w_p)_{\Omega_p}
-(A(\bsi_p^i + \alpha_p \, p_p^i \, \bI), \btau_p + \alpha_p \, w_p \, \bI)_{\Omega_p}
- (s_0 \, p_p^i, w_p)_{\Omega_p} \\[2ex]
\ds\quad \ =\, -b_p(w_p, \bu_p^i) - b_{\bn_p}(\btau_p, \btheta^i) + b_s(\bu_s^i, \btau_p) + b_{\sk}(\bgamma_p^i, \btau_p) \quad \forall\,(\btau_p,w_p)\in \bbX_p\times \W_p\,.
\end{array}
\end{equation*}
Then, using the association $\bv^i = (\bsi_p^i, p_p^i)$ and $\f^i = (\wh{\f}_p^i, \wh{q}_p^i )-\cN(\bsi_p^i, p_p^i)$, $i\in \{1,2\}$, we deduce that
\begin{equation}\label{eq:M-is-monotone-0}
\begin{array}{l}
\ds (\f^1- \f^2)(\bv^1 - \bv^2)
= -b_p(p_p^1 - p_p^2, \bu_p^1 - \bu_p^2) - b_{\bn_p}(\bsi_p^1 - \bsi_p^2, \btheta^1 - \btheta^2) + b_s(\bu_s^1 - \bu_s^2, \bsi_p^1 - \bsi_p^2)\\[2ex]
\ds \qquad \qquad \qquad \qquad \quad \ +\, b_{\sk}(\bgamma_p^1 - \bgamma_p^2, \bsi_p^1 - \bsi_p^2)\,.
\end{array}
\end{equation}

In turn, from \eqref{eq:domain-D}, it follows that $((\bu_p^i, \bT_f^i, \bu_f^i, \btheta^i),\lambda^i,\bu_s^i, \bgamma_p^i)$ satisfy
\begin{equation}\label{eq:M-monotone-1}
\begin{array}{l}
\ds s_0\,(p_p^i,w_p)_{\Omega_p} + a_e( \bsi_p^i,  p_p^i; \btau_p, w_p) + a_p(\bu_p^i,\bv_p) + a_f(\bT_f^i,\bu_f^i;\bR_f,\bv_f) 
+\kappa_{\bu_f^i}(\bT_f^i,\bu_f^i;\bR_f,\bv_f)
\\ [2ex]
\ds\quad +\,\,
a_{\BJS}(\bu_f^i,\btheta^i;\bv_f,\bphi) +  b_p(p_p^i,\bv_p) - b_p(w_p,\bu_p^i) +
b_{\bn_p}(\bsi_p^i,\bphi) - b_{\bn_p}(\btau_p,\btheta^i) + b_s(\bu_s^i,\btau_p) \\ [2ex]
\ds\quad
+ \, b_\sk(\bgamma_p^i,\btau_p)
+  b_\Gamma(\bv_p,\bv_f,\bphi;\lambda^i) \,=\, (\wh{\f}_p^i , \btau_p )_{\Omega_p}
+ (\wh{q}_p^i, w_p)_{\Omega_p}
+ (\f_f,\bv_f - \kappa_1\,\bdiv(\bR_f))_{\Omega_f}\,, \\[2ex]
\ds -\,b_s(\bv_s,\bsi_p^i) - b_\sk(\bchi_p,\bsi_p^i) -  b_\Gamma(\bu_p^i,\bu_f^i,\btheta^i;\xi) 
\, = \,(\f_p,\bv_s)_{\Omega_p} \quad \forall\,(\bq,\bs)\in \bQ\times \bS\,.
\end{array}
\end{equation}
Testing \eqref{eq:M-monotone-1} with 
$\bq = (\0,0,\bu_p^1 -\bu_p^2, \bT_f^1 - \bT_f^2, \bu_f^1- \bu_f^2, \btheta^1-\btheta^2)$ 
and $\bs=(\lambda^1 - \lambda^2, \bu_s^1 - \bu_s^2, \bgamma_p^1 - \bgamma_p^2)$, for $i\in \{1, 2\}$, we find that
\begin{equation*}
\begin{array}{l}
\ds -\,b_p(p_p^1 - p_p^2,\bu_p^1 -\bu_p^2)
- b_{\bn_p}(\bsi_p^1 - \bsi_p^2,\btheta^1-\btheta^2)
+ b_s(\bu_s^1 - \bu_s^2,\bsi_p^1 -\bsi_p^2) 
+ b_\sk(\bgamma_p^1 - \bgamma_p^2,\bsi_p^1 -\bsi_p^2)  \\[2ex]
\ds =\, a_p(\bu_p^1 - \bu^2_p,\bu_p^1 -\bu_p^2) + a_f(\bT_f^1 - \bT^2_f,\bu_f^1 - \bu^2_f;\bT_f^1 - \bT_f^2,\bu_f^1- \bu_f^2) 
+ \kappa_{\bu_f^1}(\bT_f^1,\bu_f^1;\bT_f^1 - \bT_f^2,\bu_f^1 - \bu_f^2) 
\\ [2ex]
\ds\quad -\, \kappa_{\bu_f^2}(\bT_f^2,\bu_f^2;\bT_f^1 - \bT_f^2,\bu_f^1 - \bu_f^2) 
+ a_{\BJS}(\bu_f^1 - \bu_f^2,\btheta^1 - \btheta^2;\bu_f^1 - \bu_f^2,\btheta^1 -\btheta^2)\,,
\end{array}
\end{equation*}
which, replaced back into \eqref{eq:M-is-monotone-0} together with 
the stability properties developed for $a_p, a_f, a_{\BJS}$ in \eqref{eq:monotone-01}--\eqref{eq:monotone-02} (cf. Lemma \ref{lem:monotone})
and the continuity of $\kappa_{\bu_f}$, cf. \eqref{eq:continuity-of-Kwf}, yields 
\begin{equation*}
\begin{array}{l}
\ds (\f^1- \f^2)(\bv^1 - \bv^2) 
= a_p(\bu_p^1 - \bu_p^2,\bu_p^1 -\bu_p^2) + a_f(\bT_f^1-\bT_f^2,\bu_f^1-\bu_f^2;\bT_f^1 - \bT_f^2,\bu_f^1- \bu_f^2) \\[2ex]
\ds \quad
+ \, \kappa_{\bu_f^1-\bu_f^2}(\bT_f^1,\bu_f^1;\bT_f^1 - \bT_f^2,\bu_f^1- \bu_f^2)
+\kappa_{\bu_f^2}(\bT_f^1 - \bT_f^2,\bu_f^1- \bu_f^2;\bT_f^1 - \bT_f^2,\bu_f^1- \bu_f^2) \\[2ex]
\ds \quad
+ \, a_{\BJS}(\bu_f^1-\bu_f^2,\btheta^1-\btheta^2;\bu_f^1-\bu_f^2,\btheta^1-\btheta^2) \\[2ex]
\ds \
\geq \left( \alpha_f - C_{\cK}\big(\|\bu_f^1\|_{\bV_f} + \|\bu_f^2\|_{\bV_f}\big)\right)
\|(\bT_f^1-\bT_f^2, \bu_f^1-\bu_f^2)\|^2\,.
\end{array}
\end{equation*}
Finally, recalling from the definition of the domain $\cD$ (cf. \eqref{eq:domain-D}) that both $\|\bu_f^1\|_{\bV_f}$ and $\|\bu^2_f\|_{\bV_f}$ are bounded by $r$, we obtain
\begin{equation*}
(\f^1- \f^2)(\bv^1 - \bv^2) 
\,\geq\, 2\,\,C_{\cK}\,(r_0 - r)\,\|(\bT_f^1 - \bT_f^2,\bu_f^1 - \bu_f^2)\|^2 
\,\geq\, 0\,,
\end{equation*}
which implies the monotonicity of $\cM$.
\end{proof}

Now, we are in position to establish the well-posedness of \eqref{eq:parabolic-problem}.
\begin{lem}\label{lem:well-posedness-of-parabolic-problem}
Under the conditions of Lemma~\ref{lem:M-is-monotone},
for each $(h_{\bsi_p}, h_{p_p}) \in \W^{1,1}(0,T; \bbL^2(\Omega_p)) \times\linebreak \W^{1,1}(0,T;\L^2(\Omega_p))$ and each $(\bsi_{p,0}, p_{p,0}) \in \cD$, there exists a solution $(\bsi_p, p_p):[0,T] \to \cD$ to \eqref{eq:parabolic-problem} with
\begin{equation*}
  (\bsi_p, p_p) \in \W^{1,\infty}(0,T; \bbL^2(\Omega_p)) \times \W^{1,\infty}(0,T; \W_p) \qan
  (\bsi_p(0), p_p(0))=(\bsi_{p,0}, p_{p,0}).
\end{equation*}
\end{lem}
\begin{proof}
We recall that \eqref{eq:parabolic-problem} fits in the framework of 
Theorem \ref{thm:solvability-parabolic-problem} with $E = \bbX_p \times \W_p, E_b'=\bbL^2(\Omega_p)\times \L^2(\Omega_p)$, and $\cN, \cM$ defined in \eqref{defn-N} and \eqref{eq:operator-M}, respectively.
Note that $\cN$ is linear, symmetric and monotone.
In addition, from Lemma \ref{lem:M-is-monotone}, we obtain that $\cM$ is monotone. On the other hand, for the range condition $Rg(\cN+\cM)=E_b'$ in 
Theorem \ref{thm:solvability-parabolic-problem}, we note that in our case $Rg(\cN+\cM)$ is a subset of $E_b'$, see its definition \eqref{eq:operator-M}.
Therefore it is enough to establish the range condition $Rg(\cN+\cM)=\tilde E_b'$, where $\tilde E_b':= \{(\wh{\f}_p, \wh{q}_p) \in E_b': \eqref{eq:domain-D-hat} \hbox{ holds}\}$. This follows from Theorem~\ref{thm:domain-is-nonempty}, where we established that for each $(\wh{\f}_p, \wh{q}_p) \in \tilde E_b'$ there exists $(\wh{\bsi}_p, \wh{p}_p) \in \cD$ a solution to \eqref{eq:resolvent-weak-formulation}.
Therefore, applying Theorem \ref{thm:solvability-parabolic-problem}
in our context, we conclude that there exists a solution 
$(\bsi_p,p_p):[0,T] \to \cD$ to \eqref{eq:parabolic-problem}, 
with $(\bsi_p, p_p) \in \W^{1,\infty}(0,T; \bbL^2(\Omega_p))\times \W^{1,\infty}(0,T;\W_p)$ and $(\bsi_p(0), p_p(0))=(\bsi_{p,0}, p_{p,0})$.
\end{proof}

\subsubsection{Construction of compatible initial data}

We next construct initial data $(\bsi_{p,0}, p_{p,0}) \in \cD$, which is needed in Lemma~\ref{lem:well-posedness-of-parabolic-problem}.

\begin{lem}\label{lem:initial-condition}
Let $(\f_f, \f_p)\in \bL^2(\Omega_f)\times \bL^2(\Omega_p)$.
Assume that the conditions of Lemma~\ref{lem:monotone} are satisfied.
Assume that the initial condition $p_{p,0} \in \H_p$, where
\begin{align}
\H_p  \,:=\, \Big\{ w_p\in \H^1(\Omega_p) :\quad \
 \bK\,\nabla\,w_p\in \bH^1(\Omega_p), \  \bK\,\nabla\,w_p\cdot\bn_p = 0 \,\mbox{ on }\,
\Gamma^{\rN}_p,  \ w_p = 0 \,\mbox{ on }\, \Gamma^{\rD}_p \Big\}. \label{eq:H-definition}
\end{align}
Furthermore, assume that there exists $C_0 > 0$ such that
\begin{equation}\label{eq:extra-pp0-assumption}
\|p_{p,0}\|_{\H^1(\Omega_p)} + \|\bK\nabla p_{p,0}\|_{\bH^1(\Omega_p)} \,\leq\, C_0\,\left( \|\f_f\|_{\bL^2(\Omega_f)} + \|\f_p\|_{\bL^2(\Omega_p)} \right)\,,
\end{equation}
\begin{equation}\label{eq:initial-small-data}
\qan C_{\cJ_0}\,\left(\|\f_f\|_{\bL^2(\Omega_f)} + \|p_{p,0}\|_{\H^1(\Omega_p)}
+ \|\bK\nabla p_{p,0}\|_{\bH^1(\Omega_p)} \right) \,\leq\, r \,, 
\end{equation}
for $r\in (0,r_0)$, where $r_0$ is defined in \eqref{eq:r0-definition} 
and $C_{\cJ_0}>0$ is defined in \eqref{eq:stability-initial-Tf0-uf0} below.
Then, there exists $\bsi_{p,0} \in \bbX_p$ such that
  $(\bsi_{p,0}, p_{p,0}) \in \cD$. In particular, there exist
$\bp_0:=(\bsi_{p,0}, p_{p,0}, \bu_{p,0},\bT_{f,0}, \bu_{f,0}, \btheta_0)\in \bQ$ 
and $\br_0:=(\lambda_0, \bu_{s,0}, \bgamma_{p,0})\in \bS$ with $\bu_{f,0}\in \bW_r$ such that
\begin{equation}\label{eq:initial-data-system}
\begin{array}{rcl}
\ds (\cE + \cA + \cK_{\bu_{f,0}})(\bp_0) + \cB'(\br_0) & = & \wh{\bF}_0 \qin \bQ_2'\,, \\ [2ex]
\ds - \cB(\bp_0) & = & \bG_0 \qin \bS'\,,
\end{array}
\end{equation}
where $\bG_0(\bs) := (\f_p,\bv_s)_{\Omega_p} \ \forall \ \bs \in \bS$ 
and $\wh{\bF}_0(\bq) := (\wh{\f}_{p,0}, \btau_p)_{\Omega_p} 
+ (\wh{q}_{p,0}, w_p)_{\Omega_p}
+ (\f_f, \bv_f - \kappa_1 \bdiv(\bR_f))_{\Omega_f} \ \forall\,\bq\in \bQ$,
with some $(\wh{\f}_{p,0},\wh{q}_{p,0})\in E_b'$ satisfying
\begin{equation}\label{eq:initial-data-bound}
\|\wh{\f}_{p,0}\|_{\bbL^2(\Omega_p)} + \|\wh{q}_{p,0}\|_{\L^2(\Omega_p)} 
\,\leq\, \wh{C}_{ep} \left(\|\f_f\|_{\bL^2(\Omega_f)} + \|\f_p\|_{\bL^2(\Omega_p)} \right),
\end{equation}
where $\wh{C}_{ep}$ is specified in \eqref{eq:initial-hat-data-bound-again} below.
\end{lem}
\begin{proof}
We proceed as in \cite[Lemma~4.15]{aeny2019}.
We solve a sequence of well-defined sub-problems, using the previously obtained solutions as data to guarantee that we obtain a solution of the coupled problem.
We take the following steps.	
	
1. Define $\bu_{p,0} := -\dfrac{1}{\mu}\,\bK\nabla p_{p,0}$, with $p_{p,0}\in \H_p$, cf. \eqref{eq:H-definition}. 
It follows that $\bu_{p,0}\in \bH(\div;\Omega_p)$ and
\begin{equation}\label{eq:sol0-up0-pp0}
\mu\,\bK^{-1}\bu_{p,0} = -\nabla p_{p,0},\quad 
\div(\bu_{p,0}) = -\frac{1}{\mu}\,\div(\bK\nabla p_{p,0}) \qin \Omega_p,\quad
\bu_{p,0}\cdot\bn_p = 0 \qon \Gamma^{\rN}_p.
\end{equation}
Next, defining $\lambda_0 := p_{p,0}|_{\Gamma_{fp}}\in \Lambda_p$, \eqref{eq:sol0-up0-pp0} yields
\begin{equation}\label{eq:sol0-up0-pp0-2}
a_p(\bu_{p,0},\bv_p) + b_p(\bv_p,p_{p,0}) + b_{\Gamma}(\bv_p, \0, \0;\lambda_0) = 0 \quad \forall\,\bv_p\in \bV_p \,. 
\end{equation}
	
2. Define $(\bT_{f,0},\bu_{f,0})\in \bbX_{f}\times \bV_f$ associated to the problem
\begin{equation}\label{eq:sol0-Tf0-uf0}
\begin{array}{l}
\ds a_f(\bT_{f,0},\bu_{f,0}; \bR_f, \bv_f) + \kappa_{\bu_{f,0}}(\bT_{f,0},\bu_{f,0}; \bR_f, \bv_f)  \\[2ex]
\ds = - \mu\,\alpha_{BJS} \sum^{n-1}_{j=1} \pil \sqrt{\bK^{-1}_j}\bu_{p,0}\cdot\bt_{f,j},\bv_f\cdot\bt_{f,j} \pir_{\Gamma_{fp}}  - \langle \bv_f\cdot\bn_f, \lambda_0 \rangle_{\Gamma_{fp}} 
+ (\f_f,\bv_f - \kappa_1\,\bdiv(\bR_f))_{\Omega_f} \,,
\end{array}
\end{equation}
for all $(\bR_f,\bv_f)\in \bbX_f\times \bV_f$.
Notice that \eqref{eq:sol0-Tf0-uf0} is well-posed, since it corresponds to 
the weak solution of the augmented mixed formulation for the Navier--Stokes problem 
with mixed boundary conditions. 
Notice that $\bu_{p,0}$ and $\lambda_0$ are data for this problem.
The well-posedness of \eqref{eq:sol0-Tf0-uf0},
follows from a fixed point approach as in \eqref{eq:def-T} combined with 
the {\it a priori} estimate
\begin{equation}\label{eq:stability-initial-Tf0-uf0}
\|(\bT_{f,0},\bu_{f,0})\| 
\,\leq\, C_{\cJ_0}\,\left( \|\f_f\|_{\bL^2(\Omega_f)} 
+ \|p_{p,0}\|_{\H^1(\Omega_p)}
+ \|\bK\nabla p_{p,0}\|_{\bH^1(\Omega_p)} \right)
\end{equation}
and the data assumption \eqref{eq:initial-small-data}.
We refer to \cite{cot2016} for a similar approach applied to the stationary Navier--Stokes problem. We note that \eqref{eq:initial-small-data} and \eqref{eq:stability-initial-Tf0-uf0} imply that $\bu_{f,0}\in \bW_r$.

3. Define $\ds (\bsi_{p,0}, \bbeta_{p,0}, \brho_{p,0}, \bpsi_0) \in \bbX_p
\times \bV_s \times \bbQ_p \times \bLambda_s $ such that
\begin{equation}\label{eq:sol0-sigma_p0}
\begin{array}{ll}
\ds (A \bsi_{p,0}, \btau_p)_{\Omega_p}
+ b_s(\bbeta_{p,0}, \btau_p)
+ b_{\sk}(\brho_{p,0},\btau_p)
- b_{\bn_p}(\bpsi_0, \btau_p)
= - (A \alpha p_{p,0} \bI,\btau_p)_{\Omega_p}
& \ \forall \, \btau_p \in \bbX_p\,, \\[2ex]
\ds  - b_s ( \bsi_{p,0}, \bv_s ) = (\f_p,\bv_s)_{\Omega_p}  & \ \forall \, \bv_s \in \bV_s, \\[2ex]
\ds  - b_{\sk} ( \bsi_{p,0}, \bchi_p ) = 0 & \ \forall \, \bchi_p \in \bbQ_p \,,\\[2ex]
\ds b_{\bn_p}(\bsi_{p,0}, \bphi) = -\mu\,\alpha_{BJS} \sum^{n-1}_{j=1} \pil \sqrt{\bK^{-1}_j}\bu_{p,0}\cdot\bt_{f,j},\bphi\cdot\bt_{f,j} \pir_{\Gamma_{fp}} 
- \langle \bphi \cdot \bn_p, \lambda_0 \rangle_{\Gamma_{fp}}
& \ \forall \, \bphi \in \bLambda_s, 
\end{array}
\end{equation}
This is a well-posed problem corresponding to the weak solution of
the mixed elasticity system with mixed boundary conditions on $\Gamma_{fp}$.
Note that $p_{p,0}$, $\bu_{p,0}$ and $\lambda_0$ are data for this problem. The following stability bound holds:
\begin{equation}\label{bound-sigmap0}
\|\bsi_{p,0}\|_{\bbX_p} + \|\bbeta_{p,0}\|_{\bV_s} + \|\brho_{p,0}\|_{\bbQ_p} + \|\bpsi_0\|_{\bLambda_s}
\le C \left(\|p_{p,0}\|_{\H^1(\Omega_p)} + \|\bK\nabla p_{p,0}\|_{\bH^1(\Omega_p)} + \|\f_p\|_{\bL^2(\Omega_p)} \right).
\end{equation}
We note that $\bbeta_{p,0}$, $\brho_{p,0}$, and
$\bpsi_0$ are auxiliary variables that are not part of the constructed initial data.
However, they can be used to recover the variables $\bbeta_{p}$, $\brho_{p}$, and
$\bpsi$ that satisfy the non-differentiated equation \eqref{non-diff-eq}. 

4. Define $\btheta_0 \in \bLambda_s$ as
\begin{equation}\label{eq:sol0-theta0}
\btheta_0 =  \bu_{f,0} - \bu_{p,0} 
\qon \Gamma_{fp}, 
\end{equation}
where $\bu_{f,0}$ and $\bu_{p,0}$ are data obtained in the previous steps. It holds that
\begin{equation}\label{bound-theta0}
\|\btheta_0\|_{\bLambda_s} \le C \big(\|\bu_{f,0}\|_{\bH^1(\Omega_f)} + \|\bu_{p,0}\|_{\bH^1(\Omega_p)}\big) 
\le C \big(\|\f_f\|_{\bL^2(\Omega_f)}
+ \|p_{p,0}\|_{\H^1(\Omega_p)}
+ \|\bK\nabla p_{p,0}\|_{\bH^1(\Omega_p)} \big).
\end{equation}
Note that \eqref{eq:sol0-theta0} implies that the BJS terms in \eqref{eq:sol0-Tf0-uf0} and \eqref{eq:sol0-sigma_p0} can be rewritten with $\bu_{p,0}\cdot \bt_{f,j}=(\bu_{f,0}-\btheta_0)\cdot \bt_{f,j}$ and that \eqref{eq:continuous-weak-formulation-1i} holds for the initial data.

5. Finally, define $(\wh{\bsi}_{p,0}, \bu_{s,0}, \bgamma_{p,0}) \in \bbX_p \times \bV_s \times \bbQ_p$,
as the unique solution of the problem
\begin{equation}\label{eq:sol0-u_s0}
\begin{array}{ll}
\ds (A \wh\bsi_{p,0}, \btau_p)_{\Omega_p}  + b_s( \bu_{s,0}, \btau_p)
+ b_{\sk}(\bgamma_{p,0}, \btau_p) = b_{\bn_p}(\btheta_0, \btau_p) 
& \ \forall \, \btau_p \in \bbX_p, \\[2ex]
\ds  -b_s ( \wh\bsi_{p,0}, \bv_s ) = 0 & \ \forall \, \bv_s \in \bV_s, \\[2ex]
\ds  - b_{\sk} ( \wh\bsi_{p,0}, \bchi_p ) = 0 & \ \forall \, \bchi_p \in \bbQ_p.
\end{array}
\end{equation}
This is a well-posed problem, since it corresponds to the weak solution of the
mixed elasticity system with Dirichlet data $\btheta_0$ on $\Gamma_{fp}$.
Using \eqref{bound-theta0}, we have the stability bound
\begin{equation}\label{bound-hat-sigmap0}
\|\wh\bsi_{p,0}\|_{\bbX_p} + \|\bu_{s,0}\|_{\bV_s} + \|\bgamma_{p,0}\|_{\bbQ_p}
\le C \|\btheta_0\|_{\bLambda_s} \le C \big(\|\f_f\|_{\bL^2(\Omega_f)}
+ \|p_{p,0}\|_{\H^1(\Omega_p)}
+ \|\bK\nabla p_{p,0}\|_{\bH^1(\Omega_p)} \big).
\end{equation}
We note that
$\wh\bsi_{p,0}$ is an auxiliary variable not used in the initial data.

Combining \eqref{eq:sol0-up0-pp0}--\eqref{eq:sol0-u_s0}, we obtain
$(\bsi_{p,0}, p_{p,0}, \bu_{p,0}, \bT_{f,0}, \bu_{f,0}, \btheta_0)\in \bQ$ and $(\blambda_0, \bu_{s,0}, \bgamma_{p,0})\in \bS$ satisfying \eqref{eq:initial-data-system} with $\wh{\f}_{p,0}$ and $\wh{q}_{p,0}$ such that
\begin{equation}\label{eq:hat-functions}
\begin{array}{c}
\ds (\wh{\f}_{p,0}, \btau_p)_{\Omega_p}= a_e(\bsi_{p,0}, p_{p,0}; \btau_p, 0)- (A(\wh{\bsi}_{p,0}),\btau_p)_{\Omega_p},\qan \\[2ex]
\ds (\wh{q}_{p,0}, w_p)_{\Omega_p}= (s_0\,p_{p,0},w_p)_{\Omega_p} + a_e(\bsi_{p,0}, p_{p,0}; \0, w_p)- b_p( \bu_{p,0}, w_p).
\end{array}
\end{equation}
Using \eqref{bound-sigmap0}, \eqref{bound-hat-sigmap0}, and \eqref{eq:extra-pp0-assumption}, we obtain
\begin{align}
\|\wh{\f}_{p,0}\|_{\bbL^2(\Omega_p)} 
+ \|\wh{q}_{p,0}\|_{\L^2(\Omega_p)}
\,& \leq\, C\, \big(\|\f_f\|_{\bL^2(\Omega_f)} + \|\f_p\|_{\bL^2(\Omega_p)}
+ \|p_{p,0}\|_{\H^1(\Omega_p)} + \|\bK\nabla p_{p,0}\|_{\bH^1(\Omega_p)} \big) \nonumber \\
& \leq\, \wh{C}_{ep}\, \big(\|\f_f\|_{\bL^2(\Omega_f)} + \|\f_p\|_{\bL^2(\Omega_p)} \big),\label{eq:initial-hat-data-bound-again}
\end{align}
hence $(\wh{\f}_{p,0}, \wh{q}_{p,0})\in E_b'$ and \eqref{eq:initial-data-bound} holds.
\end{proof}

\subsection{Main result}

We establish the existence of a solution to \eqref{eq:continuous-weak-formulation-2} as a direct consequence of Lemma~\ref{lem:well-posedness-of-parabolic-problem} and Lemma~\ref{lem:well-posedness-2}.

\begin{thm}\label{thm:well-posedness-main-result}
Assume that the conditions of Lemma~\ref{lem:monotone} are satisfied. Then, for each
\begin{equation*}
\f_f\in \bL^2(\Omega_f),\quad \f_p\in \bL^2(\Omega_p),\quad q_p\in \W^{1,1}(0,T;\L^2(\Omega_p)), \quad p_{p,0}\in \H_p \ (cf. \, \eqref{eq:H-definition}),
\end{equation*}
under the assumptions of Theorem~\ref{thm:domain-is-nonempty} (cf. \eqref{eq:Ct-less-than-r}) and Lemma~\ref{lem:initial-condition} (cf. \eqref{eq:extra-pp0-assumption} and
\eqref{eq:initial-small-data}),
there exists a unique solution of \eqref{eq:continuous-weak-formulation-2}, $(\bp, \br): [0,T] \rightarrow \bQ \times \bS$ with $\bu_f(t)\in \bW_r$ (cf. \eqref{eq:W_r-definition}), $(\bsi_p, p_p) \in \W^{1,\infty}(0,T; \bbL^2(\Omega_p))\times \W^{1,\infty}(0,T;\W_p)$ and $(\bsi_p(0), p_p(0)) = (\bsi_{p,0}, p_{p,0})$, where $\bsi_{p,0}$ is constructed in Lemma \ref{lem:initial-condition}. In addition, $\bu_p(0) =
\bu_{p,0}$, $\bT_f(0) = \bT_{f,0}$, $\bu_f(0) = \bu_{f,0}$, $\btheta(0) = \btheta_0$ and $\lambda(0) = \lambda_0$.
\end{thm}
\begin{proof}
  First, existence of a solution $(\bp, \br):[0,T] \rightarrow \bQ \times \bS$ of \eqref{eq:continuous-weak-formulation-2} with $(\bsi_p, p_p)\in \W^{1,\infty}(0,T;\bbL^2(\Omega_p))\times \W^{1,\infty}(0,T;\W_p)$ and $(\bsi_p(0), p_p(0)) = (\bsi_{p,0}, p_{p,0})$ follows from Lemmas \ref{lem:well-posedness-of-parabolic-problem} and \ref{lem:well-posedness-2}. Moreover, since $(\bsi_p(t), p_p(t))\in \cD$ for each $t \in [0,T]$ (cf. Lemma~\ref{lem:well-posedness-of-parabolic-problem}), it follows from the definition of the domain $\cD$ (cf. \eqref{eq:domain-D}) that 
$\bu_f(t)\in \bW_r$ for $t \in [0,T]$.

We next show that the solution of \eqref{eq:continuous-weak-formulation-2} is unique. 
To that end, let $(\bp, \br)$ and $(\wt{\bp}, \wt{\br})$ be two solutions corresponding to 
the same data and denote $\ov \bp = \bp-\wt{\bp}$ with similar notations for 
the rest of variables. We find that
\begin{equation}\label{eq:uniqueness-2}
\begin{array}{rl}
\partial_t \,\cE \, (\ov\bp) \, (\bq) + \cA \, (\ov\bp) \, (\bq) + \cK_{\bu_f}\,  (\ov\bp) \, (\bq) + 
\cK_{\ov\bu_f}\,  (\wt{\bp}) \, (\bq) + \cB'\,(\ov\br)\, (\bq) & = \0  \quad \forall\, \bq \in \bQ, \\[1ex]
- \cB \, (\ov\bp)\, (\bs) & = \0 \quad \forall\, \bs \in \bS.
\end{array}
\end{equation}
Taking \eqref{eq:uniqueness-2} with $\bq=\ov\bp$ and $\bs=\ov\br$, making use of the continuity of $\cK_{\bw_f}$ in \eqref{eq:continuity-of-Kwf} and the estimates \eqref{eq:positive-bound-E}, \eqref{eq:monotone-01}, and \eqref{eq:monotone-02} in Lemma \ref{lem:monotone} for $\cE$ and $\cA+\cK_{\bw_f}$, we deduce that
\begin{equation}\label{eq:uniqueness-3}
\begin{array}{l}
\ds \frac{1}{2}\,\partial_t\,\left( \|A^{1/2}\,(\ov\bsi_p + \alpha_p\,\ov p_p\,\bI)\|^2_{\bbL^2(\Omega_p)} + s_0\,\|\ov p_p\|^2_{\W_p} \right) \\[2ex]
\ds\quad +\,\left(\alpha_f - C_{\cK}\big(\|\bu_f\|_{\bV_f} + \|\wt{\bu}_f\|_{\bV_f} \big)\right)\|(\ov \bT_f, \ov \bu_f)\|^2 + \mu\, k^{-1}_{\max} \|\ov\bu_p\|^2_{\bL^2(\Omega_p)}
+ c_{\BJS} |\ov\bu_f  - \ov\bphi|^2_{\BJS} \,\leq\, 0\,.
\end{array}
\end{equation}
Integrating in time \eqref{eq:uniqueness-3} from $0$ to $t\in (0,T]$, using $\ov \bsi_p(0)=\0$ and $\ov p_p(0)=0$, we obtain
\begin{equation*}\label{eq:uniqueness-4}
\ds \|A^{1/2}\,(\ov\bsi_p + \alpha_p\,\ov p_p\,\bI)\|^2_{\bbL^2(\Omega_p)} 
+ s_0\,\|\ov p_p \|^2_{\W_p}
+ \int_0^t \|\ov{\bu}_p\|^2_{\bL^2(\Omega_p)}\,ds
+ 2\,C_{\cK}(r_0-r)
\int_0^t \|(\ov\bT_f, \ov\bu_f)\|^2\, ds \leq 0\,,
\end{equation*}
which implies that $A^{1/2}(\ov\bsi_p+\alpha_p \,\ov p_p \,\bI)(t) = \0, \ov\bu_p(t) = \0, \ov \bT_f(t) = \0$, and $\ov\bu_f(t) = \0$ for all $t\in (0,T]$.
In turn, using the inf-sup conditions \eqref{eq:continuous-inf-sup-1} and \eqref{eq:continuous-inf-sup-2} 
in Lemma \ref{lem:inf-sup} for $(\bv_s, \bchi_p, \bphi) = (\ov{\bu}_s, \ov{\bgamma}_p, \ov{\btheta})$ 
and $(w_p, \xi) = (\ov{p}_p, \ov{\lambda})$, respectively, and the first row of \eqref{eq:uniqueness-2}, we get
\begin{equation*}
\begin{array}{l}
\|\ov{\bu}_s\|_{\bV_s} + \|\ov{\bgamma}_p\|_{\bbQ_p} + \|\ov{\btheta}\|_{\bLambda_s}
+ \|\ov{p}_p\|_{\W_p} + \|\ov{\lambda}\|_{\Lambda_p}  \\[2ex]
\ds\quad \,\le\, C \sup_{\0\neq (\btau_p, \bv_p) \in \bbX_p \times \bV_p} \frac{
(A\partial_t(\ov\bsi_p + \alpha_p \, \ov p_p \, \bI), \btau_p)_{\Omega_p} 
+ (\mu \, \bK^{-1} \, \ov\bu_p, \bv_p)_{\Omega_p}
}{\|(\btau_p, \bv_p)\|} = 0\,.
\end{array}
\end{equation*}
Then, $\ov \bu_s(t) = \0$, $\ov \bgamma_p(t) = \0$, $\ov p_p(t) = 0$, $\ov \lambda(t)=0$, and  $\ov\btheta(t)=\0$ for all $t\in (0,T]$, which implies $\ov \bsi_p(t) = \0$ for all $t\in (0,T]$, concluding the proof of uniqueness of the solution to \eqref{eq:continuous-weak-formulation-2}.

Finally, let $\ov\bu_{f,0} := \bu_f(0) - \bu_{f,0}$, with a similar definition and notation for the rest of the variables. 
Due to the fact that $\f_f$, $\f_p$ are independent of time, and the assumed smoothness in time of $q_p$, we can take $t \to 0^+$ in \eqref{eq:continuous-weak-formulation-2}. 
Using that the initial data $(\bp_0, \br_0)$ constructed in Lemma~\ref{lem:initial-condition} satisfies \eqref{eq:domain-D} at
$t=0$, and that $\ov\bsi_{p,0} = \0$ and $\ov p_{p,0} = 0$, we obtain
\begin{subequations}\label{eq:init-0}
\begin{align}
&\ds  a_p(\ov\bu_{p,0},\bv_p) + a_f(\ov\bT_{f,0},\ov\bu_{f,0};\bR_f,\bv_f)
+\kappa_{\ov\bu_{f,0}}(\bT_f(0),\bu_f(0);\bR_f,\bv_f)
+\kappa_{\bu_{f,0}}(\ov\bT_{f,0},\ov\bu_{f,0};\bR_f,\bv_f) \nonumber \\[1ex]
&\ds\quad 
+\ a_{\BJS}(\ov\bu_{f,0},\ov\btheta_0;\bv_f,\bphi)  +  b_\Gamma(\bv_p,\bv_f,\bphi;\ov\lambda_0) =0, \label{init-1} \\[1ex]
&\ds - \ b_\Gamma(\ov\bu_{p,0},\ov\bu_{f,0},\ov\btheta_0;\xi) 
\, = \,0. \label{init-2}
\end{align}
\end{subequations}
Taking $(\bv_p, \bR_f, \bv_f, \bphi, \xi) = (\ov\bu_{p,0}, \ov\bT_{f,0}, \ov\bu_{f,0},\ov\btheta_0, \ov\lambda_0)$ in \eqref{eq:init-0}, using that $\bu_{f,0}\in \bW_r$ (cf. Lemma~\ref{lem:initial-condition}) and $\bu_f(0)\in \bW_r$, and proceeding as in \eqref{eq:uniqueness-3}, we get
\begin{equation*}
\|\ov\bu_{p,0}\|_{\bL^2(\Omega_p)}^2 
+ 2\,C_{\cK}\,(r_0 - r)\,\|(\ov\bT_{f,0}, \ov\bu_{f,0})\|^2 
+ |\ov\bu_{f,0} - \ov\btheta_0|_{\BJS}^2 \,\le\, 0,
\end{equation*}
which implies that
$\ov\bu_{p,0} = \0$, $\ov\bT_{f,0} = \0$, $\ov\bu_{f,0} = \0$ and $\ov\btheta_0\cdot\bt_{f,j}=0$. 
In addition, \eqref{init-2} implies that $\langle \ov\btheta_0\cdot \bn_p, \xi \rangle_{\Gamma_{fp}} =0$ 
for all $\xi\in \H^{1/2}(\Gamma_{fp})$. Since $\H^{1/2}(\Gamma_{fp})$ 
is dense in $\L^2(\Gamma_{fp})$, it follows that $\ov\btheta_0\cdot\bn_p=0$; hence $\ov\btheta_0 = \0$. The inf-sup condition
\eqref{eq:continuous-inf-sup-2}, together with \eqref{init-1},
implies that $\ov\lambda_0 = 0$.
\end{proof}

\begin{rem}\label{rem:non-diff-eq}
As we noted in Remark~\ref{rem: time-diff}, the time differentiated
equation \eqref{eq:continuous-weak-formulation-1d} can be used to recover the
non-differentiated equation \eqref{non-diff-eq}. In particular,
recalling the initial data construction \eqref{eq:sol0-sigma_p0}, let
$$
\forall \, t \in [0,T], \quad  \bbeta_p(t) = \bbeta_{p,0} + \int_0^t \bu_s(s) \, ds,
\quad \brho_p(t) = \brho_{p,0} + \int_0^t \bgamma_p(s) \, ds, \quad
\bpsi(t) = \bpsi_0 + \int_0^t \btheta(s) \, ds.
$$
Then \eqref{non-diff-eq} follows from integrating \eqref{eq:continuous-weak-formulation-1d} from $0$ to 
$t \in (0,T]$ and using the first equation in \eqref{eq:sol0-sigma_p0}.
\end{rem}

Before proving a stability bound for the solution of \eqref{eq:continuous-weak-formulation-2}, we establish a bound at $t = 0$.

\begin{lem}
Under the assumptions of Theorem~\ref{thm:well-posedness-main-result}, there exists a positive constant $C$, independent of $s_{0,\min}$, such that
\begin{align}\label{eq:continuous-stability-initial-1}
\ds \|A^{1/2}(\bsi_p + \alpha_p\,p_p\,\bI)(0)\|_{\bbL^2(\Omega_p)} + \|p_p(0)\|_{\W_p} 
\,\leq\, C\,\left( \|p_{p,0}\|_{\H^1(\Omega_p)} + \|\bK \nabla p_{p,0}\|_{\bH^1(\Omega_p)} 
+ \|\f_p\|_{\bL^2(\Omega_p)} \right).
\end{align}
and
\begin{equation}\label{eq:continuous-stability-initial-2}
\begin{array}{l}
\ds \| \partial_t \,A^{1/2}(\bsi_p + \alpha_p p_p \bI)(0)\|_{\bbL^2(\Omega_p)}
+ \sqrt{s_0} \|\partial_t  p_p(0)\|_{\W_p} \\[2ex]
\ds\quad \,\leq\, C\,\bigg(
\frac{1}{\sqrt{s_0}}\,\|q_p(0)\|_{\L^2(\Omega_p)}
+ \left(1 + \frac{1}{\sqrt{s_0}}\right) \left( \|p_{p,0}\|_{\H^1(\Omega_p)} + \|\bK \nabla p_{p,0}\|_{\bH^1(\Omega_p)} + \|\f_f\|_{\bL^2(\Omega_f)} \right)
\bigg).
\end{array}
\end{equation}
\end{lem}
\begin{proof} 
First, since $(\bsi_{p}(0),p_p(0)) = (\bsi_{p,0},p_{p,0})$, bound \eqref{bound-sigmap0} gives
\eqref{eq:continuous-stability-initial-1}.
On the other hand, using the facts that $(\bsi_p,p_p)\in \W^{1,\infty}(0,T; \bbL^2(\Omega_p))\times \W^{1,\infty}(0,T; \L^2(\Omega_p))$ and $\f_p$ is independent of $t$, we can differentiate in time \eqref{eq:continuous-weak-formulation-1e}--\eqref{eq:continuous-weak-formulation-1f} and combine them with \eqref{eq:continuous-weak-formulation-1d} and \eqref{eq:continuous-weak-formulation-1h} at time $t=0$. Choosing $(\btau_p, w_p, \bv_s, \bchi_p)=(\partial_t \bsi_p(0), \partial_t p_p(0), \bu_s(0), \bgamma_p(0))$, implies
\begin{equation}\label{eq:continuous-stability-8}
\begin{array}{l}
\ds \|\partial_t \,A^{1/2}(\bsi_p + \alpha_p p_p \bI)(0)\|^2_{\bbL^2(\Omega_p)}
+ s_0 \|\partial_t \, p_p(0)\|^2_{\W_p}
\\[2ex]
\ds\quad 
= \, \langle \partial_t \bsi_p(0)\bn_p, \btheta(0)\rangle_{\Gamma_{fp}}
+ (\partial_t p_p(0),\div(\bu_p)(0))_{\Omega_p} +
(q_p(0), \partial_t p_p(0))_{\Omega_p} \,.
\end{array}
\end{equation}
Using the normal trace inequality \eqref{trace-sigma} and estimate \eqref{eq:estimate-sigmap-A} to bound the first term on the right-hand side, as well as the Cauchy--Schwarz and Young's inequalities, we obtain
\begin{align}\label{eq:continuous-stability-9}
&\ds \| \partial_t A^{1/2}(\bsi_p + \alpha_p p_p \bI)(0)\|^2_{\bbL^2(\Omega_p)}
+ s_0 \|\partial_t p_p(0)\|^2_{\W_p} \nonumber \\[1ex]
&\ds \quad  
\,\leq\, C\,\bigg(  \|\partial_t \bdiv(\bsi_p)(0)\|^2_{\bbL^2(\Omega_p)}
+ \frac{1}{s_0} \|q_p(0)\|^2_{\L^2(\Omega_p)}
+ \left( 1 + \frac{1}{s_0} \right)\|\btheta(0)\|^2_{\bLambda_s} 
+ \frac{1}{s_0}\,\|\div (\bu_p) (0)\|^2_{\L^2(\Omega_p)} \bigg) \nonumber \\[1ex]
&\ds \quad +\,
\delta\left( \|\partial_t \,A^{1/2}(\bsi_p + \alpha_p p_p \bI)(0)\|^2_{\bbL^2(\Omega_p)}
+ s_0 \|\partial_t \, p_p(0)\|^2_{\W_p} \right) \,.
\end{align}
In turn, using the fact that $\btheta(0)=\btheta_0$, $\bu_p(0)=\bu_{p,0}$ (cf.
Theorem~\ref{thm:well-posedness-main-result}), bound \eqref{bound-theta0}, and the identity \eqref{eq:sol0-up0-pp0}, we find that
\begin{equation}\label{eq:continuous-stability-10}
\begin{split}
&\|\btheta(0)\|_{\bLambda_s}
\,\leq\, C\, \big( \|\f_f\|_{\bL^2(\Omega_f)}
+ \|p_{p,0}\|_{\H^1(\Omega_p)}
+ \|\bK\nabla p_{p,0}\|_{\bH^1(\Omega_p)} \big), \\
&\|\div(\bu_p)(0)\|_{\L^2(\Omega_p)} 
\,\leq\, C\,\| \div(\bK \nabla p_{p,0})\|_{\L^2(\Omega_p)}.
\end{split}
\end{equation}
In addition, from \eqref{eq:continuous-weak-formulation-1e} and the fact that $\f_p$ does not depend on $t$, we deduce
\begin{equation}\label{eq:continuous-stability-11}
\|\partial_t \bdiv(\bsi_p)(0)\|_{\bbL^2(\Omega_p)} \,=\, 0\,.
\end{equation}
Then, combining \eqref{eq:continuous-stability-9} with   \eqref{eq:continuous-stability-10}--\eqref{eq:continuous-stability-11}, and taking $\delta$ small enough, we obtain \eqref{eq:continuous-stability-initial-2}.
\end{proof}

We end this section with establishing regularity and a stability bound for the solution of \eqref{eq:continuous-weak-formulation-2}.
\begin{thm}\label{thm: continuous stability}
Under the assumptions of Theorem~\ref{thm:well-posedness-main-result}, if 
$q_p \in \H^1(0,T;\L^2(\Omega_p))$, then the solution of \eqref{eq:continuous-weak-formulation-2} has regularity
$\bT_f \in \H^1(0,T;\bbX_f)$, $\bu_f \in \H^1(0,T;\bV_f)$, $\bsi_p \in \W^{1,\infty}(0,T;\bbX_p)$, $\bu_s \in \L^\infty(0,T;\bV_s)$, $\bgamma_p \in \L^\infty(0,T;\bbQ_p)$, $\bu_p \in \L^2(0,T;\bV_p)\cap \H^1(0,T;\bL^2(\Omega_p))$, $p_p \in \W^{1,\infty}(0,T;\W_p)$, $\lambda \in \H^1(0,T;\Lambda_p)$, and $\btheta \in \L^\infty(0,T;\bLambda_s)$.
In addition, there exists a positive constant $C$, independent of $s_{0,\min}$, such that
\begin{align}\label{eq:continuous-stability}
&\ds \|A^{1/2}(\bsi_p + \alpha_p p_p \bI)\|_{\W^{1,\infty}(0,T;\bbL^2(\Omega_p))}
+ \|\bdiv(\bsi_p)\|_{\bbL^2(\Omega_p)}
\nonumber \\[1ex]
& \ds\quad + \, 
\sqrt{s_0} \|p_p\|_{\W^{1,\infty}(0,T;\W_p)} 
+ \|p_p\|_{\H^1(0,T;\W_p)} 
+ \|\bu_p\|_{\L^2(0,T;\bV_p)} 
+ \|\partial_t \bu_p\|_{\L^2(0,T;\bL^2(\Omega_p))}
+ \|\bT_f\|_{\H^1(0,T;\bbX_f)}
\nonumber \\[1ex]
&\ds\quad 
+\, \|\bu_f\|_{\H^1(0,T;\bV_f)} 
+ |\bu_f - \btheta|_{\H^1(0,T;\BJS)}
+ \|\btheta\|_{\L^\infty(0,T;\bLambda_s)}
+ \|\btheta\|_{\L^2(0,T;\bLambda_s)}
+ \|\lambda\|_{\H^1(0,T;\Lambda_p)} 
\nonumber \\[1ex]
&\ds\quad
+ \|\bu_s\|_{\L^\infty(0,T;\bV_s)}
+ \|\bu_s\|_{\L^2(0,T;\bV_s)}
+ \|\bgamma_p\|_{\L^\infty(0,T;\bbQ_p)}
+ \|\bgamma_p\|_{\L^2(0,T;\bbQ_p)} \nonumber \\[1ex]
&\ds \leq\, C\,\left( 
\|\f_f\|_{\bL^2(\Omega_f)} 
+ \|\f_p\|_{\bL^2(\Omega_p)}
+ \|q_p\|_{\H^1(0,T;\L^2(\Omega_p))}
+ \frac{1}{\sqrt{s_0}}\,\|q_p(0)\|_{\L^2(\Omega_p)} \right.
\nonumber
\\[1ex]
&\ds\quad
\left.
+ \, \sqrt{s_0}\,\|p_{p,0}\|_{\W_p}
+ \left(1 + \frac{1}{\sqrt{s_0}}\right)\left( \|p_{p,0}\|_{\H^1(\Omega_p)} + \|\bK \nabla p_{p,0}\|_{\bH^1(\Omega_p)} \right) \right)\,.
\end{align}
\end{thm}
\begin{proof}
Choosing $(\btau_p, w_p, \bv_p, \bR_f, \bv_f, \bphi, \xi, \bv_s, \bchi_p) = (\bsi_p, p_p, \bu_p, \bT_f, \bu_f, \btheta,\lambda, \bu_s, \bgamma_p)$ in \eqref{eq:continuous-weak-formulation-2}, we get
\begin{equation}\label{eq:continuous-stability-1}
\begin{array}{l}
\ds \frac{1}{2}\,\partial_t \big(\|A^{1/2}(\bsi_p + \alpha_p p_p \bI)\|^2_{\bbL^2(\Omega_p)}
+ s_0 \|p_p\|^2_{\W_p} \big) 
+ a_p(\bu_p, \bu_p) +  a_f(\bT_f, \bu_f; \bT_f, \bu_f) \\[2ex]
\ds\,\,\,\, +\, \kappa_{\bu_f}(\bT_f, \bu_f; \bT_f, \bu_f)
+ a_{\BJS}(\bu_f, \btheta; \bu_f, \btheta) 
= (q_p, p_p)_{\Omega_p}
+ (\f_f,\bu_f - \kappa_1 \bdiv(\bT_f))_{\Omega_f} 
+ (\f_p, \bu_s)_{\Omega_p}.
\end{array}
\end{equation}
We integrate \eqref{eq:continuous-stability-1} from $0$ to $t \in (0,T]$, 
use the fact that $\bu_f: [0,T] \rightarrow \bW_r$ (cf. \eqref{eq:W_r-definition}), the stability bounds \eqref{eq:positive-bound-E} and \eqref{eq:positive-bound-A+Kwf} 
in Lemma \ref{lem:monotone}, and the Cauchy-Schwarz and Young's inequalities, to deduce
\begin{align}
& \ds \|A^{1/2}(\bsi_p + \alpha_p p_p \bI)(t)\|^2_{\bbL^2(\Omega_p)} 
+ s_0 \|p_p (t)\|^2_{\W_p}
+ \int_0^t \left( \|\bu_p\|^2_{\bL^2(\Omega_p)} 
+ \|(\bT_f, \bu_f)\|^2
+ |\bu_f - \btheta|^2_{\BJS} \right) ds 
\nonumber \\[1ex]
& \ds \leq C\,\bigg( \|\f_f\|^2_{\bL^2(\Omega_f)} 
+ \|\f_p\|^2_{\bL^2(\Omega_p)} + \int_0^t \|q_p\|^2_{\L^2(\Omega_p)} \, ds 
+ \|A^{1/2}(\bsi_p + \alpha_p p_p \bI)(0)\|^2_{\bbL^2(\Omega_p)} 
+ s_0 \|p_p (0)\|^2_{\W_p}
\bigg) \nonumber \\[1ex]
& \ds \quad  +\, \delta \int_0^t \left( \|(\bT_f, \bu_f)\|^2 + \|p_p\|^2_{\W_p} + \|\bu_s\|^2_{\bV_s} \right)ds\,. \label{eq:continuous-stability-2}
\end{align}
In turn, taking $(\btau_p, w_p, \bv_p, \bR_f, \bv_f, \bphi) =(\btau_p,0,\bv_p,\0,\0,\0)$ in the first equation of \eqref{eq:continuous-weak-formulation-2}, we obtain
	\begin{equation*}
	b_s(\bu_s,\btau_p) + b_{\sk}(\bgamma_p,\btau_p) - b_{\bn_p}(\btau_p,\btheta)
	+ b_p(p_p,\bv_p) + b_\Gamma(\0,\bv_p,\0; \lambda)
	= -a_e(\partial_t\,\bsi_p, \partial_t\, p_p;\btau_p, 0)
	- a_p(\bu_p,\bv_p)\,,
	\end{equation*}
	which combined with the inf-sup conditions \eqref{eq:continuous-inf-sup-1}--\eqref{eq:continuous-inf-sup-2} 
	in Lemma \ref{lem:inf-sup}, yields
	\begin{equation}\label{eq:continuous-stability-inf-sup-1}
	\|\bu_s\|_{\bV_s} + \|\bgamma_p\|_{\bbQ_p} + \|\btheta\|_{\bLambda_s}  
	+ \|p_p\|_{\W_p} + \|\lambda\|_{\Lambda_p}
	\,\leq\, C\,\left( \|\partial_t\,A^{1/2}(\bsi_p+\alpha_p\,p_p\,\bI)\|_{\bbL^2(\Omega_p)}
	+ \|\bu_p\|_{\bL^2(\Omega_p)}
	\right) .
	\end{equation}
	On the other hand, choosing $\bv_s = \bdiv (\bsi_p)$ and $w_p=\div(\bu_p)$ 
	in \eqref{eq:continuous-weak-formulation-2}, and applying 
	the Cauchy--Schwarz inequality, we deduce that 
	\begin{equation}\label{eq:continuous-stability-5}
	\begin{array}{c}
	\ds \|\bdiv(\bsi_p)\|_{\bbL^2(\Omega_p)}
	\,\leq\, \|\f_p\|_{\bL^2(\Omega_p)}\,,\qan \\[2ex]
	\ds \|\div(\bu_p)\|_{\L^2(\Omega_p)} 
	\,\leq\, C\,\left( \|q_p\|_{\L^2(\Omega_p)} 
	+ \|\partial_t\,A^{1/2}(\bsi_p + \alpha_p p_p \bI) \|_{\bbL^2(\Omega_p)} 
	+ s_0\,\|\partial_t p_p\|_{\W_p}
	\right)\,.
	\end{array}
	\end{equation}
Then, combining \eqref{eq:continuous-stability-2} with \eqref{eq:continuous-stability-inf-sup-1} and \eqref{eq:continuous-stability-5}, and choosing $\delta$ small enough, we obtain
\begin{align}\label{eq:prelim-continuous-stability-1}
&\ds \|A^{1/2}(\bsi_p + \alpha_p p_p \bI)(t)\|^2_{\bbL^2(\Omega_p)}
+ s_0 \|p_p(t)\|^2_{\W_p} 
+ \|\bdiv(\bsi_p)\|^2_{\bbL^2(\Omega_p)}
\nonumber \\[1ex]
&\ds \quad +\, \int_0^t \left( \|p_p\|^2_{\W_p} + \|\bu_p\|^2_{\bV_p} 
+ \|(\bT_f, \bu_f)\|^2
+ |\bu_f - \btheta|^2_{\BJS} + \|\btheta\|^2_{\bLambda_s} 
+ \|\lambda\|^2_{\Lambda_p} + \|\bu_s\|^2_{\bV_s} + \|\bgamma_p\|^2_{\bbQ_p} \right) ds \nonumber \\[1ex]
&\ds \leq C\,\bigg( \|\f_f\|^2_{\bL^2(\Omega_f)} 
+ \|\f_p\|^2_{\bL^2(\Omega_p)}
+ \int_0^t \|q_p\|^2_{\L^2(\Omega_p)}\,ds 
+ \|A^{1/2}(\bsi_p + \alpha_p\,p_p\,\bI)(0)\|^2_{\bbL^2(\Omega_p)} 
\nonumber \\[1ex]
&\ds \quad  + \, s_0 \|p_p(0)\|^2_{\W_p}
+ \int_0^t \left( \|\partial_t\,A^{1/2}(\bsi_p + \alpha_p p_p \bI)\|^2_{\bbL^2(\Omega_p)}
+ s^2_0\,\|\partial_t p_p\|^2_{\W_p} \right) ds \bigg)\,.
\end{align}

Next, in order to bound the last two terms in \eqref{eq:prelim-continuous-stability-1}, we take a finite difference in time of the whole system \eqref{eq:continuous-weak-formulation-2}. In particular, given $t \in [0,T)$ and $s > 0$ with $t+s \le T$, let $\dts\phi := \frac{\phi(t+s) - \phi(t)}{s}$. Applying this operator to \eqref{eq:continuous-weak-formulation-2}, 
noting that $\partial^s_t\f_f = \0$ and $\partial^s_t\f_p=\0$ since both $\f_f$ and $\f_p$ are independent of $t$, and testing with $(\btau_p, w_p, \bv_p, \bR_f, \bv_f, \bphi, \xi, \bv_s, \bchi_p) = (\partial^s_t\bsi_p, \partial^s_t p_p, \partial^s_t\bu_p,\partial^s_t\bT_f, \partial^s_t\bu_f, \partial^s_t\btheta, \partial^s_t\lambda, \partial^s_t\bu_s, \partial^s_t\bgamma_p)$, 
similarly to \eqref{eq:continuous-stability-1}, we get
\begin{equation*}\label{eq:continuous-stability-4}
\begin{array}{l}
\ds \frac{1}{2}\,\partial_t \big(\|\partial^s_t\,A^{1/2}(\bsi_p + \alpha_p p_p \bI)\|^2_{\bbL^2(\Omega_p)}
+ s_0 \|\partial^s_t\,p_p\|^2_{\W_p} \big) 
+ a_p(\partial^s_t\bu_p, \partial^s_t\bu_p) +  a_f(\partial^s_t\bT_f, \partial^s_t\bu_f; \partial^s_t\bT_f, \partial^s_t\bu_f) \\[2ex]
\ds\quad +\, \kappa_{\partial^s_t\bu_f}(\bT_f(t), \bu_f(t); \partial^s_t\bT_f, \partial^s_t\bu_f)
+ \kappa_{\bu_f(t+s)}(\partial^s_t\bT_f, \partial^s_t\bu_f; \partial^s_t\bT_f, \partial^s_t\bu_f)
+ a_{\BJS}(\partial^s_t\bu_f, \partial^s_t\btheta; \partial^s_t\bu_f, \partial^s_t\btheta) \\[2ex]
\ds\quad =\, (\partial^s_t q_p, \partial^s_t p_p)_{\Omega_p} \,,
\end{array}
\end{equation*}
We integrate from 0 to $t \in (0,T)$, use that $\bu_f:[0,T]\to \bW_r$, the continuity of $\kappa_{\bw_f}$ (cf. \eqref{eq:continuity-of-Kwf}) and the positivity bounds of $a_p$, $a_f$, and $a_{\BJS}$ in Lemma \ref{lem:monotone} (cf. \eqref{eq:monotone-01} and \eqref{eq:monotone-02}), and take $s \to 0$, obtaining
\begin{align}
& \ds \|\partial_t\,A^{1/2}(\bsi_p + \alpha_p p_p \bI)(t)\|^2_{\bbL^2(\Omega_p)} 
+ s_0 \|\partial_t\,p_p (t)\|^2_{\W_p} \nonumber \\
& \ds\quad + \int_0^t \left( \|\partial_t\bu_p\|^2_{\bL^2(\Omega_p)} 
+ 2\,C_{\cK}\,(r_0 - r)\|(\partial_t\bT_f, \partial_t\bu_f)\|^2 
+ |\partial_t \bu_f - \partial_t \btheta|^2_{\BJS} \right) ds \label{eq:continuous-stability-6} \\
& \ds \leq C \bigg( \int_0^t \|\partial_t q_p\|^2_{\L^2(\Omega_p)} \, ds 
+ \|\partial_t\,A^{1/2}(\bsi_p + \alpha_p p_p \bI)(0)\|^2_{\bbL^2(\Omega_p)} + s_0\,\|\partial_t\,p_p (0)\|^2_{\W_p} \bigg) 
+ \delta \int_0^t 
\|\partial_t\,p_p\|^2_{\W_p} \, ds . \nonumber 
\end{align}
In turn, using the inf-sup conditions \eqref{eq:continuous-inf-sup-1}--\eqref{eq:continuous-inf-sup-2} 
in Lemma \ref{lem:inf-sup}, we find that for a.e. $t \in (0,T)$
\begin{equation}\label{eq:auxiliar-bound-1}
\begin{array}{c}
\|\bu_s(t)\|_{\bV_s} + \|\bgamma_p(t)\|_{\bbQ_p} + \|\btheta(t)\|_{\bLambda_s}  
\,\leq\, C\,\|\partial_t\,A^{1/2}(\bsi_p+\alpha_p\,p_p\,\bI)(t)\|_{\bbL^2(\Omega_p)}\,, \\[1ex]
\qan \|\partial_t\,p_p(t)\|_{\W_p} + \|\partial_t\,\lambda(t)\|_{\Lambda_p}
\,\leq\, C\,\|\partial_t\,\bu_p(t)\|_{\bL^2(\Omega_p)} \,,
\end{array}
\end{equation}
where the second bound is obtained by applying the operator $\dts$ and taking $s \to 0$. Then, combining \eqref{eq:continuous-stability-6} with \eqref{eq:auxiliar-bound-1}, and taking $\delta$ small enough, yields
\begin{align}\label{eq:prelim-continuous-stability-2}
	& \ds \|\partial_t\,A^{1/2}(\bsi_p + \alpha_p p_p \bI)(t)\|^2_{\bbL^2(\Omega_p)} 
+ s_0 \|\partial_t\,p_p (t)\|^2_{\W_p}
+ \| \bu_s(t)\|^2_{\bV_s}
+ \| \bgamma_p(t) \|^2_{\bbQ_p}
+ \| \btheta(t) \|^2_{\bLambda_s}
\nonumber \\
& \ds\quad + \int_0^t \left(
\| \partial_t\, p_p\|^2_{\W_p}+
\|\partial_t\bu_p\|^2_{\bL^2(\Omega_p)} 
+ \|(\partial_t\bT_f, \partial_t\bu_f)\|^2 
+ |\partial_t \bu_f - \partial_t \btheta|^2_{\BJS}
+ \|\partial_t \lambda\|^2_{\Lambda_p} \right) ds \nonumber \\
& \ds \leq C\,\left( \int_0^t \|\partial_t q_p\|^2_{\L^2(\Omega_p)} \,ds
+ \|\partial_t\,A^{1/2}(\bsi_p + \alpha_p p_p \bI)(0)\|^2_{\bbL^2(\Omega_p)}
+ s_0\,\|\partial_t\,p_p (0)\|^2_{\W_p} \right) \,.
\end{align}
Bound \eqref{eq:continuous-stability} follows by combining \eqref{eq:prelim-continuous-stability-1} and \eqref{eq:prelim-continuous-stability-2} with the bounds at $t = 0$ \eqref{eq:continuous-stability-initial-1}--\eqref{eq:continuous-stability-initial-2}. The bound implies the stated solution regularity.
\end{proof}

\section{Semidiscrete continuous-in-time approximation}

In this section we introduce and analyze the semidiscrete continuous-in-time approximation 
of \eqref{eq:continuous-weak-formulation-3}. 
We analyze its solvability by employing the strategy developed in Section \ref{sec:well-posedness-model}.
In addition, we derive error estimates with rates of convergence. At the end of the section we introduce the fully discrete scheme based on the backward Euler time discretization and give a short outline of its analysis.

Let $\cT_h^f$ and $\cT_h^p$ be shape-regular and quasi-uniform \cite{ciarlet1978} affine finite element partitions of $\Omega_f$ and $\Omega_p$, respectively, where $h$ is the maximum element diameter. The two partitions may be non-matching along the interface $\Gamma_{fp}$. For the discretization, we consider the following conforming finite element spaces:
\begin{equation*}
\bbX_{fh}\times \bV_{fh}\subset \bbX_f \times \bV_f, \quad
\bbX_{ph}\times \bV_{sh} \times \bbQ_{ph} \subset \bbX_p \times \bV_s \times \bbQ_p, \quad 
\bV_{ph} \times \W_{ph} \subset \bV_p \times \W_p.
\end{equation*}
We choose $(\bbX_{ph},\bV_{sh},\bbQ_{ph})$ to be any stable triple for mixed elasticity with weakly imposed stress symmetry, such as the
Amara--Thomas \cite{at1979}, PEERS \cite{abd1984}, Stenberg
\cite{stenberg1988}, Arnold--Falk--Winther \cite{afw2007,awanou2013},
or Cockburn--Gopalakrishnan--Guzman \cite{cgg2010} families of spaces. 
We take $(\bV_{ph},\W_{ph})$ to be any stable mixed finite element Darcy spaces, such as the Raviart--Thomas (RT) or Brezzi--Douglas--Marini (BDM) spaces \cite{Brezzi-Fortin}. We note that these spaces satisfy
\begin{equation}\label{eq: div-prop}
\bdiv(\bbX_{ph})=\bV_{sh}, \qquad 
\div(\bV_{ph})=\W_{ph}.
\end{equation}
Since $\bV_{sh}$ and $\W_{ph}$ contain discontinuous piecewise polynomials, the method exhibits local poroelastic momentum conservation (cf. \eqref{eq:continuous-weak-formulation-1e}) and local mass conservation for the Darcy fluid (cf. \eqref{eq:continuous-weak-formulation-1h}). We further note that an inf-sup condition is not required for the pair $(\bbX_{fh},\bV_{fh})$. Therefore we can take any $\bbH(\bdiv; \Omega_f)$-conforming space for $\bbX_{fh}$, such as the RT or BDM spaces, combined with continuous piecewise polynomials for $\bV_{fh}$. For the Lagrange multipliers, we choose the non-conforming approximations
\begin{equation}\label{defn-Lambda-h}
\Lambda_{ph} := \bV_{ph} \cdot \bn_p \vert_{\Gamma_{fp}}, \quad 
\bLambda_{sh} := \bbX_{ph} \bn_p \vert_{\Gamma_{fp}},
\end{equation}
which consist of discontinuous piecewise polynomials and are equipped with $\L^2$-norms. 
\begin{rem}
We note that, since $\H^{1/2}(\Gamma_{fp})$ is dense in $\L^2(\Gamma_{fp})$, 
\eqref{eq:continuous-weak-formulation-1i} and \eqref{eq:continuous-weak-formulation-1k}
in the continuous weak formulation hold for test functions in $\L^2(\Gamma_{fp})$, 
assuming that the solution is smooth enough so that the traces are well-defined in $\L^2(\Gamma_{fp})$; e.g., $\bu_p \in \bH^{1/2+\epsilon}(\Omega_p)$ for some $\epsilon > 0$. 
In particular, these equations hold for $\xi_h \in \Lambda_{ph}$ and $\bphi_h \in \bLambda_{sh}$, respectively.
\end{rem}

Now, we group the spaces, unknowns and test functions similarly to the continuous case:
\begin{equation*}
\begin{array}{c}
\ds \bQ_h := \bbX_{ph}\times \W_{ph} \times \bV_{ph} \times \bbX_{fh} \times \bV_{fh} \times \bLambda_{sh},\quad
\bS_h := \Lambda_{ph} \times \bV_{sh}\times \bbQ_{ph} , \\ [2ex]
\ds \bp_h := (\bsi_{ph},  p_{ph}, \bu_{ph}, \bT_{fh}, \bu_{fh}, \btheta_h )\in \bQ_h,\quad 
\br_h := (\lambda_h, \bu_{sh}, \bgamma_{ph} )\in \bS_h, \\[1ex]
\ds \bq_h := (\btau_{ph},  w_{ph}, \bv_{ph}, \bR_{fh}, \bv_{fh}, \bphi_h )\in \bQ_h,\quad 
\bs_h := (\xi_h, \bv_{sh}, \bchi_{ph} )\in \bS_h,
\end{array}
\end{equation*}
where the spaces $\bQ$ and $\bS$ are respectively endowed with the norms
\begin{align*}
\|\bq_h\|_{\bQ_h}^2 & = \|\btau_{ph}\|_{\bbX_{p}}^2 + \|w_{ph}\|_{\W_{p}}^2 + \|\bv_{ph}\|_{\bV_{p}}^2 + \|\bR_{fh}\|_{\bbX_{f}}^2 + \|\bv_{fh}\|_{\bV_{f}}^2 + \|\bphi_h\|_{\bLambda_{sh}}^2 \nonumber \,, \\[1ex]
\|\bs_h\|_{\bS_h}^2 & = \|\xi_h\|_{\Lambda_{ph}}^2 + \|\bv_{sh}\|_{\bV_p}^2 + \|\bchi_{ph}\|_{\bbQ_p}^2 \,, \nonumber
\end{align*}
with $\ds \|\bphi_h\|_{\bLambda_{sh}} = \|\bphi_h\|_{\bL^2(\Gamma_{fp})}$ and $\ds \|\xi_h\|_{\Lambda_{ph}}=\|\xi_h\|_{\L^2(\Gamma_{fp})}$.
The semidiscrete continuous-in-time approximation to \eqref{eq:continuous-weak-formulation-3} is: Find $\ds (\bp_h, \br_h):[0,T]\rightarrow \bQ_h \times \bS_h$ such that for a.e. $t \in (0,T)$,
\begin{equation}\label{eq: NS-Biot-semiformulation-5}
\begin{array}{rll}
\ds \frac{\partial}{\partial t}\,\cE(\bp_h(t)) + (\cA + \cK_{\bu_{fh}(t)})(\bp_h(t)) + \cB'(\br_h(t)) & = & \bF(t) \qin \bQ_h'\,, \\ [2ex]
\ds - \cB(\bp_h(t)) & = & \bG \qin \bS_h'\,.
\end{array}
\end{equation}

\begin{rem}
Lemma~\ref{lem:cont} holds for the non-conforming Lagrange multipliers spaces (cf. \eqref{defn-Lambda-h}), even though the trace inequalities \eqref{eq:trace-inequality-1} and \eqref{trace-sigma} no longer hold. In particular, for the continuity of the bilinear forms
$b_{\bn_p}(\btau_{ph}, \bphi_h)$ and $b_{\Gamma}(\bv_{fh},\bv_{ph},\bphi_h;\xi_h)$, using the discrete trace-inverse inequality for piecewise polynomial
functions, $\|\varphi\|_{\L^2(\Gamma_{fp})} \le C h^{-1/2}\|\varphi\|_{\L^2(\Omega_p)}$, we have
\begin{equation*}
\begin{array}{c}
\ds b_{\bn_p}(\btau_{ph}, \bphi_h)
\,\le\, C h^{-1/2}\|\btau_{ph}\|_{\bbL^2(\Omega_p)}\|\bphi_h\|_{\bL^2(\Gamma_{fp})}, \\[2ex]
\ds b_{\Gamma}(\bv_{fh},\bv_{ph},\bphi_h;\xi_h)
\,\le\, C\,\left( \|\bv_{fh}\|_{\bV_f} + h^{-1/2}\|\bv_{ph}\|_{\bL^2(\Omega_p)}
+ \|\bphi_h\|_{\bL^2(\Gamma_{fp})} \right)\|\xi_h\|_{\L^2(\Gamma_{fp})} \,.
\end{array}
\end{equation*}
Therefore these bilinear forms are continuous for any given mesh and so are the operators $\cA$ and $\cB$; hence Lemma~\ref{lem:cont} holds.
\end{rem}

We next state the discrete inf-sup conditions that are satisfied by the finite element spaces. 
\begin{lem}\label{lem: discrete inf-sup}
  There exists constants $\wt{\beta}_1, \wt{\beta}_2 > 0$ such that for all $(\bv_{sh}, \bchi_{ph}, \bphi_h)\in \bV_{sh}\times \bbQ_{ph}\times \bLambda_{sh}$,
\begin{equation}\label{eq:discrete-inf-sup-1}
\wt{\beta}_1 \left( \|\bv_{sh}\|_{\bV_{s}} + \|\bchi_{ph}\|_{\bbQ_{p}} + \|\bphi_h\|_{\bLambda_{sh}} \right)
\,\leq\, \underset{\0 \neq \btau_{ph}\in \bbX_{ph}}{\sup}
\frac{b_s(\btau_{ph},\bv_{sh}) + b_{\sk}(\btau_{ph},\bchi_{ph}) + b_{\bn_p}(\btau_{ph},\bphi_p)}{\|\btau_{ph}\|_{\bbX_{p}}},
\end{equation}
and for all $(w_{ph},\xi_h) \in \W_{ph}\times \Lambda_{ph}$,
\begin{equation}\label{eq:discrete-inf-sup-2}
\wt{\beta}_2 \left( \|w_{ph}\|_{\W_{p}} + \|\xi_h\|_{\Lambda_{ph}} \right)  
\,\leq\, \underset{\0 \neq \bv_{ph} \in \bV_{ph}}{\sup}
\frac{b_p(\bv_{ph},w_{ph})+ b_{\Gamma}(\bv_{ph},\0,\0;\xi_h)}{\|\bv_{ph}\|_{\bV_{p}}}.
\end{equation}
\end{lem}
\begin{proof}
Inequality \eqref{eq:discrete-inf-sup-1} can be shown using the arguments developed in \cite[Theorem 4.1]{akny2018-a}, whereas \eqref{eq:discrete-inf-sup-2} can be proved similarly to \cite[eq. (5.7) and Lemma 5.1]{aeny2019}. 
\end{proof}

\subsection{Existence and uniqueness of a solution}

The existence of a solution to \eqref{eq: NS-Biot-semiformulation-5} will be established following the proof of solvability of the continuous formulation \eqref{eq:continuous-weak-formulation-2} developed in Section~\ref{sec:well-posedness-model}. To this end, we define a discrete version of the domain $\cD$:

\medskip

$\cD_h :=\, \ds \Big\{  (\bsi_{ph}, p_{ph})\in \bbX_{ph}\times \W_{ph} :$ \ \  for given  $(\f_f, \f_p)\in \bL^2(\Omega_f)\times \bL^2(\Omega_p)$, there exist \\
\indent 

$((\bu_{ph}, \bT_{fh}, \bu_{fh}, \btheta_h),(\lambda_h,\bu_{sh}, \bgamma_{ph}))\in (\bV_{ph} \times \bbX_{fh} \times \bV_{fh} \times \bLambda_{sh}) \times \bS_h$ with $\bu_{fh}\in \bW_r$, such that
\begin{equation}\label{discrete-domain}
\arraycolsep=1.7pt
\begin{array}{rcll}    
(\cE+\cA+\cK_{\bu_{fh}})(\bp_{h}) + \cB'(\br_{h}) & = & \wh{\bF}_{h} & \mbox{in }\,  \bQ_h', \\[2ex]
- \cB(\bp_{h}) & = & \bG & \mbox{in }\, \bS_h',
\end{array}
\end{equation}
\indent where
$$
\wh{\bF}_{h}(\bq_h) = (\wh{\f}_{ph}, \btau_{ph})_{\Omega_p} 
+ (\wh{q}_{ph}, w_{ph})_{\Omega_p}
+ (\f_f, \bv_{fh} - \kappa_1\,\bdiv(\bR_{fh}))_{\Omega_f} \quad \forall\,\bq_h\in \bQ_h,
$$
\indent for some $(\wh{\f}_{ph},\wh{q}_{ph})\in E_b'$
satisfying 
\begin{equation}\label{eq:domain-Dh-hat}
\ds \|\wh{\f}_{ph}\|_{\bbL^2(\Omega_p)} + \|\wh{q}_{ph}\|_{\L^2(\Omega_p)} 
\,\leq\, \wh{C}_{ep,h} \left( \|\f_f\|_{\bL^2(\Omega_f)} + \|\f_p\|_{\bL^2(\Omega_p)}\right)
\end{equation}
\indent with $\wh{C}_{ep,h}$ a fixed positive constant $\Big\}$.

\noindent The constant $\wh{C}_{ep,h}$ is determined from the construction of compatible discrete initial data $(\bp_{h,0}, \br_{h,0})$, which is discussed next.

\begin{lem}\label{lem: discrete initial condition}
  Let $(\f_f, \f_p)\in \bL^2(\Omega_f)\times \bL^2(\Omega_p)$.
  Assume that the conditions of Lemma~\ref{lem:monotone} and
Lemma~\ref{lem:initial-condition} are satisfied.
Assume in addition that the data satisfy 
\begin{equation}\label{eq:initial-small-data-discrete}
C_{\wt{\cJ}_0}\,\left(\|\f_f\|_{\bL^2(\Omega_f)}
+ \|p_{p,0}\|_{\H^1(\Omega_p)}
+ \|\bK\nabla p_{p,0}\|_{\bH^1(\Omega_p)} \right) \,\leq\, r,
\end{equation}
for $r\in (0,r_0)$, where $r_0$ is defined in \eqref{eq:r0-definition} 
and $C_{\wt{\cJ}_0}$ is defined in \eqref{eq:NS-discrete-initial-data} below. 
Then, there exist discrete initial data $(\bsi_{ph,0},p_{ph,0}) \in \cD_h$. In particular, there exist $\bp_{h,0} := (\bsi_{ph,0},p_{ph,0}, \bu_{ph,0},\bT_{fh,0},\bu_{fh,0},\btheta_{h,0})\in \bQ_h$ and $\br_{h,0} := (\lambda_{h,0},\bu_{sh,0},\bgamma_{ph,0})\in \bS_h$ with $\bu_{fh,0} \in \bW_r$, satisfying
\begin{equation}\label{eq: discrete initial condition}
\arraycolsep=1.7pt
\begin{array}{rcll}    
(\cE+\cA+\cK_{\bu_{fh,0}})(\bp_{h,0}) + \cB'(\br_{h,0}) & = & \wh{\bF}_{h,0} & \mbox{in }\,  \bQ_h', \\[2ex]
- \cB(\bp_{h,0}) & = & \bG_0 & \mbox{in }\, \bS_h',
\end{array}
\end{equation}
where $\bG_0(\bs_h) := (\f_p,\bv_{sh})_{\Omega_p} \ \forall \, \bs_h \in \bS_h$ and 
$\wh{\bF}_{h,0}(\bq_h) = (\f_f, \bv_{fh} - \kappa_1\,\bdiv(\bR_{fh}))_{\Omega_f}
+ (\wh{\f}_{ph,0}, \btau_{ph})_{\Omega_p} 
+ (\wh{q}_{ph,0}, w_{ph})_{\Omega_p} \ \forall\,\bq_h\in \bQ_h$ 
for some $(\wh{\f}_{ph,0},\wh{q}_{ph,0})\in E_b'$ 
satisfying
\begin{equation}\label{eq:initial-data-bound-h}
\|\wh{\f}_{ph,0}\|_{\bbL^2(\Omega_p)} + \|\wh{q}_{ph,0}\|_{\L^2(\Omega_p)} 
\,\leq\, \wh{C}_{ep,h} \left(\|\f_f\|_{\bL^2(\Omega_f)} + \|\f_p\|_{\bL^2(\Omega_p)} \right),
\end{equation}
where $\wh{C}_{ep,h}$ is specified in \eqref{bound-hat-f-q} below.
\end{lem}
\begin{proof}
The construction is based on a modification of the step-by-step procedure for the continuous initial data
$(\bp_0,\br_0)$ presented in Lemma~\ref{lem:initial-condition}. In each step the discrete initial data is defined as a suitable projection of the continuous initial data. 

1. 
Define $\btheta_{h,0} := P^{\bLambda_s}_h(\btheta_0)$, where $P^{\bLambda_s}_h : \bLambda_s \to \bLambda_{sh}$ is the $\L^2$-projection operator, satisfying, $\forall \, \bphi\in \bL^2(\Gamma_{fp})$,
\begin{align}\label{eq: interpolation 0}
&\ds
\langle \bphi - P_h^{\bLambda_s}(\bphi), \bphi_h \rangle_{\Gamma_{fp}} = 0 \qquad \forall \, \bphi_h \in \bLambda_{sh}. 
\end{align}
It holds that
\begin{equation}\label{bound-theta-h0}
\|\btheta_{h,0}\|_{\bL^2(\Gamma_{fp})} \leq \|\btheta_{0}\|_{\bL^2(\Gamma_{fp})}.
\end{equation}

2. Define $(\bT_{fh,0},\bu_{fh,0},\bu_{ph,0}, p_{ph,0},\lambda_{h,0})\in \bbX_{fh} \times \bV_{fh}\times \bV_{ph}\times \W_{ph}\times \Lambda_{ph}$ associated to the problem
\begin{align}
& \ds a_f(\bT_{fh,0}, \bu_{fh,0};\bR_{fh}, \bv_{fh})
+ \kappa_{\bu_{fh,0}} (\bT_{fh,0}, \bu_{fh,0};\bR_{fh}, \bv_{fh}) \nonumber \\
& \ds  \quad  + \mu\,\alpha_{BJS}\sum_{j=1}^{n-1}\langle  \sqrt{\bK_j^{-1}} (\bu_{fh,0}-\btheta_{h,0})
\cdot \bt_{f,j}, \bv_{fh} \cdot \bt_{f,j} \rangle_{\Gamma_{fp}} + \langle \bv_{fh} \cdot \bn_f, \lambda_{h,0} \rangle_{\Gamma_{fp}} \nonumber \\
&  \ds = a_f(\bT_{f,0}, \bu_{f,0};\bR_{fh}, \bv_{fh})
+ \kappa_{\bu_{f,0}}(\bT_{f,0}, \bu_{f,0};\bR_{fh}, \bv_{fh}) \nonumber \\
& \ds \quad  + \mu\,\alpha_{BJS}\sum_{j=1}^{n-1}\langle  \sqrt{\bK_j^{-1}} (\bu_{f,0}-\btheta_{0})
\cdot \bt_{f,j}, \bv_{fh} \cdot \bt_{f,j} \rangle_{\Gamma_{fp}} + \langle \bv_{fh} \cdot \bn_f, \lambda_{0} \rangle_{\Gamma_{fp}} \label{eq: dis ini NS-D} \\
& \ds = (\f_f,\bv_{fh} - \kappa_1 \bdiv(\bR_{fh}))_{\Omega_f}, \nonumber \\[1ex]
& a_p(\bu_{ph,0}, \bv_{ph})+b_p(p_{ph,0},\bv_{ph})+  \langle \bv_{ph} \cdot \bn_p, \lambda_{h,0} \rangle_{\Gamma_{fp}}= a_p(\bu_{p,0}, \bv_{ph}) + b_p(p_{p,0},\bv_{ph})
+ \langle \bv_{ph} \cdot \bn_p, \lambda_0 \rangle_{\Gamma_{fp}} = 0, \nonumber\\[1ex]
& \ds -b_p(w_{ph}, \bu_{ph,0}) = -b_p(w_{ph},\bu_{p,0})
 = -\frac{1}{\mu}(\div (\bK \nabla p_{p,0}), w_{ph})_{\Omega_p}, \nonumber\\[1ex]
& -\langle  \bu_{fh,0} \cdot \bn_f + (
 \btheta_{h,0}+ \bu_{ph,0}) \cdot \bn_p, \xi_h \rangle_{\Gamma_{fp}}
 = -\langle \bu_{f,0} \cdot \bn_f
 + (\btheta_{0}+\bu_{p,0}) \cdot \bn_p, \xi_h \rangle_{\Gamma_{fp}} = 0. \nonumber
\end{align}
for all $\bR_{fh} \in \bbX_{fh}, \bv_{fh} \in \bV_{fh}$, $\bv_{ph} \in \bV_{ph}$, $w_p \in \W_{ph}$, $\xi_h \in \Lambda_{ph}$.
Notice that \eqref{eq: dis ini NS-D} is well-posed, since it corresponds to 
the weak solution of the augmented mixed formulation for the Navier--Stokes/Darcy coupled problem 
(see \cite{gov2020} for a similar approach). 
Note that $\btheta_{h,0}$ is datum for this problem.
The well-posedness of \eqref{eq: dis ini NS-D},
follows from a fixed point approach as in \eqref{eq:def-T} combined with 
the {\it a priori} estimate
\begin{equation}\label{eq:NS-discrete-initial-data}
\begin{split}
& \|\bT_{fh,0}\|_{\bbX_f} + \|\bu_{fh,0}\|_{\bV_f} + \|\bu_{ph,0}\|_{\bV_p} + \|p_{ph,0}\|_{\W_p} + \|\lambda_{h,0}\|_{\Lambda_{ph}} \\
& \qquad \leq C_{\wt{\cJ}_0} \, \left(\|\f_f\|_{\bL^2(\Omega_f)}
+ \|p_{p,0}\|_{\H^1(\Omega_p)} + \|\bK\nabla p_{p,0}\|_{\bH^1(\Omega_p)} \right)
\end{split}
\end{equation}
and the data assumption \eqref{eq:initial-small-data-discrete}.
In the above estimate we have used \eqref{bound-theta-h0} and 
\eqref{bound-theta0}. We note that \eqref{eq:initial-small-data-discrete} and \eqref{eq:NS-discrete-initial-data} imply that $\bu_{fh,0}\in \bW_r$.

3. Define $\ds (\bsi_{ph,0}, \bbeta_{ph,0}, \brho_{ph,0}, \bpsi_{h,0}) \in \bbX_{ph}
\times \bV_{sh} \times \bbQ_{ph} \times \bLambda_{sh}$ as the unique solution of the problem
\begin{align}
& \ds (A(\bsi_{ph,0}), \btau_{ph})_{\Omega_p}
+ b_s(\bbeta_{ph,0}, \btau_{ph})
+ b_{\sk}(\brho_{ph,0}, \btau_{ph})
- b_{\bn_p}(\btau_{ph}, \bpsi_{h,0})
+ (A (\alpha_p \, p_{ph,0}\, \bI),\btau_{ph})_{\Omega_p} \nonumber\\[1ex] 
& \ds =(A(\bsi_{p,0}), \btau_{ph})_{\Omega_p}
+ b_s( \bbeta_{p,0}, \btau_{ph})
+ b_{\sk}(\brho_{p,0}, \btau_{ph})
- b_{\bn_p}(\btau_{ph}, \bpsi_0)
+ (A (\alpha_p \, p_{p,0} \, \bI),\btau_{ph})_{\Omega_p} = 0, \nonumber \\[1ex] 
&\ds  - b_s (\bv_{sh},\bsi_{ph,0} ) = - b_s ( \bv_{sh}, \bsi_{p,0} )
= (\f_p,\bv_{sh})_{\Omega_p}, \nonumber \\[1ex]
&\ds  - b_{\sk} ( \bchi_{ph}, \bsi_{ph,0} ) = - b_{\sk} ( \bchi_{ph}, \bsi_{p,0} ) = 0,
\label{eq: dis ini sigma} \\
&   \ds b_{\bn_p}(\bsi_{ph,0}, \bphi_h)
- \mu\,\alpha_{BJS}\sum_{j=1}^{n-1}\langle  \sqrt{\bK_j^{-1}} (\bu_{fh,0}-\btheta_{h,0})
\cdot \bt_{f,j}, \bphi_h \cdot \bt_{f,j} \rangle_{\Gamma_{fp}}
+ \langle \bphi_h \cdot \bn_p, \lambda_{h,0} \rangle_{\Gamma_{fp}} \nonumber \\
& \ds = b_{\bn_p}(\bsi_{p,0}, \bphi_h)
- \mu\,\alpha_{BJS}\sum_{j=1}^{n-1}\langle  \sqrt{\bK_j^{-1}} (\bu_{f,0}-\btheta_{0})
\cdot \bt_{f,j}, \bphi_h \cdot \bt_{f,j} \rangle_{\Gamma_{fp}}
+ \langle \bphi_h \cdot \bn_p, \lambda_0 \rangle_{\Gamma_{fp}} = 0. \nonumber
\end{align}
for all
$\btau_{ph} \in \bbX_{ph}$, $\bv_{sh} \in \bV_{sh}$, $\bchi_{ph} \in \bbQ_{ph}$,
$\bphi_h \in \bLambda_{sh}$.
Note that \eqref{eq: dis ini sigma}  is a mixed elasticity system with mixed boundary conditions on $\Gamma_{fp}$ and its well-posedness follows from the
classical Babu{\v s}ka--Brezzi theory. Note also that $p_{ph,0}, \bu_{fh,0}, \btheta_{h,0}$, and $\lambda_{h,0}$ are data for this problem. It holds that
\begin{equation}\label{bound-sigmap-h0}
\|\bsi_{ph,0}\|_{\bbX_p} + \|\bbeta_{ph,0}\|_{\bV_s} + \|\brho_{ph,0}\|_{\bbQ_p} + \|\bpsi_{h,0}\|_{\bLambda_{sh}}
\le C \left(\|p_{p,0}\|_{\H^1(\Omega_p)} + \|\bK\nabla p_{p,0}\|_{\bH^1(\Omega_p)} + \|\f_p\|_{\bL^2(\Omega_p)} \right),
\end{equation}
where we have used \eqref{bound-theta-h0}, \eqref{bound-theta0}, and \eqref{eq:NS-discrete-initial-data}.

4. Finally, define $(\wh{\bsi}_{ph,0}, \bu_{sh,0}, \bgamma_{ph,0}) \in \bbX_{ph} \times \bV_{sh} \times \bbQ_{ph}$
as the unique solution of the problem 
\begin{align}
& \ds (A \wh\bsi_{ph,0}, \btau_{ph})_{\Omega_p}  + b_s( \btau_{ph}, \bu_{sh,0})
  + b_{\sk}(\btau_{ph}, \bgamma_{ph,0}) = b_{\bn_p}(\btau_{ph}, \btheta_{h,0}), \nonumber \\[1ex]
& \ds  -b_s ( \wh\bsi_{ph,0}, \bv_{sh} ) = 0, \label{eq: dis ini us} \\[1ex]
& \ds  - b_{\sk} ( \wh\bsi_{ph,0}, \bchi_{ph} ) = 0, \nonumber
\end{align}
for all $\btau_{ph} \in \bbX_{ph}$, $\bv_{sh} \in \bV_{sh}$, $\bchi_{ph} \in \bbQ_{ph}$.
Problem \eqref{eq: dis ini us} is well-posed 
as a direct application of the classical Babu{\v s}ka--Brezzi theory.
Note that $\btheta_{h,0}$ is datum for this problem. It holds that
\begin{equation}\label{bound-hat-sigmap-h0}
\|\wh\bsi_{ph,0}\|_{\bbX_p} + \|\bu_{sh,0}\|_{\bV_s} + \|\bgamma_{ph,0}\|_{\bbQ_p}
\le C \big(\|\f_f\|_{\bL^2(\Omega_f)}
+ \|p_{p,0}\|_{\H^1(\Omega_p)}
+ \|\bK\nabla p_{p,0}\|_{\bH^1(\Omega_p)} \big),
\end{equation}
where we have used that $b_{\bn_p}(\btau_{ph}, \btheta_{h,0}) = b_{\bn_p}(\btau_{ph}, \btheta_{0})$
(cf. \eqref{eq: interpolation 0} and \eqref{defn-Lambda-h}), \eqref{trace-sigma}, and \eqref{bound-theta0}.

We then define $\bp_{h,0} := (\bsi_{ph,0},p_{ph,0},\bu_{ph,0},\bT_{fh,0},\bu_{fh,0},\btheta_{h,0})$ and $\br_{h,0} := (\lambda_{h,0},\bu_{sh,0},\bgamma_{ph,0})$. 
The above construction implies that $(\bp_{h,0}, \br_{h,0})$ satisfy \eqref{eq: discrete initial condition} with
\begin{equation}\label{eq: discrete hat functions}
\begin{array}{c}
\ds (\wh{\f}_{ph,0}, \btau_{ph})_{\Omega_p}= a_e(\bsi_{ph,0}, p_{ph,0}; \btau_{ph}, 0)- (A(\wh{\bsi}_{ph,0}),\btau_{ph})_{\Omega_p},\\[2ex]
\ds (\wh{q}_{ph,0}, w_{ph})_{\Omega_p}= (s_0\,p_{ph,0},w_{ph})_{\Omega_p} + a_e(\bsi_{ph,0}, p_{ph,0}; \0, w_{ph})- b_p( \bu_{ph,0}, w_{ph}) \,. 
\end{array}
\end{equation}
From the stability bounds \eqref{eq:NS-discrete-initial-data}, \eqref{bound-sigmap-h0}, \eqref{bound-hat-sigmap-h0}, and \eqref{eq:extra-pp0-assumption}, we obtain
\begin{align}
\|\wh{\f}_{ph,0}\|_{\bbL^2(\Omega_p)} 
+ \|\wh{q}_{ph,0}\|_{\L^2(\Omega_p)}
\,& \leq\, C\, \big(\|\f_f\|_{\bL^2(\Omega_f)} + \|\f_p\|_{\bL^2(\Omega_p)}
+ \|p_{p,0}\|_{\H^1(\Omega_p)}
+ \|\bK\nabla p_{p,0}\|_{\bH^1(\Omega_p)} \big) \nonumber \\
& \le \, \wh{C}_{ep,h}
    \big(\|\f_f\|_{\bL^2(\Omega_f)} + \|\f_p\|_{\bL^2(\Omega_p)} \big),
\label{bound-hat-f-q}
\end{align}  
hence $(\wh{\f}_{ph,0}, \wh{q}_{ph,0})\in E'_b$ and \eqref{eq:initial-data-bound-h} holds.
\end{proof}

\begin{rem}
The above construction provides compatible initial data for the non-differentiated elasticity variables $(\bbeta_{ph,0},\brho_{ph,0},\bpsi_{h,0})$ in the sense of the first equation in \eqref{eq: dis ini sigma}.
\end{rem}

Now, we establish the well-posedness of problem \eqref{eq: NS-Biot-semiformulation-5} 
and the corresponding stability bound.
\begin{thm}\label{thm: well-posedness main result semi}
Assume that the conditions of Lemma~\ref{lem:monotone} are satisfied.
Then, for each
\begin{equation*}
\f_f\in \bL^2(\Omega_f),\quad \f_p\in \bL^2(\Omega_p),
\quad q_p\in \W^{1,1}(0,T;\L^2(\Omega_p)), \quad p_{p,0}\in \H_p \ (cf. \, \eqref{eq:H-definition}),
\end{equation*}
satisfying \eqref{eq:Ct-less-than-r}, 
\eqref{eq:extra-pp0-assumption}, \eqref{eq:initial-small-data}, and \eqref{eq:initial-small-data-discrete},
and for each compatible discrete initial data $(\bp_{h,0}, \br_{h,0})$ constructed in Lemma \ref{lem: discrete initial condition}, there exists a unique solution of \eqref{eq: NS-Biot-semiformulation-5}, $(\bp_h, \br_h): [0,T] \rightarrow \bQ_h \times \bS_h$ with $\bu_{fh}(t)\in \bW_r$ (cf. \eqref{eq:W_r-definition}), $(\bsi_{ph}, p_{ph})\in \W^{1,\infty}(0,T;\bbX_{ph}) \times \W^{1,\infty}(0,T;\W_{ph})$ and $(\bsi_{ph}(0), p_{ph}(0), \bu_{ph}(0), \bT_{fh}(0), \linebreak \bu_{fh}(0), \btheta_h(0), \lambda_h(0)) = (\bsi_{ph,0}, p_{ph,0}, \bu_{ph,0}, \bT_{fh,0}, \bu_{fh,0}, \btheta_{h,0}, \lambda_{h,0})$. Moreover, if 
$q_p \in \H^1(0,T;\L^2(\Omega_p))$, there exists a positive constant $C$, independent of $h$ and $s_{0,\min}$, such that
\begin{align}\label{eq: discrete stability}
&\ds \|A^{1/2} (\bsi_{ph} + \alpha_p p_{ph} \bI)\|_{\W^{1,\infty}(0,T;\bbL^2(\Omega_p))}
+ \|\bdiv(\bsi_{ph})\|_{\L^\infty(0,T; \bbL^2(\Omega_p))}
+ \|\bdiv(\bsi_{ph})\|_{\L^2(0,T; \bbL^2(\Omega_p))}
 \nonumber \\[1ex]
& \ds\quad + \,  
\sqrt{s_0} \|p_{ph}\|_{\W^{1,\infty}(0,T;\W_p)}
+ \|p_{ph}\|_{\H^1(0,T;\W_p)} 
+ \|\bu_{ph}\|_{\L^2(0,T;\bV_p)} 
+ \|\partial_t\,\bu_{ph}\|_{\L^2(0,T;\bL^2(\Omega_p))} 
+ \|\bT_{fh} \|_{\H^1(0,T;\bbX_f)} 
 \nonumber \\[1ex]
&\ds\quad + \, 
\|\bu_{fh}\|_{\H^1(0,T;\bV_f)}
+ |\bu_{fh} - \btheta_h|_{\H^1(0,T;\BJS)} 
+ \|\btheta_h\|_{\L^\infty(0,T;\bLambda_{sh})} 
+ \|\btheta_h\|_{\L^2(0,T;\bLambda_{sh})} 
+ \|\lambda_h\|_{\H^1(0,T;\Lambda_{ph})}  
\nonumber \\[1ex]
&\ds\quad
+ \|\bu_{sh}\|_{\L^\infty(0,T;\bV_s)}
+ \|\bu_{sh}\|_{\L^2(0,T;\bV_s)}  
+ \|\bgamma_{ph}\|_{\L^\infty(0,T;\bbQ_p)}
+ \|\bgamma_{ph}\|_{\L^2(0,T;\bbQ_p)} \nonumber \\[1ex]
&\ds \leq\, C\,\bigg(
\|\f_f\|_{\bL^2(\Omega_f)} 
+ \|\f_p\|_{\bL^2(\Omega_p)}
+ \|q_p\|_{\H^1(0,T;\L^2(\Omega_p))}
+ \frac{1}{\sqrt{s_0}}\,\|q_p(0)\|_{\L^2(\Omega_p)}
\nonumber
\\[1ex]
& \ds\quad  + \, \sqrt{s_0}\,\|p_{p,0}\|_{\W_p}
+ \left(1 + \frac{1}{\sqrt{s_0}}\right)\left( \|p_{p,0}\|_{\H^1(\Omega_p)} + \|\bK \nabla p_{p,0}\|_{\bH^1(\Omega_p)} \right) \bigg)\,.
\end{align}
\end{thm}
\begin{proof}
With the discrete inf-sup conditions \eqref{eq:discrete-inf-sup-1}--\eqref{eq:discrete-inf-sup-2}
and the discrete initial data construction described in
\eqref{eq: interpolation 0}--\eqref{eq: dis ini us},
the proof is similar to the proofs of Theorem~\ref{thm:well-posedness-main-result} and Theorem~\ref{thm: continuous stability}, with several differences.
First, due to non-conforming choices of the Lagrange multiplier spaces equipped with $\L^2$-norms, the operators $L_{\lambda}$ and
$R_{\btheta}$ from Lemma \ref{lem:R-operators2} are now defined as
$L_{\lambda}: \Lambda_{ph} \rightarrow \Lambda_{ph}'$, $\ L_{\lambda}(\lambda_h)(\xi_h) := \langle \lambda_h, \xi_h \rangle_{\Gamma_{fp}}$ and
$R_{\btheta}: \bLambda_{sh} \rightarrow \bLambda_{sh}'$,
$R_{\btheta}(\btheta_h)(\bphi_h) := \langle \btheta_h, \bphi_h
\rangle_{\Gamma_{fp}}$. The fact that $L_{\lambda}$ and
$R_{\btheta}$ are continuous and coercive follows
immediately from their definitions, since $L_{\lambda}(\xi_h)(\xi_h) = \| \xi \|^2_{\Lambda_{ph}}$ and $R_{\btheta}(\bphi_h)(\bphi_h) = \| \bphi_h \|^2_{\bLambda_{sh}}$.
Second, the proof in Theorem~\ref{thm:well-posedness-main-result}
that the solution at $t = 0$ equals the initial data 
works in the discrete case due to the choice of the discrete initial data
as the elliptic projection of the continuous initial data, cf. \eqref{eq: dis ini NS-D} and \eqref{eq: dis ini sigma}. Third, the control of $p_{ph}$, $\btheta_h$, $\lambda_h$, $\bu_{sh}$, and $\bgamma_{ph}$ follows from the discrete inf-sup conditions \eqref{eq:discrete-inf-sup-1} and \eqref{eq:discrete-inf-sup-2}. Fourth, the discrete version of the initial data bound \eqref{eq:continuous-stability-initial-1} follows from \eqref{eq:NS-discrete-initial-data} and \eqref{bound-sigmap-h0}. 
Finally, in the discrete version of \eqref{eq:continuous-stability-8} we apply the orthogonality property \eqref{eq: interpolation 0} to deduce that $\langle \partial_t \bsi_{ph}(0)\bn_p, \btheta_h(0)\rangle_{\Gamma_{fp}} = \langle \partial_t \bsi_{ph}(0)\bn_p, \btheta(0)\rangle_{\Gamma_{fp}}$ and then the proof continues as in the continuous case, using the normal trace inequality \eqref{trace-sigma}.
\end{proof}

\begin{rem}
As in the continuous case, we can recover the non-differentiated elasticity variables with
$$
\forall \, t \in [0,T], \quad  \bbeta_{ph}(t) = \bbeta_{ph,0} + \int_0^t \bu_{sh}(s) \, ds,
\ \  \brho_{ph}(t) = \brho_{ph,0} + \int_0^t \bgamma_{ph}(s) \, ds, \ \ 
\bpsi_h(t) = \bpsi_{h,0} + \int_0^t \btheta_h(s) \, ds.
$$
Then \eqref{non-diff-eq} holds discretely, which follows from integrating the equation associated to $\btau_{ph}$ in \eqref{eq: NS-Biot-semiformulation-5} from $0$ to $t \in (0,T]$ and using the first equation in \eqref{eq: dis ini sigma}.
\end{rem}

\subsection{Error analysis}

We proceed with establishing rates of convergence. Let the polynomial degrees in $\bbX_{ph}\times \W_{ph} \times \bV_{ph} \times \bbX_{fh} \times \bV_{fh} \times \bLambda_{sh}
\times \Lambda_{ph} \times \bV_{sh}\times \bbQ_{ph}$ be, respectively,
$s_{\bsi_p},
s_{p_p},
s_{\bu_p},
s_{\bT_f},
s_{\bu_f}
s_{\btheta},
s_{\lambda},
s_{\bu_s},
s_{\bgamma_p}$.
Let us set $\V\in \{ \W_p, \bV_s, \bbQ_p \}$, $\Lambda\in \{ \bLambda_s, \Lambda_p \}$ and let $\V_h, \Lambda_h$ be the discrete counterparts. Let $P_h^{\V}: \V\to \V_h$ and $P_h^{\Lambda}: \Lambda\to \Lambda_h$ be the $\L^2$-projection operators, satisfying
\begin{equation}\label{eq: interpolation 1}
\ds 
( u - P_h^{\V}(u), v_{h} )_{\Omega_p} \,=\, 0  \quad \forall \, v_h\in \V_{h},\quad 
\langle \theta - P_h^{\Lambda}(\theta), \phi_h \rangle_{\Gamma_{fp}} \,=\, 0 \quad \forall\, \phi_h\in \Lambda_{h},
\end{equation}
where $u\in \{ p_p, \bu_s, \bgamma_p \}$, $\theta \in \{ \btheta, \lambda \}$, and $v_h, \phi_h$ are the corresponding discrete test
functions.  We have the approximation properties \cite{ciarlet1978}:
\begin{equation}\label{eq: approx property 1}
\|u - P^{\V}_h (u)\|_{\L^2(\Omega_p)} \,\leq\, Ch^{s_{u} + 1}\, \|u\|_{\H^{s_{u} + 1}(\Omega_p)},\quad
\|\theta - P^{\Lambda}_h(\theta)\|_{\Lambda_h} \,\leq\, Ch^{s_{\theta} + 1}\,
\|\theta\|_{\H^{s_{\theta} + 1}(\Gamma_{fp})},
\end{equation}
where $s_{u}\in \{ s_{p_p}, s_{\bu_s}, s_{\bgamma_p} \}$ and $s_{\theta}\in \{ 
s_{\btheta}, s_{\lambda} \}$.

Since the discrete Lagrange multiplier spaces are chosen as
$\bLambda_{sh} = \bbX_{ph} \bn_p |_{\Gamma_{fp}}$ and $\Lambda_{ph} = \bV_{ph} \cdot \bn_p |_{\Gamma_{fp}}$, respectively, we have
\begin{equation}\label{eq: interpolation 2}
\langle \btheta - P_h^{\bLambda_s}(\btheta), \btau_{ph}\bn_p \rangle_{\Gamma_{fp}} \,=\, 0 \quad \forall\, \btau_{ph}\in \bbX_{ph},\quad
\langle \lambda - P_h^{\Lambda_p}(\lambda), \bv_{ph} \cdot \bn_p \rangle_{\Gamma_{fp}} \,=\, 0 \quad \forall\, \bv_{ph}\in \bV_{ph}.
\end{equation}

Next, denote $\X\in \{\bbX_f, \bbX_p, \bV_p \}$, $\sigma \in \{ \bT_f, \bsi_p, \bu_p \}\in \X$ and let $\X_h, \tau_h$ be their discrete counterparts. Let $I^{\X}_h : \X \cap \H^{1}(\Omega_{\star})\to \X_{h}$ be the mixed finite element
projection operator \cite{Brezzi-Fortin} satisfying
\begin{equation}\label{eq: interpolation 3}
(\div(I^{\X}_h \sigma), w_h) = (\div(\sigma), w_h) \quad \forall\,w_h\in \W_h,\quad
\pil I^{\X}_h(\sigma)\bn_{\star}, \tau_h\bn_{\star} \pir_{\Gamma_{fp}} = \pil \sigma \bn_{\star}, \tau_{h}\bn_{\star} \pir_{\Gamma_{fp}} \quad \forall\,\tau_h\in \X_h, 
\end{equation}
and
\begin{equation}\label{eq: approx property 2}
\|\sigma - I^{\X}_h(\sigma)\|_{\L^2(\Omega_{\star})} \leq C\,h^{s_{\sigma} + 1} \| \sigma \|_{\H^{s_{\sigma} + 1}(\Omega_{\star})} ,\quad
\|\div(\sigma - I^{\X}_h(\sigma))\|_{\L^2(\Omega_{\star})} \leq C\,h^{s_{\sigma} + 1} \|\div(\sigma)\|_{\H^{s_{\sigma} + 1}(\Omega_{\star})},
\end{equation}
where $\star\in \{f,p\}$, $w_h\in \big\{ \bv_{fh}, \bv_{sh}, w_{ph} \big\}$, $\W_h\in
\big\{ \bV_f, \bV_s, \W_p \big\}$, and $s_{\sigma}\in \big\{
s_{\bT_f}, s_{\bu_p}, s_{\bsi_p} \big\}$.

Finally, let $S_h^{\bV_f}$ be the Scott-Zhang interpolation operators onto $\bV_{fh}$, satisfying \cite{sz1990}
\begin{equation}\label{eq: approx property 3}
\|\bv_f - S^{\bV_f}_h(\bv_f)\|_{\bH^1(\Omega_f)} \,\leq\, C\,h^{s_{\bu_f}} \| \bv_f \|_{\bH^{s_{\bu_f} + 1}(\Omega_f)}.
\end{equation}

Now, let $(\bsi_p, p_p, \bu_p, \bT_f, \bu_f,  \btheta, \lambda, \bu_s, \bgamma_p)$ and $(\bsi_{ph}, p_{ph}, \bu_{ph}, \bT_{fh}, \bu_{fh},  \btheta_h, \lambda_h, \bu_{sh}, \bgamma_{ph})$ be the solutions of
\eqref{eq:continuous-weak-formulation-3} and
\eqref{eq: NS-Biot-semiformulation-5}, respectively.  We
introduce the error terms as the difference of these two solutions and
decompose them into approximation and discretization errors
using the interpolation operators:
\begin{equation}
\begin{array}{l}
\ds \be_{\sigma} := \sigma - \sigma_h = (\sigma - I_h^{\X} (\sigma)) + ( I_h^{\X} (\sigma) - \sigma_h) := \be^I_{\sigma} + \be^h_{\sigma}, \quad \sigma \in \{ \bT_f, \bsi_p, \bu_p \},  \\[2ex]
\ds \be_{u}  :=  u - u_{h} = (u - P_h^{\V} (u)) + ( P_h^{\V} (u) - u_{h}) := \be^I_{u} + \be^h_{u}, \quad u \in \{ p_p, \bu_s, \bgamma_p \},\\[2ex]
\ds \be_{\theta} := \theta - \theta_h = (\theta - P_h^{\Lambda} (\theta)) + ( P_h^{\Lambda} (\theta) - \theta_h) := \be^I_{\theta} + \be^h_{\theta}, \quad \theta \in \{\btheta, \lambda\}, \\[2ex]
\ds \be_{\bu_f}  :=  \bu_f - \bu_{fh}  =  (\bu_f - S_h^{\bV_f} \bu_f) + ( S_h^{\bV_f} \bu_f - \bu_{fh})  :=  \be^I_{\bu_f} + \be^h_{\bu_f}.
\end{array}
\end{equation}

Then, we set the global errors endowed with the above decomposition:
\begin{equation*}
\be_{\bp} := ( \be_{\bsi_p}, \be_{p_p}, \be_{\bu_p}, \be_{\bT_f}, \be_{\bu_f}, \be_{\btheta} )\,,\quad 
\be_{\br} := ( \be_{\lambda}, \be_{\bu_s}, \be_{\bgamma_p} ).
\end{equation*}

We form the error equation by subtracting the discrete equations \eqref{eq: NS-Biot-semiformulation-5} from the continuous one \eqref{eq:continuous-weak-formulation-3}:
\begin{equation}\label{eq: error equation}
\begin{array}{rlll}
\ds \partial_t\,\cE(\be_{\bp})(\bq_h) + \cA(\be_{\bp})(\bq_h) + \cK_{\bu_{f}}(\bp)(\bq_h) - \cK_{\bu_{fh}}(\bp_h)(\bq_h) + \cB'(\be_{\br})(\bq_h) & = & 0 \ &\forall \,\bq_h \in \bQ_h, \\ [2ex]
\ds  -\,\cB(e_{\bp})(\bs_h) & = & 0 \ &\forall \,\bs_h \in \bS_h.
\end{array}
\end{equation}

We now establish the main result of this section.

\begin{thm}\label{thm: error analysis}
Let the assumptions in Theorem~\ref{thm: well-posedness main result semi} holds. For the solutions of the continuous and semidiscrete problems \eqref{eq:continuous-weak-formulation-3} and \eqref{eq: NS-Biot-semiformulation-5}, respectively, assuming  sufficient regularity of the true solution according to \eqref{eq: approx property 1}, \eqref{eq: approx property 2} and \eqref{eq: approx property 3}, 
there exists a positive constant $C$ depending on the solution regularity, but independent of $h$, such that
\begin{align}\label{eq:errror-rate-of-convergence}
&\ds \|A^{1/2}(\be_{\bsi_p} + \alpha_p\,\be_{p_p}\bI) \|_{\W^{1,\infty}(0,T;\bbL^2(\Omega_p))}
+ \| \bdiv(\be_{\bsi_p}) \|_{\L^\infty(0,T; \bbL^2(\Omega_p))}
+ \| \bdiv(\be_{\bsi_p}) \|_{\L^2(0,T; \bbL^2(\Omega_p))}
\nonumber \\[1ex]
&\ds\quad 
+\, \sqrt{s_0}\|\be_{p_p}\|_{\W^{1,\infty}(0,T; \W_p)}
+ \|\be_{p_p} \|_{\H^1(0,T; \W_p)}
+\, \|\be_{\bu_p} \|_{\L^2(0,T;\bV_p)}
+ \| \partial_t \be_{\bu_p} \|_{\L^2(0,T; \bL^2(\Omega_p))}
+ \|\be_{\bT_f}\|_{\H^1(0,T;\bbX_f)}
\nonumber \\[1ex]
&\ds\quad  
+\, \|\be_{\bu_f}\|_{\H^1(0,T;\bV_f)}
+ |\be_{\bu_f} - \be_{\btheta} |_{\H^1(0,T;\BJS)}
+ \|\be_{\btheta} \|_{\L^\infty(0,T;\bLambda_{sh})}
+ \|\be_{\btheta} \|_{\L^2(0,T;\bLambda_{sh})}
+ \|\be_{\lambda} \|_{\H^1(0,T;\Lambda_{ph})}
\nonumber \\[1ex]
&\ds\quad 
+ \| \be_{\bu_s} \|_{\L^\infty(0,T; \bV_s)}
+ \| \be_{\bu_s} \|_{\L^2(0,T; \bV_s)}
+ \| \be_{\bgamma_p} \|_{\L^\infty(0,T; \bbQ_p)} 
+ \| \be_{\bgamma_p} \|_{\L^2(0,T; \bbQ_p)}\nonumber \\[1ex]
&\leq\, C\,\sqrt{\exp(T)}\,\left( h^{s_{\underline{u}}+1} +  h^{s_{\underline{\theta}}+1} +  h^{s_{\underline{\sigma}}+1} +  h^{s_{\bu_f}} \right),
\end{align}
where $s_{\underline{u}}= \min{ \{ s_{p_p}, s_{\bu_s},s_{\bgamma_p} \}}$, $s_{\underline{\theta}}= \min{ \{ s_{\btheta}, s_{\lambda}, \}}$, and $s_{\underline{\sigma}}= \min{ \{ s_{\bT_f}, s_{\bsi_p}, s_{\bu_p}\}}$.
\end{thm}
\begin{proof}
We start by taking $\bq_h=(\be_{\bsi_p}^h,\be_{p_p}^h,\be_{\bu_p}^h,\be_{\bT_f}^h,\be_{\bu_f}^h,  \be_{\btheta}^h)$ and $\bs_h=(\be_{\lambda}^h,\be_{\bu_s}^h,\be_{\bgamma_p}^h)$ in \eqref{eq: error equation}, to obtain
\begin{align}\label{eq: error equation 1}
&\ds \frac{1}{2} s_0\, \partial_t\,(\be_{p_p}^h,\be_{p_p}^h)_{\Omega_p} 
+ \frac{1}{2} \partial_t \, a_e(\be_{\bsi_p}^h, \be_{p_p}^h; \be_{\bsi_p}^h, \be_{p_p}^h) 
+ a_p(\be_{\bu_p}^h,\be_{\bu_p}^h) 
+ a_f(\be_{\bT_f}^h,\be_{\bu_f}^h;\be_{\bT_f}^h,\be_{\bu_f}^h) \nonumber \\[1ex]
& \ds\quad +\, \kappa_{\bu_{fh}}(\be_{\bT_f}^h,\be_{\bu_f}^h;\be_{\bT_f}^h,\be_{\bu_f}^h) 
+ \kappa_{\be_{\bu_f}^h}(\bT_f,\bu_f;\be_{\bT_f}^h,\be_{\bu_f}^h) 
+ a_{\BJS}(\be_{\bu_f}^h,\be_{\btheta}^h;\be_{\bu_f}^h,\be_{\btheta}^h) \nonumber \\[1ex]
& \ds 
= - \, a_e( \partial_t\,\be_{\bsi_p}^I, \partial_t\,\be_{p_p}^I;\be_{\bsi_p}^h,\be_{p_p}^h) 
- a_p(\be_{\bu_p}^I,\be_{\bu_p}^h) 
- a_f(\be_{\bT_f}^I,\be_{\bu_f}^I,\be_{\bT_f}^h,\be_{\bu_f}^h)  \nonumber  \\[1ex]
&\ds \quad - \, \kappa_{\bu_{fh}}(\be_{\bT_f}^I,\be_{\bu_f}^I;\be_{\bT_f}^h,\be_{\bu_f}^h) 
- \kappa_{\be_{\bu_f}^I}(\bT_f,\bu_f;\be_{\bT_f}^h,\be_{\bu_f}^h)
- a_{\BJS}(\be_{\bu_f}^I,\be_{\btheta}^I;\be_{\bu_f}^h,\be_{\btheta}^h)
- b_\sk(\be_{\bgamma_p}^I,\be_{\bsi_p}^h) \nonumber  \\[1ex]
&\ds \quad 
+ b_\sk(\be_{\bgamma_p}^h,\be_{\bsi_p}^I) 
- \langle \be_{\bu_f}^h \cdot \bn_f, \be_{\lambda}^I \rangle_{\Gamma_{fp}} 
+ \langle \be_{\bu_f}^I \cdot \bn_f, \be_{\lambda}^h \rangle_{\Gamma_{fp}} 
- \langle \be_{\btheta}^h \cdot \bn_p, \be_{\lambda}^I \rangle_{\Gamma_{fp}} 
+ \langle \be_{\btheta}^I \cdot \bn_p, \be_{\lambda}^h \rangle_{\Gamma_{fp}} \,,
\end{align}
where, the right-hand side of \eqref{eq: error equation 1} has been simplified, 
since the projection properties \eqref{eq: interpolation 1} and 
\eqref{eq: interpolation 2}--\eqref{eq: interpolation 3}, and the fact that 
$\bdiv(\bbX_{ph})=\bV_{sh}$, $\div(\bV_{ph})=\W_{ph}$ (cf. \eqref{eq: div-prop}), 
imply that the following terms are zero:
\begin{equation*}
\begin{array}{c}
s_0(\partial_t\,\be_{p_p}^I,\be_{p_p}^h)_{\Omega_p}, \,\, b_p(\be_{p_p}^I,\be_{\bu_p}^h), \,\, b_p(\be_{p_p}^h,\be_{\bu_p}^I), \,\, b_{\bn_p}(\be_{\bsi_p}^I,\be_{\btheta}^h), \,\, b_{\bn_p}(\be_{\bsi_p}^h,\be_{\btheta}^I)\,,  \\[2ex]
b_s(\be_{\bu_s}^I,\be_{\bsi_p}^h), \,\, b_s(\be_{\bu_s}^h,\be_{\bsi_p}^I), \,\, 
\langle \be_{\bu_p}^h \cdot \bn_p, \be_{\lambda}^I \rangle_{\Gamma_{fp}}, \,\, 
\langle \be_{\bu_p}^I \cdot \bn_p, \be_{\lambda}^h \rangle_{\Gamma_{fp}}\,.
\end{array}
\end{equation*}
Then, using the fact that $\bu_f(t), \bu_{fh}(t): [0,T] \rightarrow \bW_r$, cf. \eqref{eq:W_r-definition}, the positivity estimates \eqref{eq:positive-bound-E}--\eqref{eq:positive-bound-A+Kwf} and the continuity bounds \eqref{eq:continuity-bounds}--\eqref{eq:continuity-of-Kwf} of the bilinear forms involved, and the Cauchy--Schwarz and Young's inequalities, we get
\begin{align}\label{eq: error equation 1+}
&\ds \frac{1}{2} s_0 \partial_t \|\be_{p_p}^h\|^2_{\W_p} 
+ \frac{1}{2} \partial_t \|A^{1/2}(\be_{\bsi_p}^h + \alpha_p \be_{p_p}^h \bI)\|^2_{\bbL^2(\Omega_p)} 
\nonumber \\[1ex]
&\ds\quad + \mu k_{\max}^{-1} \|\be_{\bu_p}^h\|^2_{\bL^2(\Omega_p)}
+ 2\,C_{\cK}(r_0 - r)\|(\be_{\bT_f}^h, \be_{\bu_f}^h)\|^2  
+ c_{\BJS}\,|\be_{\bu_f}^h - \be_{\btheta}^h|^2_{\BJS} \nonumber\\[1ex]
&\ds 
\leq C\,\Big( \|\be_{\bsi_p}^I\|^2_{\bbL^2(\Omega_p)} + \|\partial_t A^{1/2}(\be_{\bsi_p}^I + \alpha_p \be_{p_p}^I \bI)\|^2_{\bbL^2(\Omega_p)}
+ \|\be_{\bu_p}^I\|^2_{\bL^2(\Omega_p)}
+ (1 + 2\,r )\|(\be_{\bT_f}^I, \be_{\bu_f}^I)\|^2 
\nonumber \\[1ex]
&\ds \quad
+ |\be_{\bu_f}^I - \be_{\btheta}^I|^2_{\BJS}
+ \|\be_{\btheta}^I\|^2_{\bLambda_{sh}}
+ \|\be_{\lambda}^I\|^2_{\Lambda_{ph}}
+ \|\be_{\bgamma_p}^I\|^2_{\bbQ_p}
\Big)
+ \delta_1\,\left( 
\|\be_{\bu_p}^h\|^2_{\bL^2(\Omega_p)} 
+ \|(\be_{\bT_f}^h, \be_{\bu_f}^h)\|^2
+ |\be_{\bu_f}^h - \be_{\btheta}^h|^2_{\BJS}
\right) \nonumber \\[1ex]
&\ds \quad 
+ \delta_2\,\left( \|A^{1/2}(\be_{\bsi_p}^h + \alpha_p \be_{p_p}^h \bI)\|^2_{\bbL^2(\Omega_p)}
+ \|\be_{p_p}^h\|^2_{\W_p} 
+ \|\be_{\btheta}^h\|^2_{\bLambda_{sh}}
+ \|\be_{\lambda}^h\|^2_{\Lambda_{ph}}
+ \|\be_{\bgamma_p}^h\|^2_{\bbQ_p} 
\right),
\end{align}
where we have used the estimate
\begin{equation}\label{eq: bsk bound}
\ds  b_{\sk}(\be_{\bgamma_p}^I,\be_{\bsi_p}^h)
\,=\, \frac{1}{c}(A^{1/2} \be_{\bgamma_p}^I, A^{1/2} \be_{\bsi_p}^h)_{\Omega_p}
\,\leq\, C\,\|\be_{\bgamma_p}^I\|_{\bbQ_p} \left( \|A^{1/2}(\be_{\bsi_p}^h + \alpha_p \be_{p_p}^h \bI)\|_{\bbL^2(\Omega_p)} + \|\be_{p_p}^h\|_{\W_p} \right)\,, 
\end{equation}
which, follows from the definition of $A$ due to the extension from $\bbS$ to $\bbM$ as in \cite{lee2016}.
Next, we choose $r=r_0/2$ and $\ds \delta_1$ small enough in \eqref{eq: error equation 1+} and
integrate from $0$ to $t\in (0,T]$ to find
\begin{align}\label{eq: error equation 2}
&\ds s_0 \| \be_{p_p}^h(t)\|^2_{\W_p} 
+ \| A^{1/2}(\be_{\bsi_p}^h + \alpha_p \be_{p_p}^h \bI)(t)  \|^2_{\bbL^2(\Omega_p)} 
+ \int_0^t \left( \|\be_{\bu_p}^h\|^2_{\bL^2(\Omega_p)} + \|(\be_{\bT_f}^h,\be_{\bu_f}^h)\|^2 
+ |\be_{\bu_f}^h - \be_{\btheta}^h |^2_{\BJS} \right) ds \nonumber \\[1ex]
&\ds 
\leq\, C\, \int_0^t \left( \|\be_{\bsi_p}^I\|^2_{\bbL^2(\Omega_p)}
+ \|\partial_t A^{1/2}(\be_{\bsi_p}^I + \alpha_p\be_{p_p}^I \bI)\|^2_{\bbL^2(\Omega_p)}
+ \|\be_{\bu_p}^I\|^2_{\bL^2(\Omega_p)}
+ \|(\be_{\bT_f}^I,\be_{\bu_f}^I)\|^2
+ \|\be_{\btheta}^I\|^2_{\bLambda_{sh}} \right.
\nonumber \\[1ex]
&\ds\quad + \left.
|\be_{\bu_f}^I - \be_{\btheta}^I|^2_{\BJS} 
+ \|\be_{\lambda}^I\|^2_{\Lambda_{ph}}
+ \|\be_{\bgamma_p}^I\|^2_{\bbQ_p} \right)ds 
+ s_0 \|\be_{p_p}^h(0)\|^2_{\W_p} 
+ \| A^{1/2}(\be_{\bsi_p}^h + \alpha_p \be_{p_p}^h \bI)(0)  \|^2_{\bbL^2(\Omega_p)}
\nonumber \\[1ex]
&\ds\quad
+\, \delta_2 \int_0^t \left( \|A^{1/2}( \be_{\bsi_p}^h + \alpha_p \be_{p_p}^h \bI)\|^2_{\bbL^2(\Omega_p)} 
+ \|\be_{p_p}^h\|^2_{\W_p}
+ \|\be_{\btheta}^h\|^2_{\bLambda_{sh}}
+ \|\be_{\lambda}^h\|^2_{\Lambda_{ph}}
+ \|\be_{\bgamma_p}^h\|^2_{\bbQ_p} 
\right)ds.
\end{align}

On the other hand, from discrete inf-sup conditions \eqref{eq:discrete-inf-sup-1}--\eqref{eq:discrete-inf-sup-2} in Lemma \ref{lem: discrete inf-sup} and the first equation in \eqref{eq: error equation}, we have
\begin{equation}\label{eq: error equation 3}
\begin{array}{c}
\ds \|\be_{\btheta}^h\|_{\bLambda_{sh}}
+ \|\be_{\bu_s}^h\|_{\bV_{s}} 
+ \|\be_{\bgamma_p}^h\|_{\bbQ_{p}} 
\,\leq\,  C\,\left( \|\partial_t A^{1/2}(\be_{\bsi_p} + \alpha_p\,\be_{p_p} \bI)\|_{\bbL^2(\Omega_p)} 
+ \|\be_{\bgamma_p}^I\|_{\bbQ_p} \right),  \\[2ex]
\ds \qan \|\be_{p_p}^h\|_{\W_p} + \|\be_\lambda^h\|_{\Lambda_{ph}} 
\,\leq\, C\, \left( \|\be_{\bu_p}^I\|_{\bL^2(\Omega_p)}
+ \|\be_{\bu_p}^h\|_{\bL^2(\Omega_p)} \right), 
\end{array}	
\end{equation}
whereas, using \eqref{eq: div-prop}, and taking $w_{ph}=\div(\be_{\bu_p}^h)$ and
$\bv_{sh}=\bdiv(\be_{\bsi_p}^h)$ in \eqref{eq: error equation}, we obtain, respectively
\begin{equation}\label{eq: bound for error div sigma}
\begin{array}{c}
\ds \|\div(\be_{\bu_p}^h)\|_{\L^2(\Omega_p)} 
\,\leq\, 
C\,\left( \|\partial_t\,A^{1/2}(\be_{\bsi_p} + \alpha_p\,\be_{p_p}\bI)\|_{\bbL^2(\Omega_p)}
+ \sqrt{s_0}\,\|\partial_t \be_{p_p}^h\|_{\W_p} \right), \\[2ex]
\ds \qan \|\bdiv(\be_{\bsi_p}^h)\|_{\bL^2(\Omega_p)}\,=\, 0. 
\end{array}
\end{equation}
Thus, combining \eqref{eq: error equation 2} with \eqref{eq: error equation 3}--\eqref{eq: bound for error div sigma}, and choosing $\delta_2$ small enough, we get
\begin{align}\label{eq: error equation 6-}
&\ds s_0 \|\be_{p_p}^h(t)\|^2_{\W_p} 
+ \|A^{1/2}(\be_{\bsi_p}^h + \alpha_p \be_{p_p}^h \bI)(t)\|^2_{\bbL^2(\Omega_p)} 
+ \|\bdiv(\be_{\bsi_p}^h(t))\|_{\bL^2(\Omega_p)} 
+ \int_0^t \Big( 
\|\be_{p_p}^h\|^2_{\W_{p}} + \|\bdiv(\be_{\bsi_p}^h)\|_{\bL^2(\Omega_p)} \nonumber  \\[1ex]
&\ds\quad + \|\be_{\bu_p}^h\|^2_{\bV_p} + \|(\be_{\bT_f}^h,\be_{\bu_f}^h)\|^2 
+ |\be_{\bu_f}^h - \be_{\btheta}^h|^2_{\BJS}
+ \|\be_{\btheta}^h\|^2_{\bLambda_{sh}} + \|\be_{\lambda}^h\|^2_{\Lambda_{ph}} 
+ \|\be_{\bu_s}^h\|^2_{\bV_{s}} + \|\be_{\bgamma_p}^h\|^2_{\bbQ_{p}}
\Big) ds \nonumber \\[1ex]
&\ds 
\leq\, C\,\bigg( \int_0^t \Big( \|\be_{\bsi_p}^I\|^2_{\bbL^2(\Omega_p)}
+ \|\partial_t A^{1/2}(\be_{\bsi_p}^I + \alpha_p \be_{p_p}^I \bI) \|^2_{\bbL^2(\Omega_p)}
+ \|\be_{\bu_p}^I\|^2_{\bL^2(\Omega_p)}
+ \|(\be_{\bT_f}^I,\be_{\bu_f}^I)\|^2
+ |\be_{\bu_f}^I - \be_{\btheta}^I|^2_{\BJS}
\nonumber \\[1ex]
&\ds\quad + \|\be_{\btheta}^I\|^2_{\bLambda_{sh}}
+ \|\be_{\lambda}^I\|^2_{\Lambda_{ph}}
+ \|\be_{\bgamma_p}^I\|^2_{\bbQ_p} \Big)\,ds  
+ s_0\|\be_{p_p}^h(0)\|^2_{\W_p} 
+ \|A^{1/2}(\be_{\bsi_p}^h + \alpha_p\be_{p_p}^h \bI)(0)  \|^2_{\bbL^2(\Omega_p)} 
\nonumber \\[1ex]
&\ds\quad 
+ \int_0^t \left( \|A^{1/2}(\be_{\bsi_p}^h + \alpha_p \be_{p_p}^h \bI) \|^2_{\bbL^2(\Omega_p)}
+ s_0 \|\partial_t \be_{p_p}^h\|^2_{\W_p}
+ \|\partial_t A^{1/2}(\be_{\bsi_p}^h + \alpha_p \be_{p_p}^h \bI)\|^2_{\bbL^2(\Omega_p)} \right)\,ds \bigg).
\end{align}

\noindent{\bf Bounds on time derivatives.}

\medskip
\noindent
In order to bound the terms $s_0 \|\partial_t\be_{p_p}^h\|^2_{\W_p}$ and $\|\partial_t A^{1/2}(\be_{\bsi_p}^h + \alpha_p\be_{p_p}^h \bI)\|^2_{\bbL^2(\Omega_p)}$ in the right-hand side of \eqref{eq: error equation 6-}, 
we differentiate in time the whole system \eqref{eq: error equation}, test with $\bq_h=(\partial_t\be_{\bsi_p}^h,\partial_t \be_{p_p}^h, \partial_t \be_{\bu_p}^h, \partial_t \be_{\bT_f}^h, \partial_t \be_{\bu_f}^h, \partial_t \be_{\btheta}^h)$ and $\bs_h=(\partial_t \be_{\lambda}^h, \partial_t \be_{\bu_s}^h, \partial_t \be_{\bgamma_p}^h)$, and proceed similarly to \eqref{eq: error equation 1}, to find that
\begin{align}\label{eq: error equation 7}
&\ds \frac{1}{2} s_0\, \partial_t \,( \partial_t \be_{p_p}^h, \partial_t \be_{p_p}^h)_{\Omega_p} 
+ \frac{1}{2} \partial_t \, a_e( \partial_t \be_{\bsi_p}^h, \partial_t \be_{p_p}^h; \partial_t  \be_{\bsi_p}^h,  \partial_t \be_{p_p}^h) 
+ a_p(\partial_t \be_{\bu_p}^h, \partial_t \be_{\bu_p}^h)  \nonumber \\[1ex]
& \ds\quad + \, a_f(\partial_t \be_{\bT_f}^h,\partial_t \be_{\bu_f}^h;\partial_t \be_{\bT_f}^h,\partial_t \be_{\bu_f}^h) 
+ \kappa_{\partial_t \be_{\bu_f}^h}(\bT_f,\bu_f;\partial_t \be_{\bT_f}^h,\partial_t \be_{\bu_f}^h) 
\nonumber \\[1ex]
& \ds\quad 
+\, \kappa_{\bu_{fh}}(\partial_t \be_{\bT_f}^h,\partial_t \be_{\bu_f}^h;\partial_t \be_{\bT_f}^h,\partial_t \be_{\bu_f}^h) 
+ a_{\BJS}(\partial_t \be_{\bu_f}^h,\partial_t \be_{\btheta}^h;\partial_t \be_{\bu_f}^h,\partial_t \be_{\btheta}^h)  \nonumber  \\[1ex]
& \ds 
= - \, a_e( \partial_{tt}\,\be_{\bsi_p}^I, \partial_{tt}\,\be_{p_p}^I; \partial_t \be_{\bsi_p}^h, \partial_t \be_{p_p}^h) 
- a_p(\partial_t \be_{\bu_p}^I, \partial_t \be_{\bu_p}^h) 
- a_f(\partial_t \be_{\bT_f}^I, \partial_t \be_{\bu_f}^I, \partial_t \be_{\bT_f}^h,\partial_t \be_{\bu_f}^h)  \nonumber  \\[1ex]
&\ds\quad -\,\kappa_{\partial_t \bu_{fh}}(\be_{\bT_f}^h,\be_{\bu_f}^h;\partial_t \be_{\bT_f}^h,\partial_t \be_{\bu_f}^h)
- \kappa_{\be_{\bu_f}^h}(\partial_t \bT_f,\partial_t \bu_f;\partial_t \be_{\bT_f}^h,\partial_t \be_{\bu_f}^h) 
- \kappa_{\partial_t \be_{\bu_f}^I}(\bT_f,\bu_f;\partial_t \be_{\bT_f}^h,\partial_t \be_{\bu_f}^h)
\nonumber  \\[1ex]
&\ds \quad -\,\kappa_{\partial_t \bu_{fh}}(\be_{\bT_f}^I,\be_{\bu_f}^I; \partial_t \be_{\bT_f}^h,\partial_t \be_{\bu_f}^h) 
- \kappa_{\be_{\bu_f}^I}(\partial_t \bT_f,\partial_t \bu_f;\partial_t \be_{\bT_f}^h,\partial_t \be_{\bu_f}^h)  
- \kappa_{\bu_{fh}}(\partial_t \be_{\bT_f}^I,\partial_t \be_{\bu_f}^I; \partial_t \be_{\bT_f}^h,\partial_t \be_{\bu_f}^h) 
\nonumber  \\[1ex]
&\ds \quad
-\, a_{\BJS}(\partial_t \be_{\bu_f}^I,\partial_t \be_{\btheta}^I;\partial_t \be_{\bu_f}^h,\partial_t \be_{\btheta}^h)
- b_\sk(\partial_t \be_{\bgamma_p}^I,\partial_t \be_{\bsi_p}^h)
- \langle \partial_t \be_{\bu_f}^h \cdot \bn_f, \partial_t \be_{\lambda}^I \rangle_{\Gamma_{fp}} 
+ \langle \partial_t \be_{\bu_f}^I \cdot \bn_f, \partial_t \be_{\lambda}^h \rangle_{\Gamma_{fp}} 	
\nonumber  \\[1ex]
&\ds \quad 
+\, \langle \partial_t \be_{\btheta}^I \cdot \bn_p, \partial_t \be_{\lambda}^h \rangle_{\Gamma_{fp}}
+ b_\sk(\partial_t \be_{\bgamma_p}^h,\partial_t \be_{\bsi_p}^I)
+ \langle \partial_t \be_{\btheta}^h \cdot \bn_p, \partial_t \be_{\lambda}^I \rangle_{\Gamma_{fp}},
\end{align}
where, using again the projection properties \eqref{eq: interpolation 1} and
\eqref{eq: interpolation 2}--\eqref{eq: interpolation 3}, and 
the fact that $\bdiv(\bbX_{ph})=\bV_{sh}$, $\div(\bV_{ph})=\W_{ph}$, cf. \eqref{eq: div-prop}, 
we have dropped from \eqref{eq: error equation 7} the following terms:
\begin{equation*}
\begin{array}{c}
s_0(\partial_{tt}\,\be_{p_p}^I, \partial_t \be_{p_p}^h)_{\Omega_p}, \,\, b_p(\partial_t \be_{p_p}^I, \partial_t \be_{\bu_p}^h), \,\, b_p(\partial_t \be_{p_p}^h, \partial_t \be_{\bu_p}^I), \,\, 
b_{\bn_p}(\partial_t \be_{\bsi_p}^I, \partial_t \be_{\btheta}^h), \,\, b_{\bn_p}(\partial_t \be_{\bsi_p}^h, \partial_t \be_{\btheta}^I),  \\[2ex]
b_s(\partial_t \be_{\bu_s}^I, \partial_t \be_{\bsi_p}^h), \,\, b_s(\partial_t \be_{\bu_s}^h, \partial_t \be_{\bsi_p}^I), \,\, 
\langle \partial_t \be_{\bu_p}^h \cdot \bn_p, \partial_t \be_{\lambda}^I \rangle_{\Gamma_{fp}}, \,\, \langle \partial_t \be_{\bu_p}^I \cdot \bn_p, \partial_t \be_{\lambda}^h \rangle_{\Gamma_{fp}}.
\end{array}
\end{equation*}
We next comment on the control of the terms involving the functional $\kappa_{\bw_f}$. The two terms on the left are controlled by $a_f(\partial_t \be_{\bT_f}^h,\partial_t \be_{\bu_f}^h;\partial_t \be_{\bT_f}^h,\partial_t \be_{\bu_f}^h)$, using that $\bu_f(t), \bu_{fh}(t) \in \bW_r$, the continuity of $\kappa_{\bw_f}$ \eqref{eq:continuity-of-Kwf} and the coercivity bound for $a_f$ \eqref{eq:coercivity-af}. The two terms on the right involving $\bu_f$ and $\bu_{fh}$, using that $\bu_f(t), \bu_{fh}(t) \in \bW_r$, as well as the Cauchy--Schwarz and Young's inequalities, are bounded by $\delta \|(\partial_t \be^h_{\bT_f},\partial_t \be^h_{\bu_f})\|^2 + C \|(\partial_t \be_{\bT_f}^I, \partial_t \be_{\bu_f}^I)\|^2$. The other four terms are bounded by $\delta \|(\partial_t \be^h_{\bT_f},\partial_t \be^h_{\bu_f})\|^2 + C ( \|\partial_t \bu_f\|_{\bV_f}^2 + \|\partial_t \bu_{fh}\|_{\bV_f}^2 ) ( \|\be_{\bu_f}^I\|_{\bV_f}^2 + \|\be_{\bu_f}^h\|_{\bV_f}^2)$.
We also use the following identities to control the last two terms in \eqref{eq: error equation 7}:
\begin{gather*}
b_\sk(\partial_t \be_{\bgamma_p}^h,\partial_t \be_{\bsi_p}^I) \,=\, \partial_t \,b_\sk( \be_{\bgamma_p}^h,\partial_t \be_{\bsi_p}^I) - b_\sk(\be_{\bgamma_p}^h,\partial_{tt} \be_{\bsi_p}^I),
\\[1ex]
\langle \partial_t \be_{\btheta}^h \cdot \bn_p, \partial_t \be_{\lambda}^I \rangle_{\Gamma_{fp}} 
\,=\, \partial_t \langle \be_{\btheta}^h \cdot \bn_p, \partial_t \be_{\lambda}^I \rangle_{\Gamma_{fp}} 
- \langle \be_{\btheta}^h \cdot \bn_p, \partial_{tt} \be_{\lambda}^I \rangle_{\Gamma_{fp}}.
\end{gather*}
Then, integrating \eqref{eq: error equation 7} for $0$ to $t \in  (0,T]$, proceeding as in \eqref{eq: error equation 2}, using that $\bu_f(t), \bu_{fh}(t) \in \bW_r$, the bounds \eqref{eq:continuity-bounds}--\eqref{eq:continuity-of-Kwf} and \eqref{eq:positive-bound-E}--\eqref{eq:positive-bound-A+Kwf}, the Cauchy--Schwarz and Young's inequalities, and the estimate
\begin{equation*}
\ds  b_{\sk}(\partial_t \be_{\bgamma_p}^I, \partial_t \be_{\bsi_p}^h)
\,\leq\, C\,\|\partial_t \be_{\bgamma_p}^I\|_{\bbQ_p} \left( \|\partial_t A^{1/2}(\be_{\bsi_p}^h + \alpha_p \be^h_{p_p}\bI)\|_{\bbL^2(\Omega_p)} + \|\partial_t \be^h_{p_p}\|_{\W_p}  \right),
\end{equation*}
which follows analogously to \eqref{eq: bsk bound}, we get
\begin{align}\label{eq: error equation 7+}
&\ds s_0 \|\partial_t \be_{p_p}^h(t)\|^2_{\W_p} 
+ \|\partial_t A^{1/2}(\be_{\bsi_p}^h + \alpha_p \be_{p_p}^h \bI)(t)\|^2_{\bbL^2(\Omega_p)} \nonumber \\[1ex]
&\ds\quad +\, \int^t_0 \left( \|\partial_t \be_{\bu_p}^h\|^2_{\bL^2(\Omega_p)}
+ \|(\partial_t \be_{\bT_f}^h, \partial_t \be_{\bu_f}^h)\|^2  
+ |\partial_t \be_{\bu_f}^h - \partial_t \be_{\btheta}^h|^2_{\BJS} \right)\,ds 
\nonumber\\[1ex]
&\ds 
\leq\, C\,\bigg( \int^t_0 \Big( 
\|\partial_{tt} A^{1/2}(\be_{\bsi_p}^I + \alpha_p \be_{p_p}^I \bI)\|^2_{\bbL^2(\Omega_p)}
+ \|\partial_{tt} \be_{\bsi_p}^I\|^2_{\bbL^2(\Omega_p)}  
+ \|\partial_t \be_{\bu_p}^I\|^2_{\bL^2(\Omega_p)}
+ \|(\partial_t \be_{\bT_f}^I, \partial_t \be_{\bu_f}^I)\|^2
\nonumber \\[1ex]
&\ds\quad 
+\, |\partial_t \be_{\bu_f}^I - \partial_t \be_{\btheta}^I|^2_{\BJS} 
+ \|\partial_t \be_{\btheta}^I\|^2_{\bLambda_{sh}}
+ \|\partial_t \be_{\lambda}^I\|^2_{\Lambda_{ph}} 
+ \|\partial_{tt} \be_{\lambda}^I\|^2_{\Lambda_{ph}}
+ \|\partial_t \be_{\bgamma_p}^I\|^2_{\bbQ_p}
\Big)\,ds
+ \|\partial_t \be_{\bsi_p}^I(t)\|^2_{\bbL^2(\Omega_p)} 
\nonumber \\[1ex]
&\ds\quad 
+\, \|\partial_t \be_{\lambda}^I(t)\|^2_{\Lambda_{ph}} 
+ \|\partial_t \be_{\bsi_p}^I(0)\|_{\bbL^2(\Omega_p)}^2 
+ \|\partial_t \be_{\lambda}^I(0)\|^2_{\Lambda_{ph}}
\bigg)
+ \delta_3\, \bigg( \int^t_0 \Big( \|\partial_t \be^h_{p_p}\|^2_{\W_p}
+ \|(\partial_t \be^h_{\bT_f},\partial_t \be^h_{\bu_f})\|^2
\nonumber\\[1ex]
&\ds\quad 
+\, 
\|\be_{\btheta}^h\|^2_{\bLambda_{sh}} 
+ \|\be_{\bgamma_p}^h\|^2_{\bbQ_p}  
+ \|\partial_t \be_{\lambda}^h\|^2_{\Lambda_{ph}}
\Big)\,ds 
+ \|\be_{\btheta}^h(t)\|^2_{\bLambda_{sh}} 
+ \|\be_{\bgamma_p}^h(t)\|^2_{\bbQ_p}
\bigg)  
\nonumber \\[1ex]
&\ds\quad
+\, C\,\bigg( \int^t_0 \|\partial_t A^{1/2}(\be_{\bsi_p}^h + \alpha_p \be_{p_p}^h \bI)\|^2_{\bbL^2(\Omega_p)}\,ds + s_0 \|\partial_t \be_{p_p}^h(0)\|^2_{\W_p} 
+ \|\partial_t A^{1/2}(\be_{\bsi_p}^h + \alpha_p \be_{p_p}^h\bI)(0)\|^2_{\bbL^2(\Omega_p)} 
\nonumber \\[1ex]
&\ds\quad 
+\, \|\be_{\btheta}^h(0)\|_{\bLambda_{sh}}^2
+ \|\be_{\bgamma_p}^h(0)\|_{\bbQ_p}^2 
\,+\, \int^t_0 \left( \|\partial_t \bu_f\|_{\bV_f}^2 + \|\partial_t \bu_{fh}\|_{\bV_f}^2 \right) 
\left( \|\be_{\bu_f}^I\|_{\bV_f}^2 + \|\be_{\bu_f}^h\|_{\bV_f}^2 \right)\,ds \bigg).
\end{align}
Note that for the last term we can use the fact that both $\|\partial_t \bu_f\|_{\L^2(0,T;\bV_f)}$ and  $\|\partial_t \bu_{fh}\|_{\L^2(0,T;\bV_f)}$ are bounded by data (cf. \eqref{eq:continuous-stability}, \eqref{eq: discrete stability}), to obtain
\begin{equation}\label{eq: u bound}
\int_0^t \left( \|\partial_t \bu_f\|_{\bV_f}^2 + \|\partial_t \bu_{fh}\|_{\bV_f}^2 \right) 
\left( \|\be_{\bu_f}^I\|_{\bV_f}^2 + \|\be_{\bu_f}^h\|_{\bV_f}^2 \right) ds 
\,\leq\, C\,\left( \|\be_{\bu_f}^I\|_{\L^{\infty}(0,t;\bV_f)}^2 
+ \|\be_{\bu_f}^h\|_{\L^{\infty}(0,t;\bV_f)}^2 \right) \,.
\end{equation}
In turn, testing \eqref{eq: error equation} with $\bq_h=(\0,0,\0,\partial_t\be^h_{\bT_f},\partial_t\be^h_{\bu_f},\0)$ and using that $\bu_f(t), \bu_{fh}(t) \in \bW_r$, we deduce
\begin{equation*}
\begin{array}{l}
\ds \|(\be^h_{\bT_f},\be^h_{\bu_f})(t)\|^2 
\,\leq\, C\,(1+r) \int^t_0 \left( \|(\be^h_{\bT_f},\be^h_{\bu_f})\|^2 
+ \|(\be^I_{\bT_f},\be^I_{\bu_f})\|^2 
+ \|\be_\btheta\|^2_{\bLambda_{sh}} 
+ \|\be_\lambda\|^2_{\Lambda_{ph}} 
\right) \,ds \\[2ex]
\ds\quad +\, \|(\be^h_{\bT_f},\be^h_{\bu_f})(0)\|^2
+ \delta_4 \int^t_0 \|(\partial_t \be^h_{\bT_f},\partial_t \be^h_{\bu_f})\|^2\,ds,
\end{array}
\end{equation*}
which implies
\begin{equation}\label{eq: u bound-2}
\begin{array}{l}
  \ds \|(\be^h_{\bT_f},\be^h_{\bu_f})\|^2_{\L^\infty(0,t;\bbX_f\times\bV_f)}
\,\leq\, C \int^t_0 \left( \|(\be^h_{\bT_f},\be^h_{\bu_f})\|^2 
+ \|(\be^I_{\bT_f},\be^I_{\bu_f})\|^2 
+ \|\be_\btheta^h\|^2_{\bLambda_{sh}} 
+ \|\be_\lambda^h\|^2_{\Lambda_{ph}} 
\right.  \\[2ex]
\ds\quad \left. + \|\be_\btheta^I\|^2_{\bLambda_{sh}} 
+ \|\be_\lambda^I\|^2_{\Lambda_{ph}} \right) ds
+ \|(\be^h_{\bT_f},\be^h_{\bu_f})(0)\|^2
+ \delta_4 \int^t_0 \|(\partial_t \be^h_{\bT_f},\partial_t \be^h_{\bu_f})\|^2\,ds.
\end{array}
\end{equation}
We further utilize the inf-sup condition bound given in the first equation in \eqref{eq: error equation 3}, as well as the bound
\begin{equation}\label{eq: error equation 10}
  \int_0^t \left(\|\partial_t \be_{p_p}^h\|_{\W_p}^2 + \|\partial_t \be_\lambda^h\|_{\Lambda_{ph}}^2\right) ds \, 
\leq\, C \int_0^t \left(\|\partial_t \be_{\bu_p}^I\|_{\bL^2(\Omega_p)}^2
+ \|\partial_t \be_{\bu_p}^h\|_{\bL^2(\Omega_p)}^2\right) ds, 
\end{equation}
which follows similarly to the second equation in \eqref{eq: error equation 3}. In addition, noting that the term $\int_0^t \|\be_\lambda^h\|^2_{\Lambda_{ph}} ds$ on the right hind side of \eqref{eq: u bound-2} needs to be controlled by Gr\"onwall's inequality, we utilize Sobolev embedding to obtain
\begin{equation}\label{linf-bound-lambda}
  \|\be_\lambda^h(t)\|^2_{\Lambda_{ph}} \le C \int_0^t \left( \|\be_\lambda^h\|^2_{\Lambda_{ph}} + \|\partial_t\be_\lambda^h\|^2_{\Lambda_{ph}} \right) ds.
\end{equation}
We now combine \eqref{eq: error equation 6-} with \eqref{eq: error equation 7+}--\eqref{linf-bound-lambda} and the first equation in \eqref{eq: error equation 3}. We use Gr\"onwall's inequality to control the terms
\begin{equation*}
\int_0^t \left( \|A^{1/2}(\be_{\bsi_p}^h + \alpha_p \be_{p_p}^h\bI)\|^2_{\bbL^2(\Omega_p)} 
+ s_0 \|\partial_t \be_{p_p}^h\|^2_{\W_p}
+ \|\partial_t A^{1/2}(\be_{\bsi_p}^h + \alpha_p \be_{p_p}^h\bI)\|^2_{\bbL^2(\Omega_p)}
\right)\,ds,
\end{equation*}
which appear in \eqref{eq: error equation 6-}, and the terms
$$
\int^t_0 \left( \|(\be^h_{\bT_f},\be^h_{\bu_f})\|^2 
+ \|\be_\btheta^h\|^2_{\bLambda_{sh}} 
+ \|\be_\lambda^h\|^2_{\Lambda_{ph}} 
\right) \,ds,
$$
which appear in \eqref{eq: u bound-2},
and choose $\delta_3, \delta_4$ small enough, to obtain
\begin{align}\label{eq: error equation 11}
&\ds s_0 \|\be_{p_p}^h(t)\|^2_{\W_p} 
+ s_0 \|\partial_t \be_{p_p}^h(t)\|^2_{\W_p}
+ \|A^{1/2}(\be_{\bsi_p}^h + \alpha_p \be_{p_p}^h \bI)(t)\|^2_{\bbL^2(\Omega_p)} 
+ \|\bdiv(\be_{\bsi_p}^h(t))\|_{\bL^2(\Omega_p)} 
\nonumber  \\
&\ds \quad 
+\, \|\partial_t A^{1/2}(\be_{\bsi_p}^h + \alpha_p\be_{p_p}^h \bI)(t)\|^2_{\bbL^2(\Omega_p)}
+ \|\be_{\bu_s}^h(t)\|_{\bV_{s}}^2
+ \|\be_{\bgamma_p}^h(t)\|_{\bbQ_{p}}^2 
+ \|\be_{\btheta}^h(t)\|_{\bLambda_{sh}}^2
+ \|\be_\lambda^h(t)\|^2_{\Lambda_{ph}}
\nonumber  \\
&\ds \quad 
+\, \int_0^t \Big( 
\|\be_{p_p}^h\|^2_{\W_{p}} 
+ \|\partial_t \be_{p_p}^h\|^2_{\W_{p}} 
+ \|\bdiv(\be_{\bsi_p}^h)\|_{\bL^2(\Omega_p)} 
+ \|\be_{\bu_p}^h\|^2_{\bV_p} 
+ \|\partial_t \be_{\bu_p}^h\|^2_{\bL^2(\Omega_p)}
\nonumber  \\[1ex]
&\ds\quad 
+ \|(\be_{\bT_f}^h,\be_{\bu_f}^h)\|^2 
+ \|(\partial_t e_{\bT_f}^h, \partial_t \be_{\bu_f}^h)\|^2 
+ |\be_{\bu_f}^h - \be_{\btheta}^h|^2_{\BJS}
+ |\partial_t \be_{\bu_f}^h - \partial_t \be_{\btheta}^h|^2_{\BJS} 
\nonumber  \\[1ex]
&\ds \quad
+ \|\be_{\btheta}^h\|^2_{\bLambda_{sh}} 
+ \|\be_{\lambda}^h\|^2_{\Lambda_{ph}}
+ \|\partial_t \be_{\lambda}^h\|^2_{\Lambda_{ph}} 
+ \|\be_{\bu_s}^h\|^2_{\bV_{s}} 
+ \|\be_{\bgamma_p}^h\|^2_{\bbQ_{p}}
\Big) ds \nonumber \\[1ex]
&\ds 
\leq\, C\,\exp(T)\,\bigg( \int_0^t \Big( 
\|\partial_t A^{1/2}(\be_{\bsi_p}^I + \alpha_p \be_{p_p}^I\bI)\|^2_{\bbL^2(\Omega_p)} 
+ \|\partial_{tt} A^{1/2}(\be_{\bsi_p}^I + \alpha_p \be_{p_p}^I\bI)\|^2_{\bbL^2(\Omega_p)}
+ \|\partial_{tt} \be_{\bsi_p}^I\|^2_{\bbL^2(\Omega_p)} 
\nonumber \\[1ex]
&\ds\quad
+\, \|\be_{\bsi_p}^I\|^2_{\bbL^2(\Omega_p)} 
+ \|\be_{\bu_p}^I\|^2_{\bL^2(\Omega_p)}
+ \|\partial_t \be_{\bu_p}^I\|^2_{\bL^2(\Omega_p)}
+ \|(\be_{\bT_f}^I,\be^I_{\bu_f})\|^2
+ \|(\partial_t \be_{\bT_f}^I,\partial_t \be_{\bu_f})\|^2
\nonumber \\[1ex]
&\ds\quad
+\, |\be_{\bu_f}^I - \be_{\btheta}^I|^2_{\BJS}
+ |\partial_t \be_{\bu_f}^I - \partial_t \be_{\btheta}^I|^2_{\BJS}
+ \|\be_{\btheta}^I\|^2_{\bLambda_{sh}}
+ \|\partial_t \be_{\btheta}^I\|^2_{\bLambda_{sh}}
+ \|\be_{\lambda}^I\|^2_{\Lambda_{ph}}
+ \|\partial_t \be_{\lambda}^I\|^2_{\Lambda_{ph}}
+ \|\partial_{tt} \be_{\lambda}^I\|^2_{\Lambda_{ph}}
\nonumber \\[1ex]
&\ds\quad  
+\, \|\be_{\bgamma_p}^I\|^2_{\bbQ_p}
+ \|\partial_t \be_{\bgamma_p}^I\|^2_{\bbQ_p} \Big)\,ds
+ \|\be_{\bu_f}^I\|_{\L^{\infty}(0,t;\bV_f)}^2 
+ \|\partial_t \be_{\bsi_p}^I(t)\|_{\bbL^2(\Omega_p)}^2 
+ \|\partial_t \be_{\lambda}^I(t)\|^2_{\Lambda_{ph}} 
+ \|\partial_t \be_{\bsi_p}^I (0)\|_{\bbL^2(\Omega_p)}^2 
\nonumber \\[1ex]
&\ds\quad 
+\, \|\partial_t \be_{\lambda}^I (0)\|^2_{\Lambda_{ph}}
+ s_0 \|\be_{p_p}^h(0)\|^2_{\W_p} 
+ \|A^{1/2}(\be_{\bsi_p}^h + \alpha_p \be_{p_p}^h \bI)(0)\|^2_{\bbL^2(\Omega_p)} 
+ \|(\be^h_{\bT_f},\be^h_{\bu_f})(0)\|^2
+ \|\be_{\btheta}^h(0)\|_{\bLambda_{sh}}^2
\nonumber \\[1ex]
&\ds\quad 
+\, \|\be_{\bgamma_p}^h(0)\|_{\bbQ_p}^2
+ s_0 \|\partial_t \be_{p_p}^h(0)\|^2_{\W_p} 
+ \|\partial_t A^{1/2}(\be_{\bsi_p}^h + \alpha_p \be_{p_p}^h\bI)(0)\|^2_{\bbL^2(\Omega_p)}
\bigg) \,.
\end{align}

\noindent{\bf Bounds on initial data.}

\medskip
\noindent
Finally, to bound the initial data terms in \eqref{eq: error equation 11}, we recall from Theorems~\ref{thm:well-posedness-main-result} and
\ref{thm: well-posedness main result semi} that $(\bsi_{p}(0), p_{p}(0),\linebreak \bu_{p}(0),\bT_{f}(0), \bu_{f}(0), \btheta(0), \lambda(0)) = (\bsi_{p,0}, p_{p,0}, \bu_{p,0}, \bT_{f,0}, \bu_{f,0},  \btheta_{0}, \lambda_{0})$ and $(\bsi_{ph}(0), p_{ph}(0), \bu_{ph}(0), \bT_{fh}(0), \linebreak \bu_{fh}(0),\btheta_h(0), \lambda_h(0)) = (\bsi_{ph,0}, p_{ph,0}, \bu_{ph,0}, \bT_{fh,0}, \bu_{fh,0}, \btheta_{h,0}, \lambda_{h,0})$, respectively. 
Note that $\be_{\btheta}^h(0)=0$ by definition of $\btheta_{h,0}$ (c.f. \eqref{eq: interpolation 0}). Recall also that the discrete initial data satisfy \eqref{eq: dis ini NS-D} and \eqref{eq: dis ini sigma}. Then, in a way similar to \eqref{eq:NS-discrete-initial-data} and \eqref{bound-sigmap-h0}, we obtain
\begin{align}\label{eq: error initial 1}
&\ds  \|\be_{p_p}^h(0)\|_{\W_p} 
+ \|A^{1/2}(\be_{\bsi_p}^h + \alpha_p \be_{p_p}^h \bI)(0)\|_{\bbL^2(\Omega_p)} 
+ \|(\be^h_{\bT_f},\be^h_{\bu_f})(0)\| \nonumber \\[1ex]
&\ds
\qquad \leq C \big(\|\be_{\bp}^I(0)\|_{\bQ} + \|\be_{\lambda}^I(0)\|_{\Lambda_{ph}} + \|\be_{\brho_p}^I(0)\|_{\bbQ_p} \big).
\end{align}
Next, we differentiate in time the equations in \eqref{eq: error equation} with test functions
$\bv_{sh}$ and $\bchi_{ph}$ and combine them with the equations with test functions $\btau_{ph}$ and $w_{ph}$ at $t = 0$. Choosing $(\btau_{ph}, w_{ph}, \bv_{sh}, \bchi_{ph}) =
(\partial_t \be_{\bsi_p}^h(0), \partial_t \be_{p_p}^h(0), \be_{\bu_s}^h(0), \be_{\bgamma_p}^h(0))$, we obtain
\begin{equation*}
  \begin{split}
  s_0 \|\partial_t & \be_{p_p}^h(0)\|_{\W_p}^2
  + \|\partial_t A^{1/2}(\be_{\bsi_p}^h + \alpha_p \be_{p_p}^h \bI)(0)\|_{\bbL^2(\Omega_p)}^2 \\
  = & - (\dt A \be_{\bsi_p}^I(0),\dt(\be_{\bsi_p}^h + \alpha_p\be_{p_p}^h\bI)(0))_{\Omega_p}
  - (\alpha_p \be_{p_p}^I\bI(0),\dt \be_{\bsi_p}^h(0))_{\Omega_p} \\
  & - (\be_{\bgamma_p}^I(0),\dt\be_{\bsi_p}^h(0))_{\Omega_p}
  - (\dt\be_{\bsi_p}^I(0),\be_{\bgamma_p}^h(0))_{\Omega_p},
  \end{split}
  \end{equation*}
where we have used the orthogonality properties \eqref{eq: interpolation 1}, \eqref{eq: interpolation 2}, and \eqref{eq: interpolation 3}, as well as $b_p(w_{ph}, \bu_{p,0} - \bu_{ph,0}) = 0 \ \forall \, w_{ph} \in W_{ph}$ (cf. \eqref{eq: dis ini NS-D}). Using the Cauchy--Schwarz and Young's inequalities for the terms on the right hand side, as well as \eqref{eq: error equation 10} at $t = 0$ to control $\|\be_{\bgamma_p}^h(0)\|_{\bbQ_p}$, we deduce
\begin{align}\label{eq: error initial}
  &\ds
s_0 \|\partial_t \be_{p_p}^h(0)\|_{\W_p}^2 
+ \|\partial_t A^{1/2}(\be_{\bsi_p}^h + \alpha_p \be_{p_p}^h \bI)(0)\|_{\bbL^2(\Omega_p)}^2
+ \|\be_{\bgamma_p}^h(0)\|_{\bbQ_p}^2
\nonumber \\[1ex]
&\ds\quad 
\leq C\,\Big(\Big(1 + \frac{1}{s_0} \Big)\|\partial_t \be_{\bsi_p}^I(0)\|_{\bbX_p}^2 
+ \|\partial_t \be_{p_p}^I(0)\|_{\W_p}^2 + \|\be_{\bgamma_p}^I(0)\|_{\bbQ_p}^2  \Big).
\end{align}
Thus, combining \eqref{eq: error equation 11} with \eqref{eq: error initial 1} and  \eqref{eq: error initial}, making use of triangle inequality and the approximation properties \eqref{eq: approx property 1}, \eqref{eq: approx property 2}, and \eqref{eq: approx property 3}, we obtain \eqref{eq:errror-rate-of-convergence}.
\end{proof}

\begin{rem} The only dependence on $\ds \frac{1}{s_0}$ of the constant $C$ in the error estimate \eqref{eq:errror-rate-of-convergence} comes from the term
  $\ds \frac{1}{\sqrt{s_0}}\|\partial_t \be_{\bsi_p}^I(0)\|_{\bbX_p}$ in
  \eqref{eq: error initial}.
\end{rem}

\subsection{Fully discrete scheme}\label{sec:fully-discrete}

For the fully discrete scheme utilized in the numerical tests we employ the backward Euler method for the time discretization. 
Let $\Delta t$ be the time step, $T=N \Delta t$, $t_m = m\,\Delta t$, $m=0,\dots, N$. 
Let $d_t\, u^m := (\Delta t)^{-1}(u^m - u^{m-1})$ be the first order (backward) discrete 
time derivative, where $u^m := u(t_m)$. Then the fully discrete model reads: 
Given $(\bp_h^0, \br_h^0)=(\bp_{h,0}, \br_{h,0})$ satisfying \eqref{eq: discrete initial condition}, find $(\bp_h^m, \br_h^m)\in \bQ_h \times \bS_h$, $m=1, \dots, N$, such that 
\begin{equation}\label{eq: discrete formulation}
\begin{array}{rll}
\ds d_t\,\cE\,(\bp_h^m)(\bq_h)
+ (\cA+ \cK_{\bu^m_{fh}})\,(\bp_h^m)(\bq_h) + \cB'\,(\br_h^m)(\bq_h) & = & \bF^m(\bq_h) \quad \forall\,\bq_h\in \bQ_h\,, \\[2ex]
\ds - \cB\,(\bp_h^m)(\bs_h) & = & \bG^m(\bs_h) \quad \forall\,\bs_h\in \bS_h \,.
\end{array}	
\end{equation}
The fully discrete method results in the solution of a nonlinear algebraic system at each time step. The system is similar to the discrete resolvent system \eqref{discrete-domain}, which was analyzed in Theorem~\ref{thm:domain-is-nonempty}. The well posedness and error analysis of the fully discrete scheme is beyond the scope of the paper.

\section{Numerical results}\label{sec:numerical}

In this section we present numerical results that illustrate the behavior of the fully discrete method \eqref{eq: discrete formulation}. We use the Newton--Rhapson method to solve this nonlinear algebraic system at each time step. Our implementation is based on a {\tt FreeFem++} code \cite{freefem} on triangular grids, in conjunction with the direct linear solver {\tt UMFPACK} \cite{umfpack}. For spatial discretization we use the following finite element spaces: $\bbBDM_1-\bP_1$ for stress--velocity in Navier--Stokes, 
$\bbBDM_1-\bP_0-\bbP_1$ for stress--displacement--rotation in elasticity,
$\bBDM_1-\rP_0$ for Darcy velocity--pressure, and $\bP^\dc_1 - \rP^\dc_1$ for the the traces of structure velocity and Darcy pressure, where $\rP^\dc_1$ denotes discontinuous piecewise linear polynomials. 

The examples considered in this section are described next.
Example~1 is used to corroborate the rates of convergence. 
In Example 2 we present a simulation of blood flow in an arterial bifurcation. Air flow through a filter is simulated in Example~3. In all examples we set $\ds \kappa_1 = \frac{1}{2\mu}$ and $\kappa_2 = 2\mu$, cf. Remark~\ref{rem:kappa}.

\subsection{Example 1: convergence test}

In this test we study the convergence for the space discretization using an analytical solution.
The domain is $\Omega = \Omega_f\cup\Gamma_{fp}\cup\Omega_p$, with $\Omega_f = (0,1)\times (0,1), \Gamma_{fp} = (0,1)\times\{0\}$, and $\Omega_p = (0,1)\times (-1,0)$; i.e.,
the upper half is associated with the Navier--Stokes flow, while the lower half represents the poroelastic medium governed by the Biot system, see Figure~\ref{fig:example1} (left). The analytical solution is given in Figure~\ref{fig:example1} (right). It satisfies the appropriate interface conditions along the interface $\Gamma_{fp}$.
\begin{figure}[ht]
\begin{minipage}{.49\textwidth}
\begin{center}
\includegraphics[width=.5\textwidth]{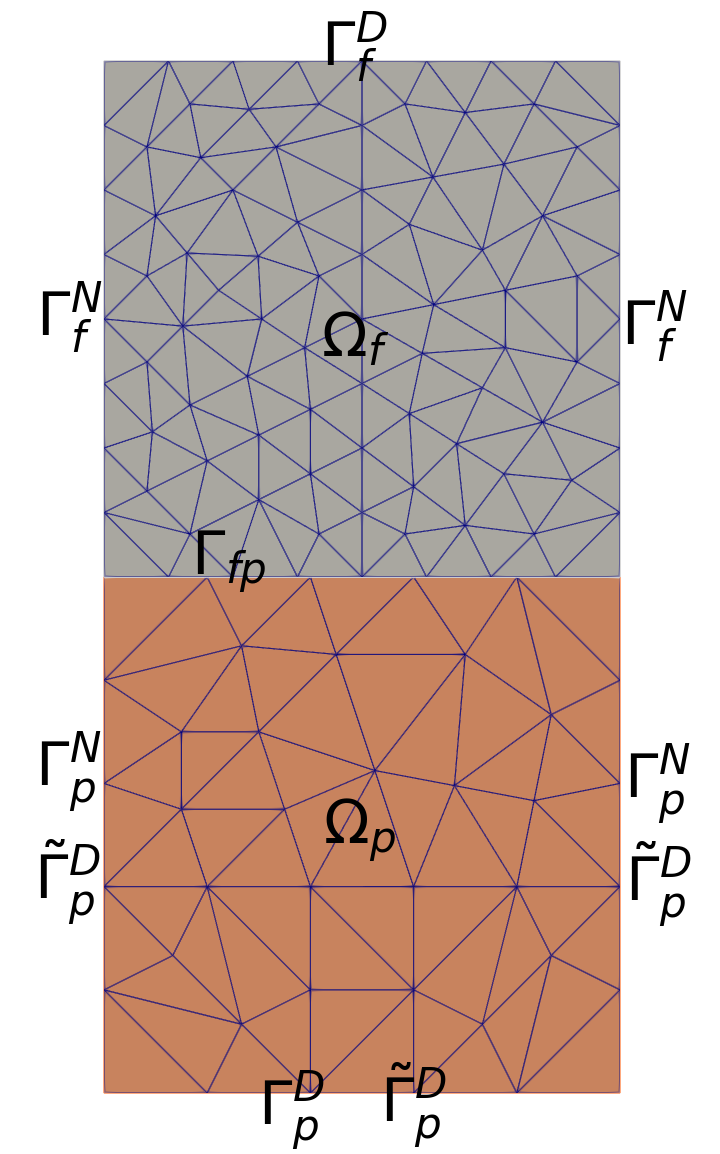}
\end{center}
\end{minipage}
\hfill
\begin{minipage}{.5\textwidth}

  Solution in the Navier--Stokes region:

  \bigskip
$ \ds \bu_f = \exp(t)
\begin{pmatrix}
\ds \sin(\pi \, x)\cos(\pi \, y) \\[1ex] \ds -\sin(\pi \, y)\cos(\pi \, x)
\end{pmatrix}$

\smallskip
$ \ds p_f = \exp(t)\,\sin(\pi x)\cos\Big(\frac{\pi y}{2}\Big) + 2\pi \cos(\pi t)$

\bigskip
\bigskip
Solution in the Biot region:

\bigskip
$\ds p_p = \exp(t)\,\sin(\pi x)\cos\Big(\frac{\pi y}{2}\Big)$

$ \ds \bbeta_p = \sin(\pi t) \begin{pmatrix} \ds -3x+\cos(y) \\[1ex] \ds y+1 \end{pmatrix}$
\end{minipage}
\caption{Example 1, Left: computational domain and boundaries. Right: analytical solution.}
\label{fig:example1}
\end{figure}
The model parameters are
$$
\mu = 1, \quad \rho = 1, \quad \lambda_p = 1, \quad \mu_p = 1, \quad s_0 = 1, \quad \bK = \bI, \quad \alpha_p = 1, \quad \alpha_{\BJS} = 1.
$$
The right hand side functions $\f_f, \f_p$ and $q_p$ are computed from \eqref{eq:Navier-Stokes-1}--\eqref{eq:Biot-model} using the analytical solution.
The model problem is then complemented with the appropriate boundary conditions, as shown in Figure~\ref{fig:example1} (left), and initial data.
Notice that the boundary conditions are not homogeneous and therefore the right-hand side of the resulting system must be modified accordingly.
The total simulation time for this test case is $T=0.01$ and the time step is $\Delta t=10^{-3}$. The time step is sufficiently small, so that the time discretization error does not affect the spatial convergence rates.
Table~\ref{table1-example1} shows the convergence history for a sequence of quasi-uniform mesh refinements, where $h_f$ and $h_p$ denote the mesh sizes in $\Omega_f$ and $\Omega_p$, respectively, The grids are non-matching on the interface $\Gamma_{fp}$, see Figure~\ref{fig:example1} (left) for the coarsest level, with the mesh sizes for their traces, denoted by $h_{tf}$ and $h_{tp}$, satisfying $h_{tf}=\frac{5}{8} h_{tp}$.
We note that the Navier--Stokes pressure and displacement at $t_m$ are recovered by the post-processed formulae $\ds p_{fh}^m = -\frac{1}{n} \left( \tr(\bT_{fh}^m)  + \rho\,\tr(\bu_{fh}^m\otimes\bu_{fh}^m) \right)$ (cf. \eqref{eq:pressure-pf})
and $\ds \bbeta^m_{ph} = \Delta t\,\bu^m_{sh} + \bbeta^{m-1}_{ph}$, respectively.
The results illustrate that at least the optimal spatial rate of convergence $\cO(h)$ established in Theorem~\ref{thm: error analysis} is attained for all subdomain variables.
The Lagrange multiplier variables, which are approximated in $\bP^\dc_1-\rP^\dc_1$, exhibit a rate of convergence $\cO(h^2)$ in the $\L^2$-norm on $\Gamma_{fp}$, which is consistent with the order of approximation.

\begin{table}[ht]
\begin{center}
\begin{tabular}{c||cc|cc|cc}
\hline
& \multicolumn{2}{|c|}{$\|\be_{\bT_f}\|_{\ell^2(0,T;\bbX_f)}$}  
& \multicolumn{2}{|c|}{$\|\be_{\bu_f}\|_{\ell^2(0,T;\bV_f)}$} 
& \multicolumn{2}{|c}{$\|\be_{p_f}\|_{\ell^2(0,T;\L^2(\Omega_f))}$} \\ 
$h_f$ & error & rate & error & rate & error & rate \\  \hline
0.1964 & 1.79E-01 &   --   & 4.57E-02 &   --   & 3.42E-03 &   --   \\ 
0.0997 & 9.12E-02 & 0.9957 & 2.33E-02 & 0.9920 & 1.29E-03 & 1.4332 \\
0.0487 & 4.43E-02 & 1.0057 & 1.19E-02 & 0.9451 & 5.79E-04 & 1.1208 \\
0.0250 & 2.23E-02 & 1.0294 & 5.91E-03 & 1.0411 & 2.32E-04 & 1.3702 \\
0.0136 & 1.11E-02 & 1.1423 & 2.94E-03 & 1.1455 & 1.11E-04 & 1.2136 \\
0.0072 & 5.51E-03 & 1.1066 & 1.46E-03 & 1.0997 & 4.68E-05 & 1.3551 \\
\hline 
\end{tabular}
		
\medskip

\begin{tabular}{c||cc|cc|cc|cc}
\hline
& \multicolumn{2}{|c|}{$\|\be_{\bsi_p}\|_{\ell^\infty(0,T;\bbX_p)}$}  
& \multicolumn{2}{|c|}{$\|\be_{p_p}\|_{\ell^\infty(0,T;\W_p)}$} 
& \multicolumn{2}{|c|}{$\|\be_{\bu_p}\|_{\ell^2(0,T;\bV_p)}$} 
& \multicolumn{2}{|c}{$\|\be_{\bu_s}\|_{\ell^2(0,T;\bV_s)}$} 
\\ 
$h_p$ & error & rate & error & rate & error & rate & error & rate  \\  \hline
0.2828 & 2.73E-01 &   --   & 7.54E-02 &   --   & 1.04E-01 &   --   & 4.31E-02 &   --   \\
0.1646 & 1.37E-01 & 1.2731 & 3.84E-02 & 1.2480 & 5.01E-02 & 1.3513 & 2.22E-02 & 1.2249 \\
0.0779 & 6.67E-02 & 0.9650 & 1.91E-02 & 0.9328 & 2.39E-02 & 0.9888 & 1.08E-02 & 0.9616 \\
0.0434 & 3.37E-02 & 1.1690 & 9.39E-03 & 1.2150 & 1.16E-02 & 1.2359 & 5.41E-03 & 1.1865 \\
0.0227 & 1.69E-02 & 1.0634 & 4.70E-03 & 1.0658 & 5.79E-03 & 1.0738 & 2.71E-03 & 1.0668 \\
0.0124 & 8.43E-03 & 1.1462 & 2.35E-03 & 1.1429 & 2.89E-03 & 1.1452 & 1.35E-03 & 1.1456 \\
\hline 
\end{tabular}

\medskip
		
\begin{tabular}{cc|cc||c||cc|cc || c}
\hline
\multicolumn{2}{c|}{$\|\be_{\bgamma_p}\|_{\ell^2(0,T;\bbQ_p)}$}
&\multicolumn{2}{c||}{$\|\be_{\bbeta_p}\|_{\ell^2(0,T;\bL^2(\Omega_p))}$} &
& \multicolumn{2}{|c|}{$\|\be_{\btheta}\|_{\ell^2(0,T;\bL^2(\Gamma_{fp}))}$} 
& \multicolumn{2}{|c||}{$\|\be_{\lambda}\|_{\ell^2(0,T;\L^2(\Gamma_{fp}))}$} & \\ 
error & rate & error & rate & $h_{tp}$ & error & rate & error & rate & iter\\  \hline
5.02E-02 &   --   & 2.67E-04 &   --   & 0.2000 & 6.80E-03 &   --   & 1.07E-03 &   --   & 2.2 \\
1.41E-02 & 2.3489 & 1.38E-04 & 1.2234 & 0.1000 & 2.42E-03 & 1.4894 & 2.69E-04 & 2.0007 & 2.2 \\
3.01E-03 & 2.0649 & 6.72E-05 & 0.9613 & 0.0500 & 5.82E-04 & 2.0571 & 6.71E-05 & 2.0005 & 2.2 \\
7.27E-04 & 2.4280 & 3.36E-05 & 1.1864 & 0.0250 & 1.46E-04 & 1.9928 & 1.69E-05 & 1.9935 & 2.2 \\
1.80E-04 & 2.1517 & 1.68E-05 & 1.0667 & 0.0125 & 3.65E-05 & 2.0037 & 4.26E-06 & 1.9833 & 2.2 \\
4.80E-05 & 2.1819 & 8.40E-06 & 1.1456 & 0.0063 & 9.25E-06 & 1.9799 & 1.09E-06 & 1.9632 & 2.2 \\
\hline 
\end{tabular}
\caption{Example 1, Mesh sizes, errors, rates of convergences and average Newton iterations for the fully discrete system $(\bbBDM_1-\bP_1)-(\bbBDM_1-\bP_0-\bbP_1)-(\bBDM_1-\rP_0) - (\bP^\dc_1 - \rP^\dc_1)$ approximation for the Navier--Stokes/Biot model in no-matching grids.}\label{table1-example1}
\end{center}
\end{table}

\subsection{Example 2: blood flow in an arterial bifurcation}

In this example we present a simulation of blood flow in an arterial bifurcation. The Navier--Stokes equations model the flow in the lumen of the artery, whereas the Biot system models the flow in the arterial wall. We use the fully dynamic Navier--Stokes--Biot model, which is better suitable for this application. In particular, the Navier--Stokes momentum equation in the fluid region is
\begin{equation*}
\rho \, \partial_t \bu_f - \rho (\nabla \bu_f) \bu_f - \bdiv (\bsi_f)=\f_f \qin \Omega_f\times (0,T],    
\end{equation*}
and the elasticity equation in the Biot system is
\begin{equation*}
\rho_p \, \partial_{tt} \bbeta_p - \beta \, \bbeta_p - \bdiv(\bsi_p) = \f_p \qin \Omega_p\times (0,T],
\end{equation*}
where $\rho_p$ is the fluid density in the poroelastic region. 
The additional term $\beta \, \bbeta_p$ comes from the axially symmetric two dimensional formulation, accounting for the recoil due to the circumferential strain \cite{bukavc2015}. The physical parameters are chosen based on \cite{bukavc2015} and fall within the range of physiological values for blood flow:
\begin{gather*}
\mu = 0.035 \text{ g/cm-s},\quad
\rho = 1 \text{ g/cm}^3, \quad
s_0 = 5 \times 10^{-6} \text{ cm}^{2}\text{/dyn},\quad
\bK = 10^{-9}\times \bI \text{ cm}^2, \quad
\rho_p = 1.1 \text{ g/cm}^3,
\\[1ex]
\lambda_p = 4.28 \times 10^6 \text{ dyn/cm}^2,\quad
\mu_p = 1.07 \times 10^6 \text{ dyn/cm}^2,\quad
\beta = 5 \times 10^{7} \text{ dyn/cm}^4,\quad
\alpha_p = 1, \quad
\alpha_{\BJS} = 1.
\end{gather*}
The body force terms $\f_f$ and $\f_p$ and external source $q_p$ are set to zero, as well as the initial conditions. The computational domain and boundary conditions are shown in Figure~\ref{fig:mesh}. We note that the flow is driven by the time-dependent pressure data on the inflow boundary $\Gamma_f^{in}$:
\begin{equation}\label{pressurefunc}
p_{in}(t) = \left\{ \begin{array}{ll}
\ds \frac{P_{\max}}{2}\left( 1-\cos \big(\frac{2\pi t}{T_{\max}} \big) \right) \,, & \text{if} \,\ t\leq T_{\max}; \\[2ex]
0 \,, & \text{if}\,\  t>T_{\max},
\end{array} \right.
\end{equation}
where $P_{\max}=13,334$ dyn/cm$^2$ and $T_{\max}=0.003$ s. The total simulation time is $T=0.006$ s with a time step of size $\Delta t=10^{-4}$ s. The final time $T$ is chosen so that the pressure wave barely reaches the outflow boundary.

\begin{figure}[ht]
  \begin{minipage}{.49\textwidth}
    \begin{center}
      \includegraphics[width=\textwidth]{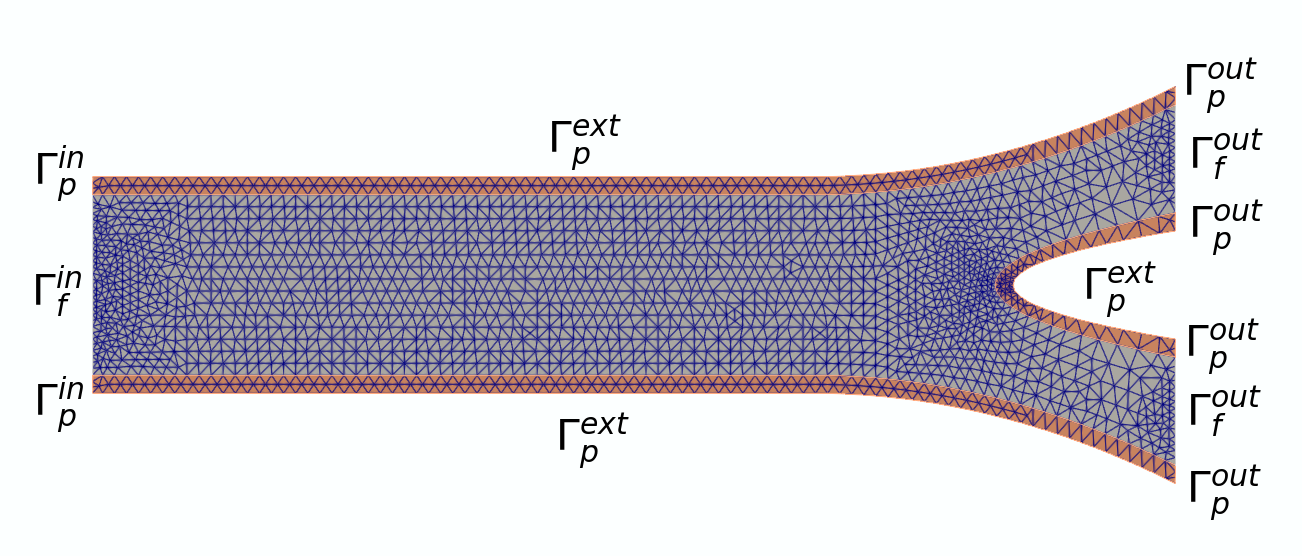}
      \end{center}
 \end{minipage}
\begin{minipage}{.5\textwidth}
\begin{align*}
& \bsi_f \, \bn_f = -p_{in} \, \bn_f \ \hbox{ on } \ \Gamma_f^{in}, 
&& \bsi_f \, \bn_f = \0  \ \hbox{ on } \ \Gamma_f^{out} \\[1ex]
  & \bu_s = \0 \ \hbox{ on } \  \Gamma_p^{in} \cup \Gamma_p^{out},
  && \bsi_p \bn_p=\0 \ \hbox{ on } \  \Gamma_p^{ext}\\[1ex]
  & \bu_p \cdot \bn_p = 0 \ \hbox{ on } \  \Gamma_p^{in} \cup \Gamma_p^{out},
  && p_p=0 \ \hbox{ on } \  \Gamma_p^{ext}
\end{align*}
\end{minipage}
\caption{Example 2, Left: computational domain and boundaries; arterial lumen $\Omega_{f}$ in gray, arterial wall $\Omega_{p}$ in brown. Right: boundary conditions.}
    \label{fig:mesh}
\end{figure}

We present the results of a simulation on a grid with a characteristic parameter $h$ four times smaller than that of the grid shown in Figure~\ref{fig:mesh} (left). We start by emphasizing that the values $s_0 = 5 \times 10^{-6}$ and $\bK = 10^{-9}\times \bI$ are in the typical locking regime for the Biot system of poroelasticity \cite{Yi-Biot-locking}. Our mixed finite element scheme provides a solution free of numerical oscillations, illustrating its locking-free behavior. We display the computed velocity and pressure at times $t=1.8, 3.6, 5.4$ ms in Figure~\ref{bifurcation}. On the top row, the arrows represent the velocity vectors $\bu_{fh}$ and $\bu_{ph}$ in the fluid and poroelastic regions, while the color shows the vector magnitude. The bottom plots present the fluid pressure $p_{fh}$ and Darcy pressure $p_{ph}$ in their corresponding regions. One can clearly see a wave propagating from left to right. The blood infiltrates from the lumen into the wall ahead of and near the highest pressure, while it flows back into the lumen after the pressure wave has passed.
We also observe singularity of $\bu_{fh}$ near the bifurcation of the fluid region at $t=5.4$ ms, which is typical for such geometry.

In Figure~\ref{bifurcation-stress}, the first and second rows of the stress tensors $\bsi_{fh}$ and $\bsi_{ph}$ are shown. The fluid stress $\bsi_{fh}$ has been recovered from the formula $\bsi_{fh}^m = \bT_{fh}^m + \rho\,(\bu_{fh}^m\otimes\bu_{fh}^m)$  (cf. \eqref{eq:nonlinear-stress-Tf}). We observe large stresses in the region of high pressure as the wave propagates. The continuity of stress across the artery-wall interfaces is also evident. To further illustrate the interface continuity conditions, we present several plots of various components of the solution along the top artery-wall interface in Figure~\ref{fig:top-interface}. The top row displays the normal components of the fluid velocity $\bu_{fh}\cdot\bn_f$, the total poroelastic velocity $-(\bu_{ph} + \bu_{sh})\cdot\bn_p$, and the Darcy velocity $-\bu_{ph}\cdot\bn_p$. We observe a good match between $\bu_{fh}\cdot\bn_f$ and $-(\bu_{ph} + \bu_{sh})\cdot\bn_p$, as expected from the conservation of mass interface condition, whereas the Darcy velocity $-\bu_{ph}\cdot\bn_p$ differs. We note that $-\bu_{ph}\cdot\bn_p$ is significantly smaller that the total velocity $-(\bu_{ph} + \bu_{sh})\cdot\bn_p$, due to relatively small permeability of the arterial wall. The bottom row shows the fluid wall shear stress $\bsi_{fh} \bn_f\cdot\bt_f$, the poroelastic shear stress $\bsi_{ph} \bn_p\cdot\bt_p$, and the normal displacement $-\bbeta_{ph}\cdot\bn_p$. We emphasize that our mixed formulation allows for direct and accurate computation of the wall shear stress, which is an important clinical marker. The profiles of $\bsi_{fh} \bn_f\cdot\bt_f$ and $\bsi_{ph} \bn_p\cdot\bt_p$ match, which is consistent with the continuity of normal stress condition. The profiles of the normal displacement at the three different times illustrate the propagating pressure wave, with large displacement in the regions of large fluid pressure. Moreover, we observe a correlation between the normal Darcy velocity $-\bu_{ph}\cdot\bn_p$ and the normal displacement $-\bbeta_{ph}\cdot\bn_p$. 

\begin{figure}[ht!]
\begin{center}
\includegraphics[width=0.325\textwidth]{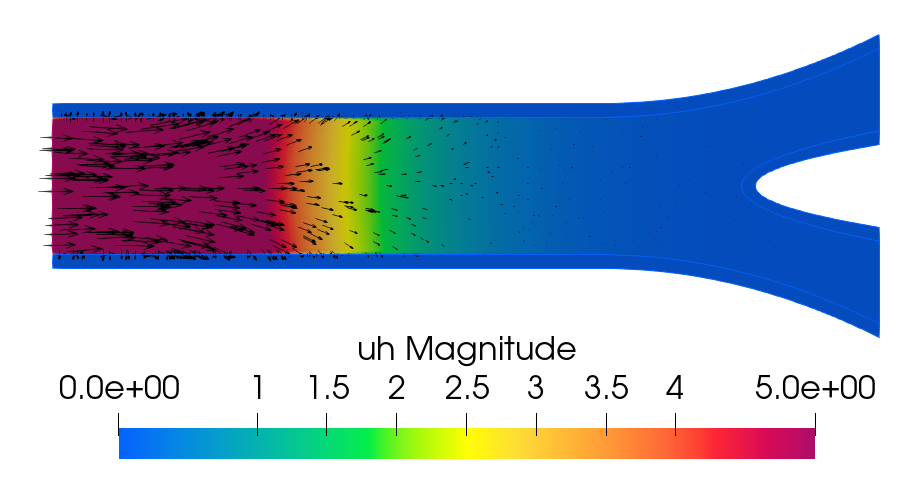}
\includegraphics[width=0.325\textwidth]{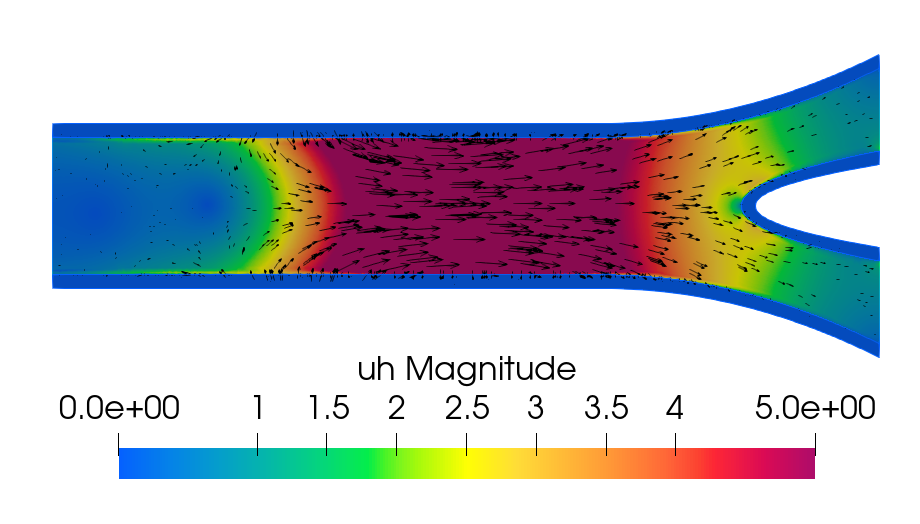}
\includegraphics[width=0.325\textwidth]{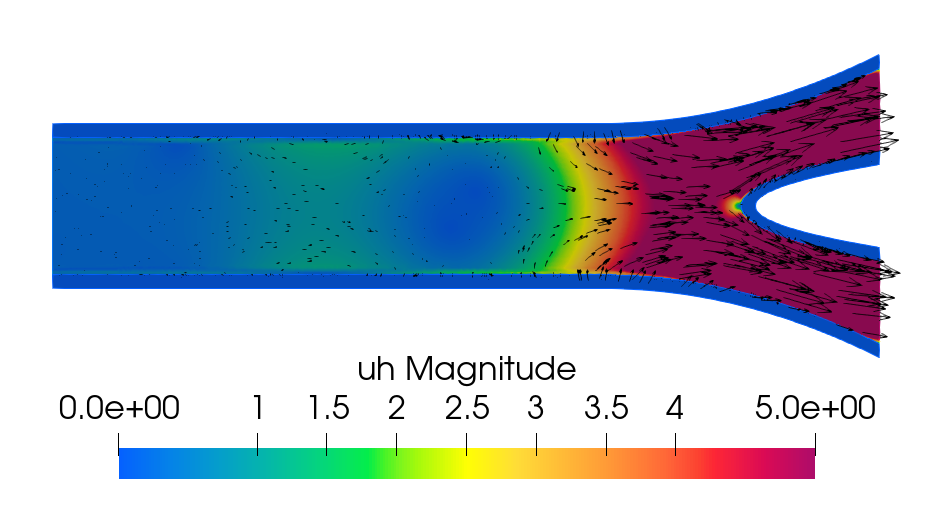}\\
\includegraphics[width=0.325\textwidth]{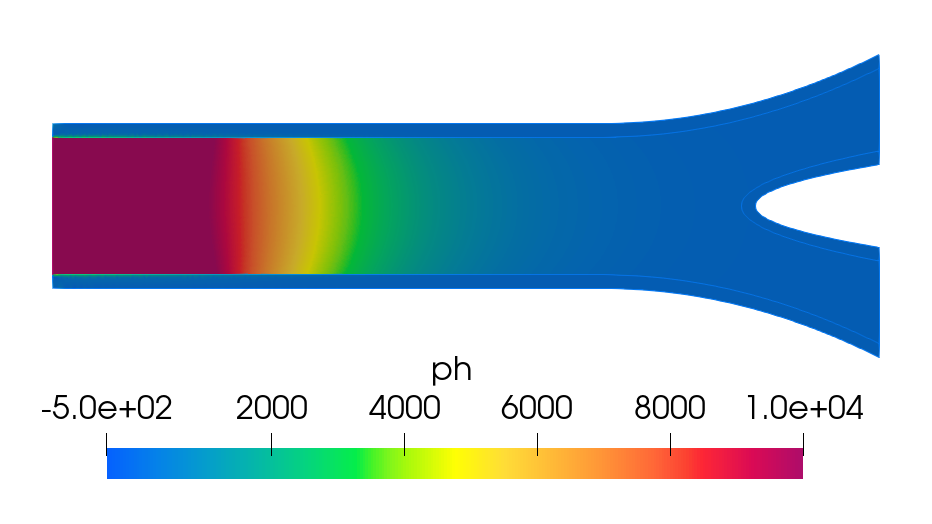}
\includegraphics[width=0.325\textwidth]{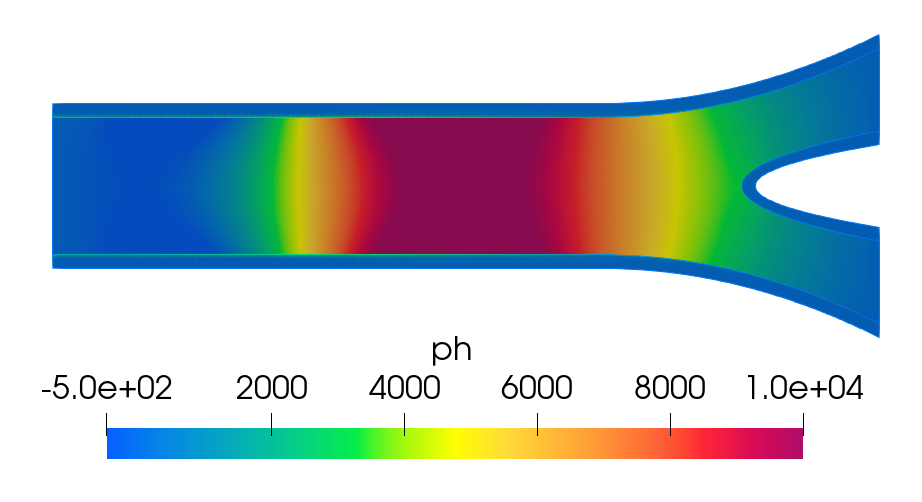}
\includegraphics[width=0.325\textwidth]{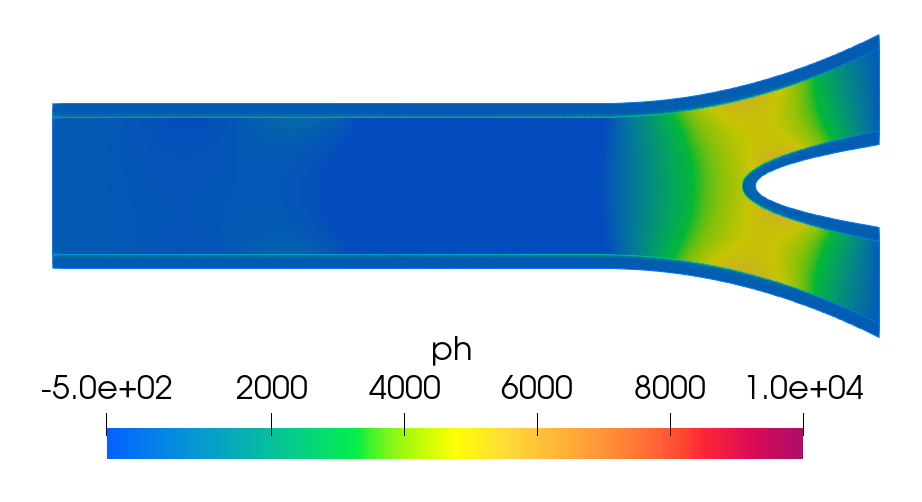}
\end{center}
\caption{Example 2, Computed solution at time t=1.8 ms, t=3.6 ms and t=5.4 ms. Top: velocities $\bu_{fh}$ and $\bu_{ph}$ (arrows), $|\bu_{fh}|$ and $|\bu_{ph}|$ (color); bottom: pressures $p_{fh}$ and $p_{ph}$ (color).}
\label{bifurcation}
\end{figure}

\begin{figure}[ht!]
\begin{center}
\includegraphics[width=0.325\textwidth]{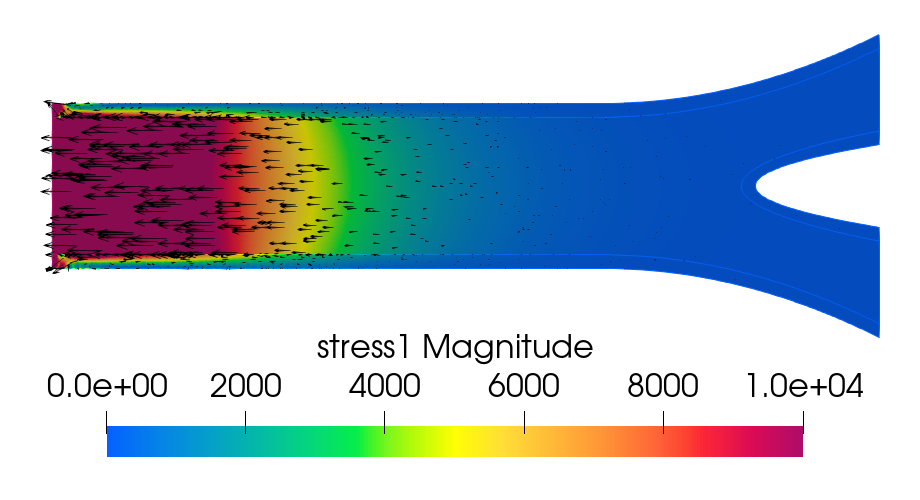}
\includegraphics[width=0.325\textwidth]{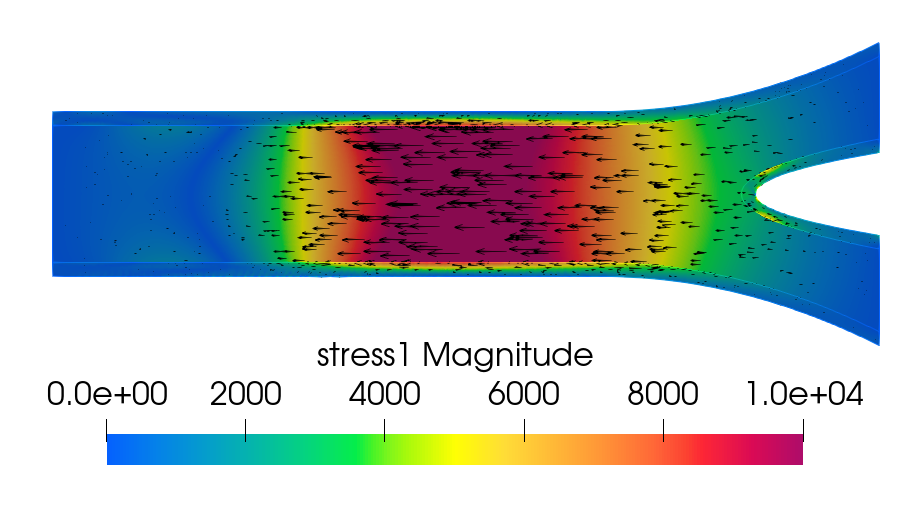}
\includegraphics[width=0.325\textwidth]{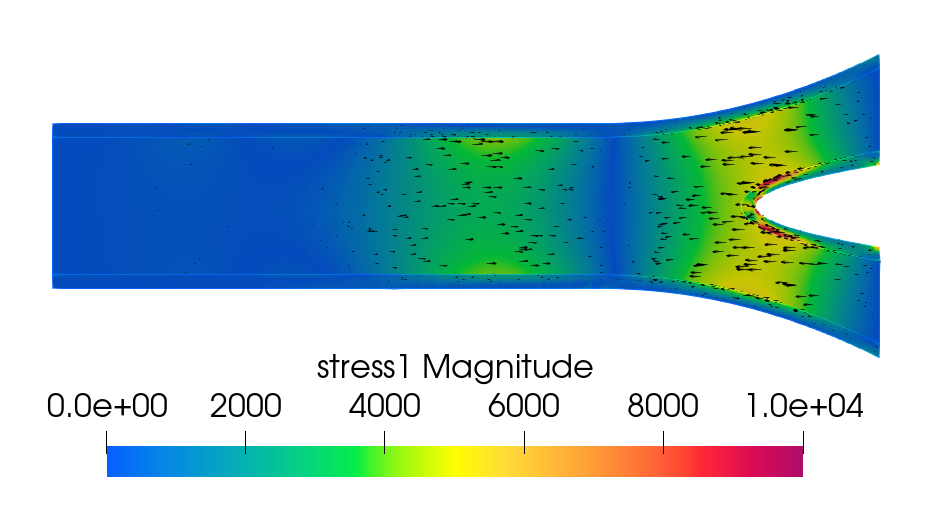}\\
\includegraphics[width=0.325\textwidth]{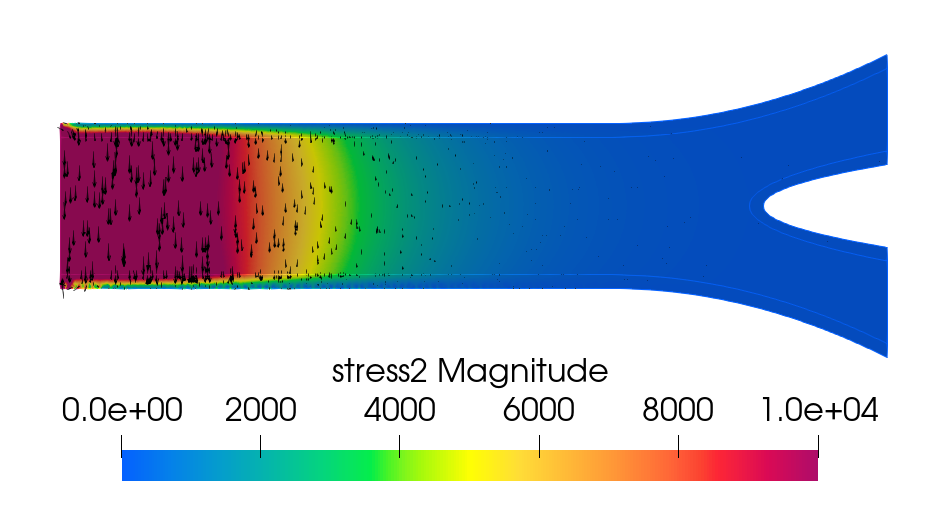}
\includegraphics[width=0.325\textwidth]{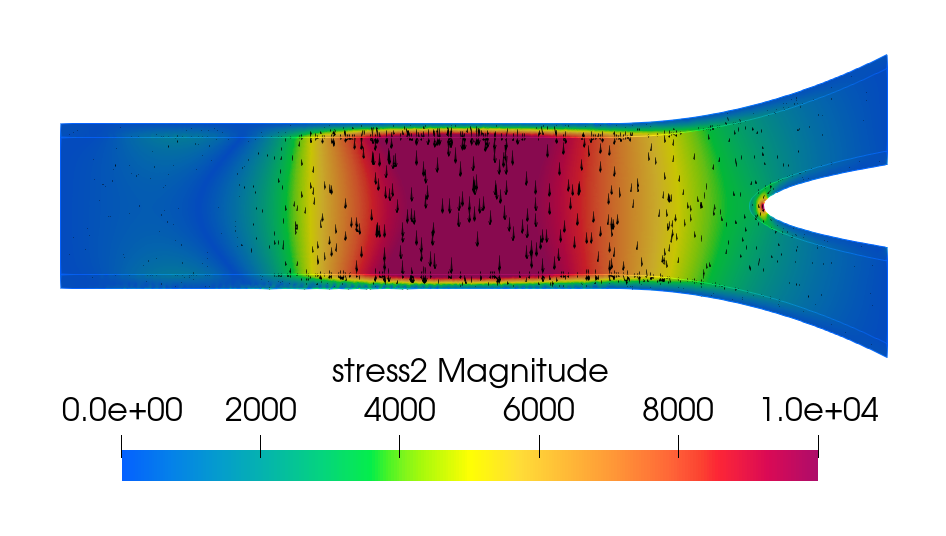}
\includegraphics[width=0.325\textwidth]{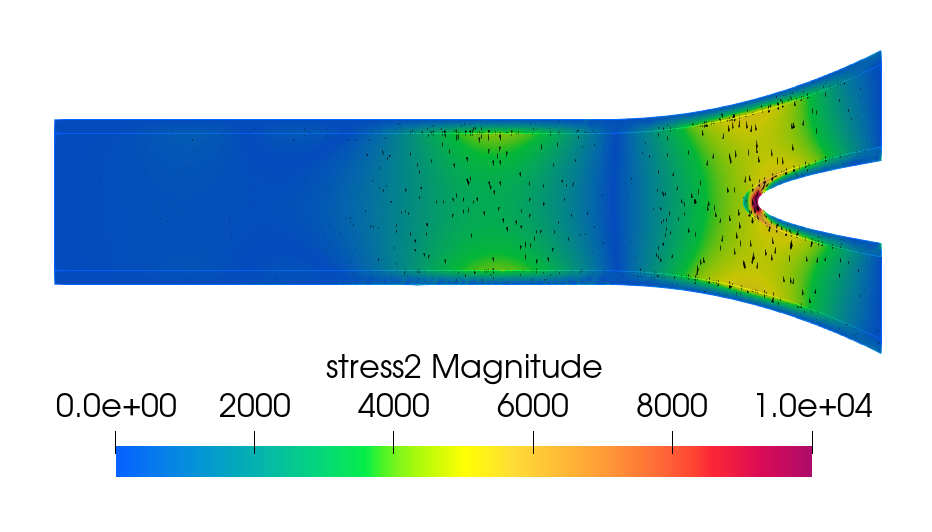}
\end{center}
\caption{Example 2, Computed solution at time t=1.8 ms, t=3.6 ms and t=5.4 ms. Top: first row of stresses $(\bsi_{fh,11},\bsi_{fh,12})^\rt$ and $(\bsi_{ph,11},\bsi_{ph,12})^\rt$ (arrows) and their magnitudes
  (color); bottom: second row of stresses $(\bsi_{fh,21},\bsi_{fh,22})^\rt$ and $(\bsi_{ph,21},\bsi_{ph,22})^\rt$ (arrows) and their magnitudes
  (color).}
\label{bifurcation-stress}
\end{figure}

\begin{figure}[ht!]
\begin{center}
\includegraphics[width=0.325\textwidth]{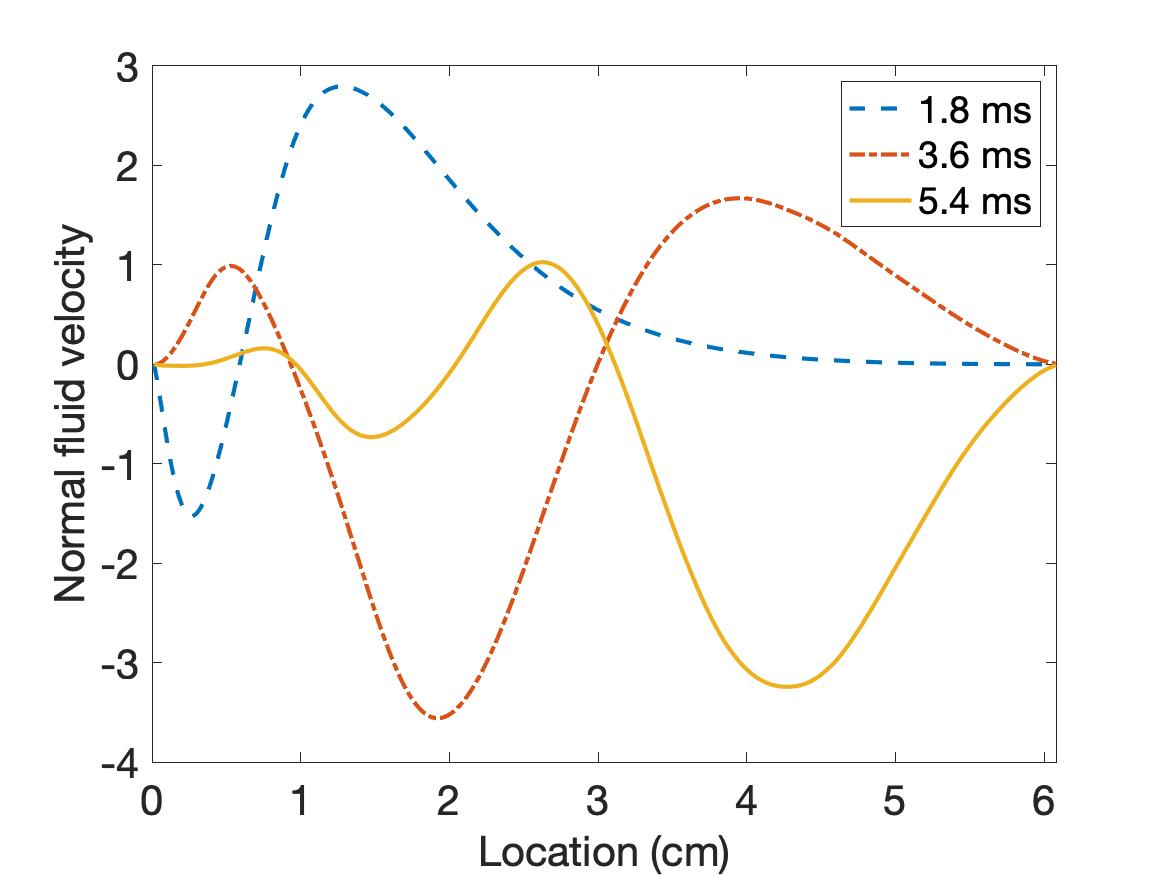}
\includegraphics[width=0.325\textwidth]{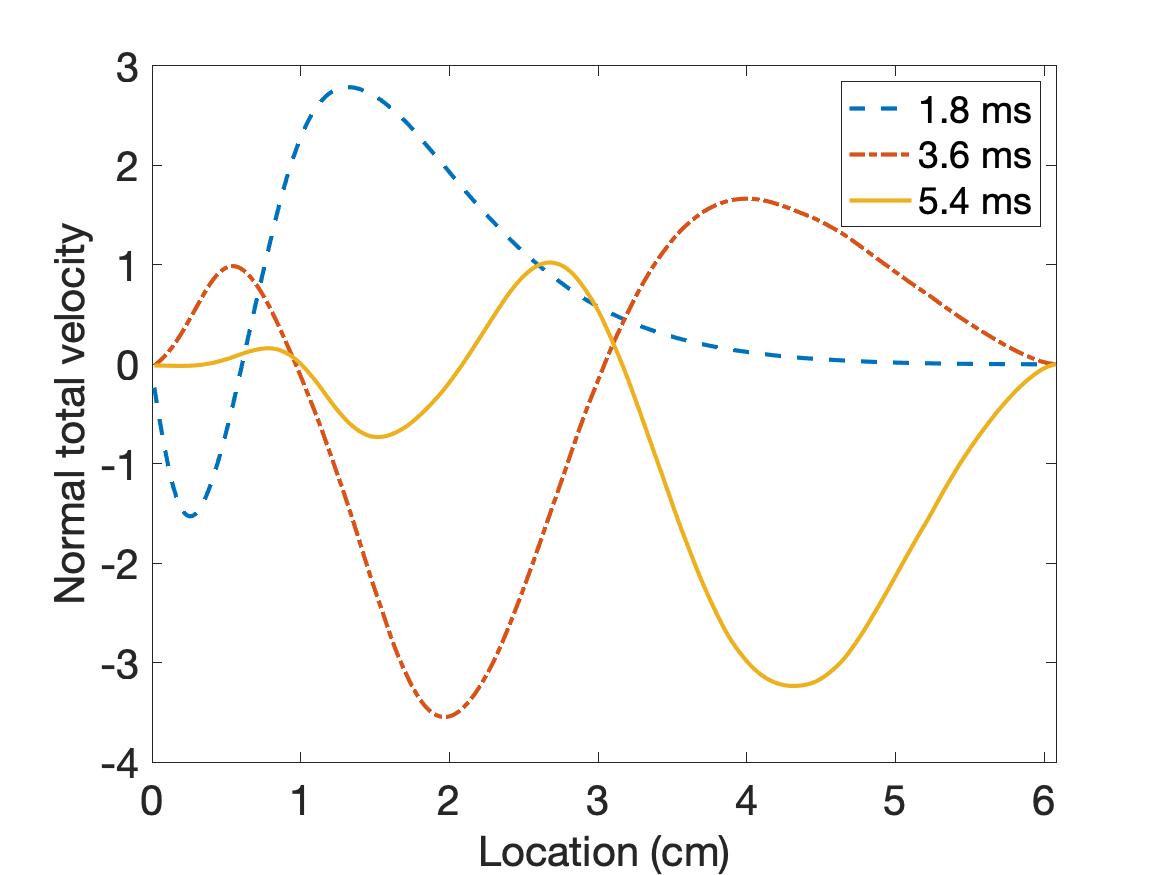}
\includegraphics[width=0.325\textwidth]{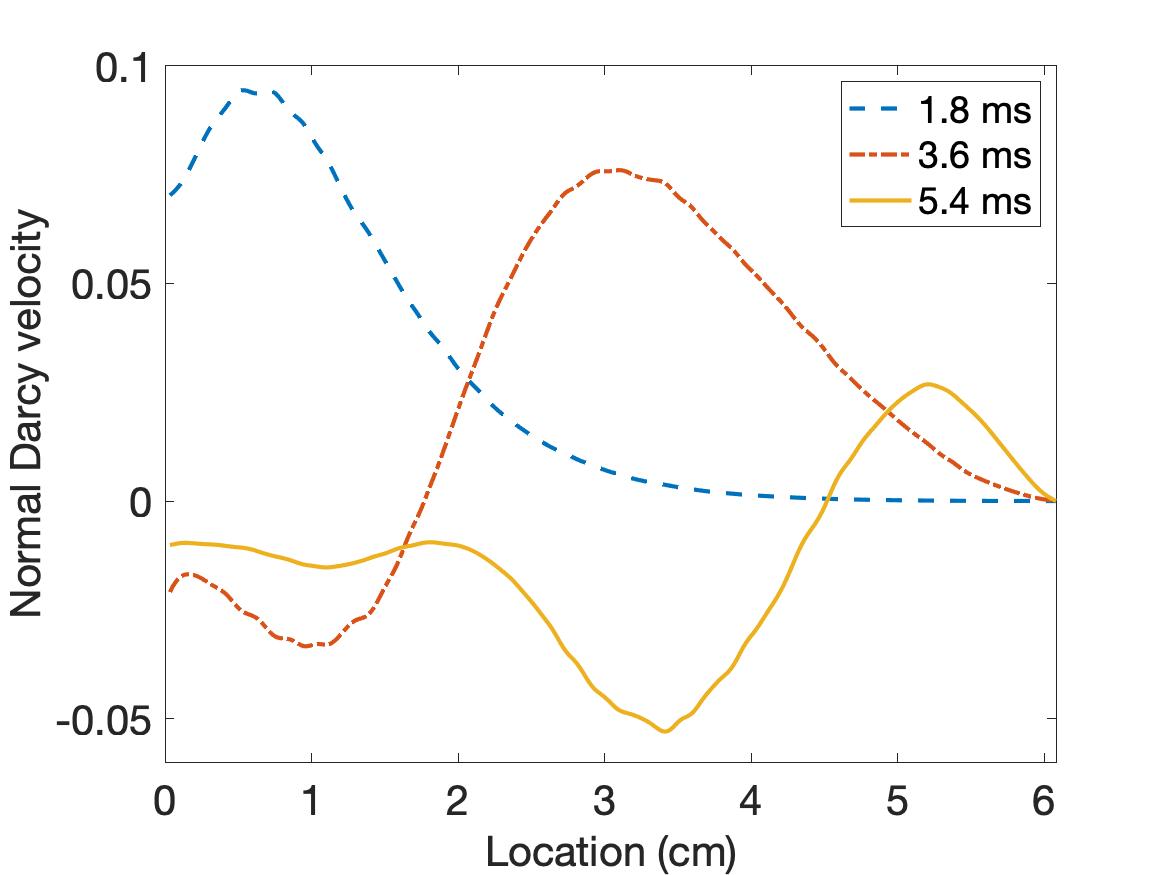}\\
\includegraphics[width=0.325\textwidth]{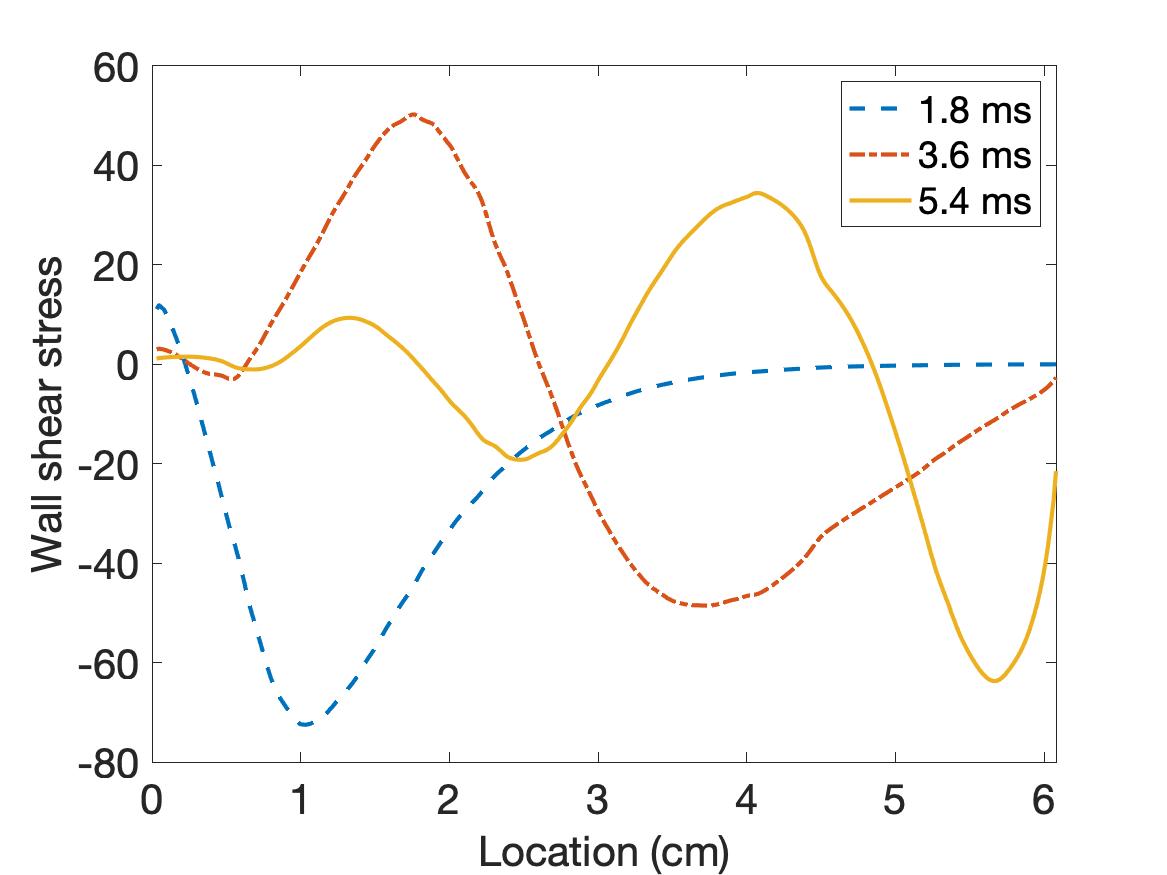}
\includegraphics[width=0.325\textwidth]{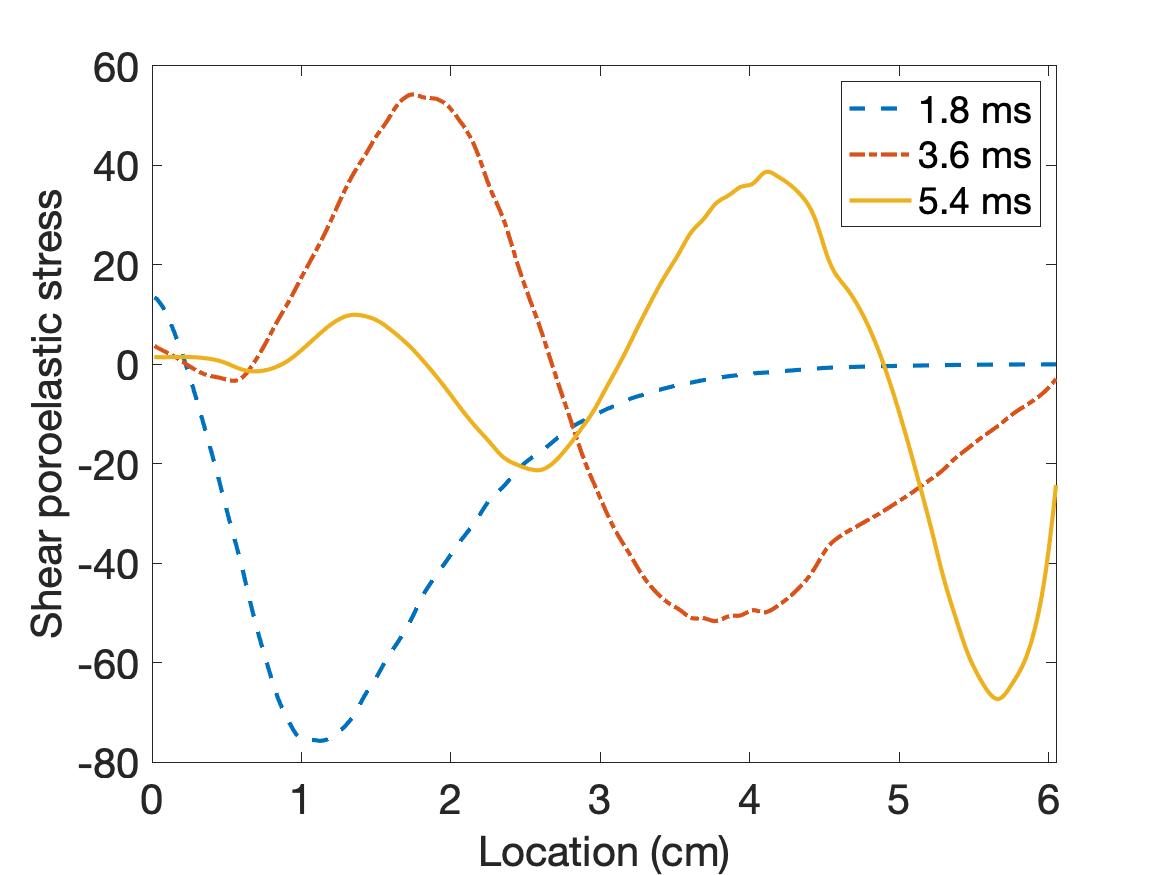}
\includegraphics[width=0.325\textwidth]{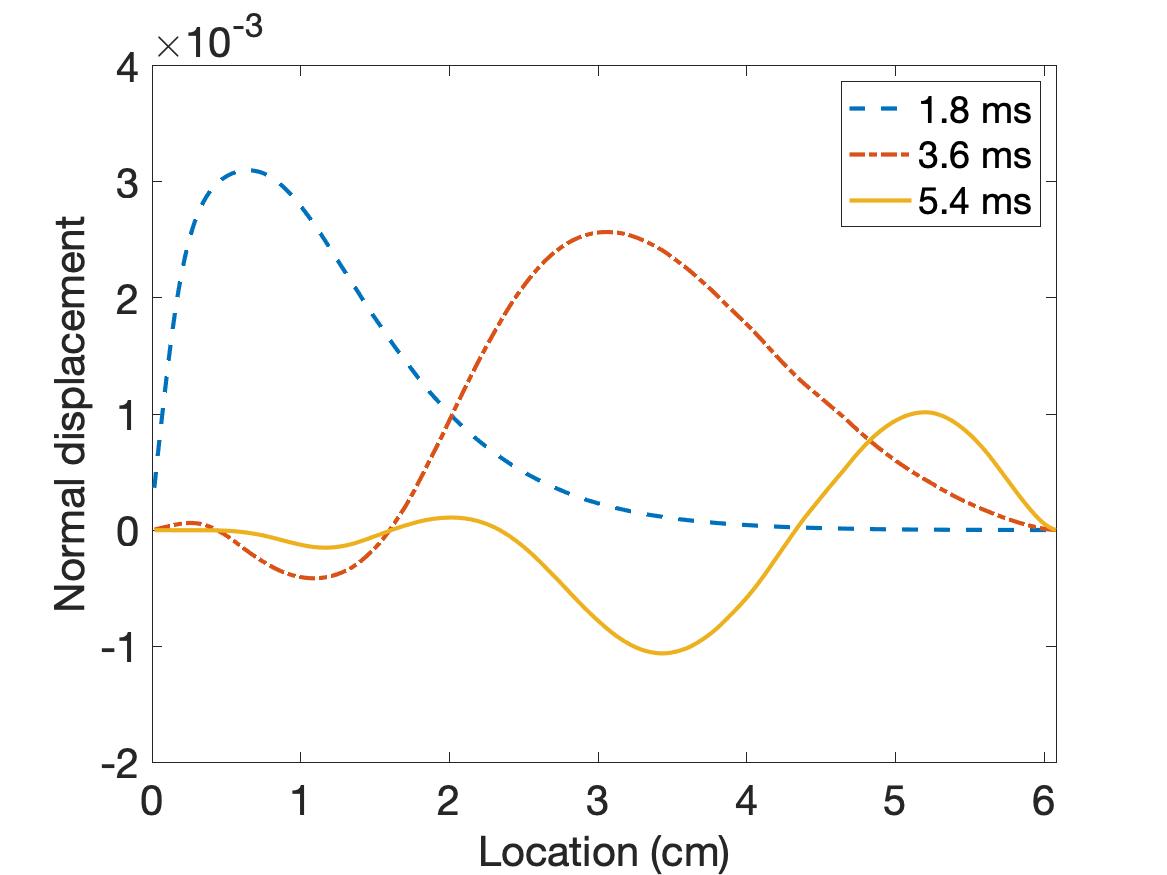}
\end{center}
\caption{Example 2, Computed solution on the top interface. Top: $\bu_{fh}\cdot\bn_f$, $-(\bu_{ph} + \bu_{sh})\cdot\bn_p$, $-\bu_{ph}\cdot\bn_p$; bottom: $\bsi_{fh} \bn_f\cdot\bt_f$, $\bsi_{ph} \bn_p\cdot\bt_p$, $-\bbeta_{ph}\cdot\bn_p$.}
\label{fig:top-interface}
\end{figure}

\subsection{Example 3: air flow through a filter}

In this example we simulate air flow through a filter. The setting is similar to the 
one presented in \cite{swgbh2020}. We consider a two-dimensional rectangular channel with length $0.75$ m and width $0.25$ m, which on the bottom center is partially blocked by a rectangular poroelastic filter of length $0.25$ m and width $0.2$ m, see Figure~\ref{fig:filter} (left).
\begin{figure}[ht]
  \begin{minipage}{.49\textwidth}
    \begin{center}
      \includegraphics[width=\textwidth]{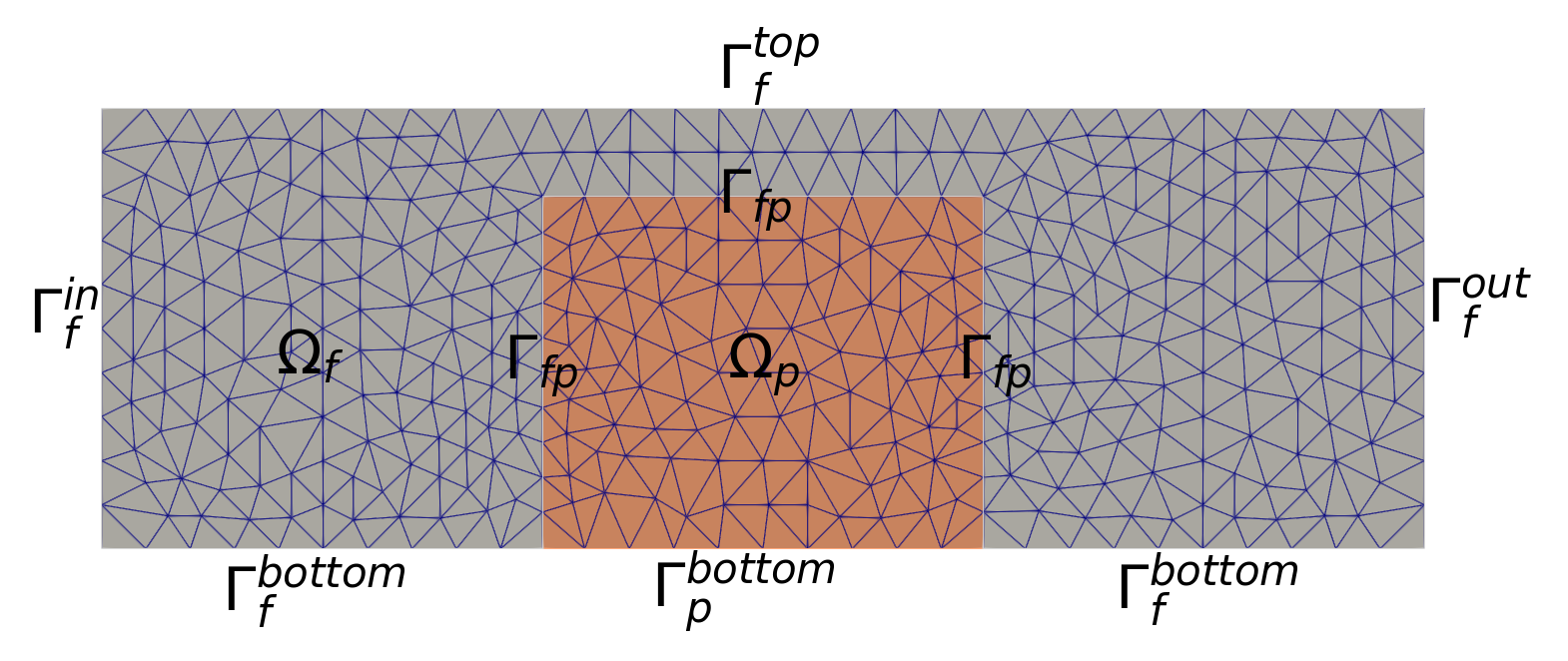}
      \end{center}
  \end{minipage}
  \begin{minipage}{.5\textwidth}

    $\ds \bsi_f \, \bn_f = -p_{in} \, \bn_f \qon \Gamma_f^{in}$
    
 \medskip   $\bsi_f \, \bn_f = -p_{out} \, \bn_f \qon \Gamma_f^{out}$
    
\medskip$ p_{in}= p_{ref} +10^{-9} \ \text{kPa}, \quad p_{out}= p_{ref} = 100 \ \text{kPa}$
    
\medskip    $ \ds \bu_f = 0 \qon \Gamma_f^{top} \cup \Gamma_f^{bottom}$
    
\medskip $ \ds \bu_s = \0 \qan \bu_p \cdot \bn_p = 0 \qon \Gamma_p^{bottom}$

\end{minipage}
\caption{Example 3, Left: computational domain and boundaries; channel $\Omega_{f}$ in gray, filter $\Omega_{p}$ in brown. Right: boundary conditions.}
\label{fig:filter}
\end{figure}
The model parameters are set as
\begin{gather*}
\mu=1.81 \times 10^{-8} \text{ kPa s}, \quad \rho =1.225 \times 10^{-3} \text{ Mg}/\text{m}^3, \quad s_0=7 \times 10^{-2} \text{ kPa}^{-1}, \\
\bK= [ 0.505, 0.495; 0.495,0.505] \times 10^{-6} \text{ m}^2, \quad \alpha_{\BJS}=1.0, \quad \alpha = 1.0.
\end{gather*}
Note that $\mu$ and $\rho$ are parameters for air. The permeability tensor $\bK$ is obtained by rotating the identity tensor by a $-45^\circ$ rotation angle in order to consider the effect of material anisotropy on the flow. We further consider two different kinds of material in the poroelastic region: ``hard'' material with parameters
$$\lambda_p=1 \times 10^5 \text{ kPa}, \quad \mu_p = 1 \times 10^4 \text{ kPa}, $$ 
and ``soft'' material with parameters
$$\lambda_p=1 \times 10^3 \text{ kPa}, \quad \mu_p = 1 \times 10^2 \text{ kPa}.$$
The top and bottom of the domain are rigid, impermeable walls. The flow is driven by a pressure difference $\Delta p = 10^{-9} \text{ kPa}$ between the left and right boundary, see Figure~\ref{fig:filter} (right) for the boundary conditions. The body force terms $\f_f$ and $\f_p$ and external source $q_p$ are set to zero. For the initial conditions, we consider
$$ p_{p,0}=100  \ \text{kPa}, \quad \bsi_{p,0} = -\alpha_p \, p_{p,0} \, \bI, \quad \bu_{f,0} = \0  \ \text{m/s}. $$
The computational grid has a characteristic parameter $h$ four times smaller than that of the grid shown in Figure~\ref{fig:filter} (left). The total simulation time is $T=80$ s with $\Delta t=1$ s.

In Figures~\ref{filter-velocity}--\ref{filter-displacement} we present various components of the computed solution at the final time. The plots on the left are for the hard material, whereas the plots on the right are for the soft material. Since the pressure variation is small relative to its value, for visualization purpose we plot its difference from the reference pressure, $p_{fh}-p_{ref}$ and $p_{ph}-p_{ref}$ in the corresponding regions. We do the same for the stress tensors, showing $\bsi_{fh}+\alpha \,p_{ref} \bI$ and $\bsi_{ph} + \alpha \, p_{ref} \bI$ respectively. In addition, the arrows representing the velocity and stress vectors are not scaled with their magnitudes.

The velocity plots in Figure~\ref{filter-velocity} show that most of the air passes through the constricted section above the filter with higher velocity in this region, due to the flow resistance in the porous medium. The effect of anisotropy is clearly visible in both the pressure and velocity profiles, with the pressure gradient and streamlines following the inclined principal direction of the permeability tensor. We also observe continuous normal velocity across all three interfaces and discontinuous tangential velocity, especially at the interfaces where the Navier--Stokes velocity is higher. This is consistent with the continuity of flux and BJS interface conditions.

Figure~\ref{filter-stress} shows the first and second rows of the fluid stress tensor $\bsi_{fh}$ and the poroelastic stress tensor $\bsi_{ph}$. The stress is larger in the poroelastic region, especially along the bottom boundary, where the displacement is set to zero. We observe continuity of the first row of the stress on the vertical interfaces and of the second row on the horizontal interface. Thus, the scheme exhibits continuity of the normal stress vector $\bsi_{fh}\bn_f + \bsi_{ph}\bn_p = 0$, since $\bn_f = \pm(1,0)^\rt$ on the vertical interfaces and $\bn_f = -(0,1)^\rt$ on the horizontal interface. 

Furthermore, the elasticity material parameters have a significant effect not only on the displacement field, but also on the velocity field outside of the poroelastic region. In particular, we observe a large vortex behind the obstacle for the soft material, as well as a smaller vortex in front of it, cf. Figure~\ref{filter-velocity} (right).  This is related to 
both the displacement and the structure velocity having larger magnitude for the soft material, as shown in Figure~\ref{filter-displacement}. We also note that the use of the inertial term in the Navier--Stokes equations plays a critical role for the accurate approximation of the recirculation zones. This example illustrates the ability of the model to capture the interplay between solid deformation and fluid flow, including the effect of material parameters and faster flows. It also shows the importance of including the poroelastic model on resolving critical flow characteristics compared with the Navier--Stokes -- Darcy model considered in \cite{swgbh2020}.

\begin{figure}[ht!]
\begin{center}
\includegraphics[width=0.45\textwidth]{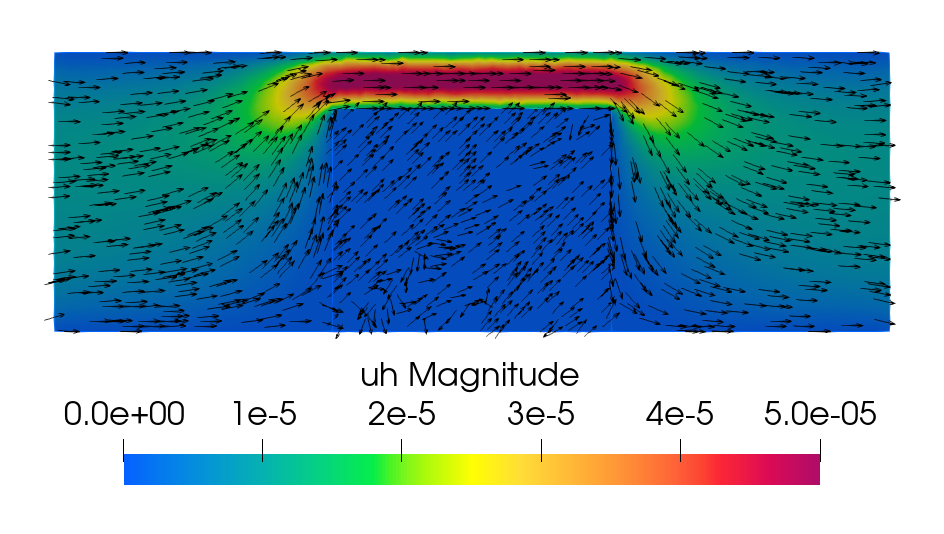}
\includegraphics[width=0.45\textwidth]{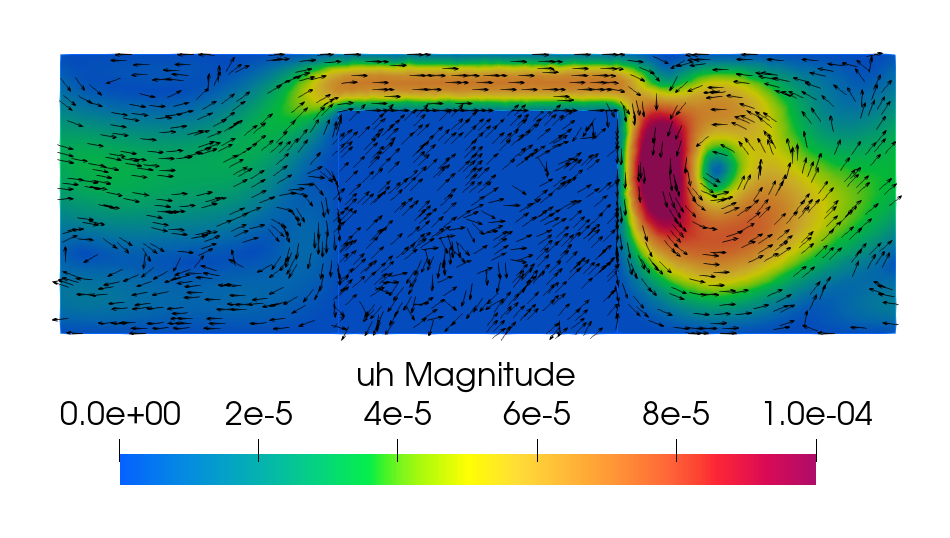}\\
\includegraphics[width=0.45\textwidth]{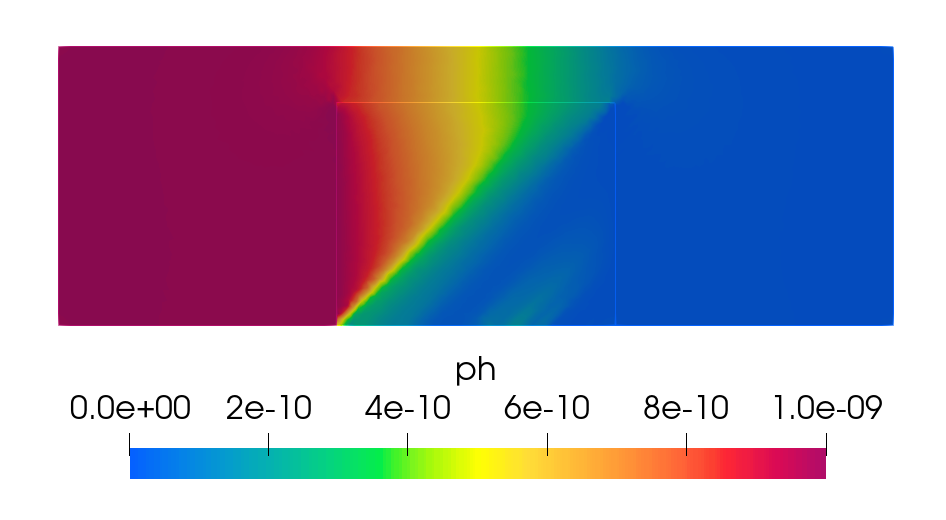}
\includegraphics[width=0.45\textwidth]{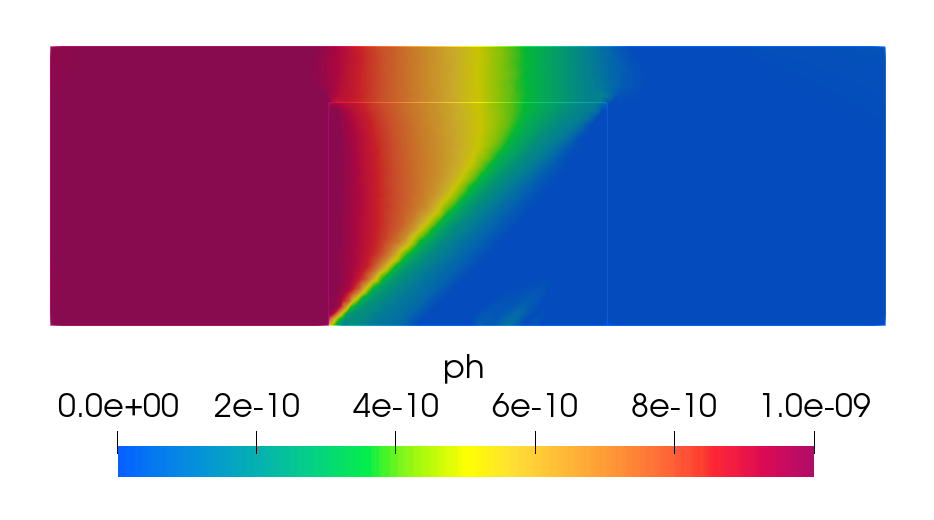}
\end{center}
\caption{Computed velocities $\bu_{fh}$ and $\bu_{ph}$ and pressures $p_{fh} - p_{ref}$ and $p_{ph} - p_{ref}$ for the hard material (left) and soft material (right) at time T=80 s. Top: velocities (arrows, not scaled) and their magnitudes (color); bottom: pressures (color).}
\label{filter-velocity}
\end{figure}

\begin{figure}[ht!]
\begin{center}
\includegraphics[width=0.45\textwidth]{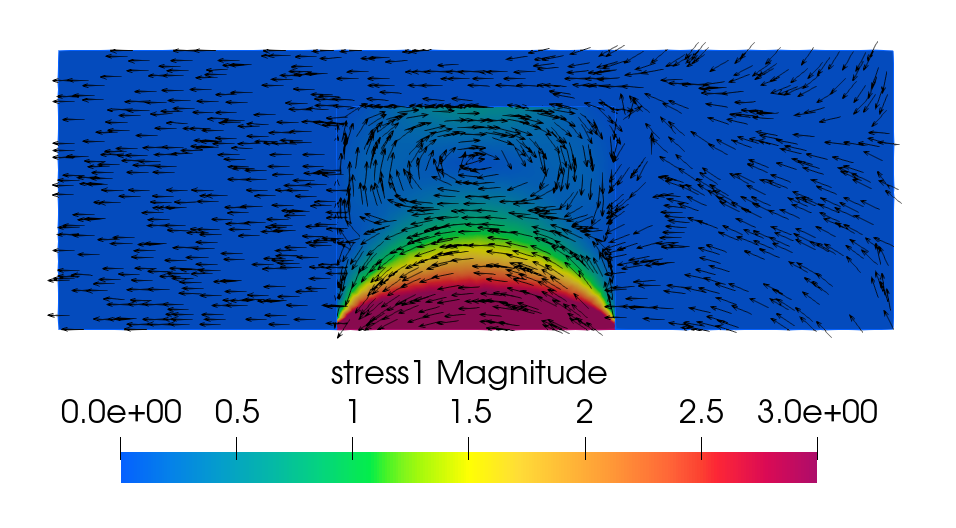}
\includegraphics[width=0.45\textwidth]{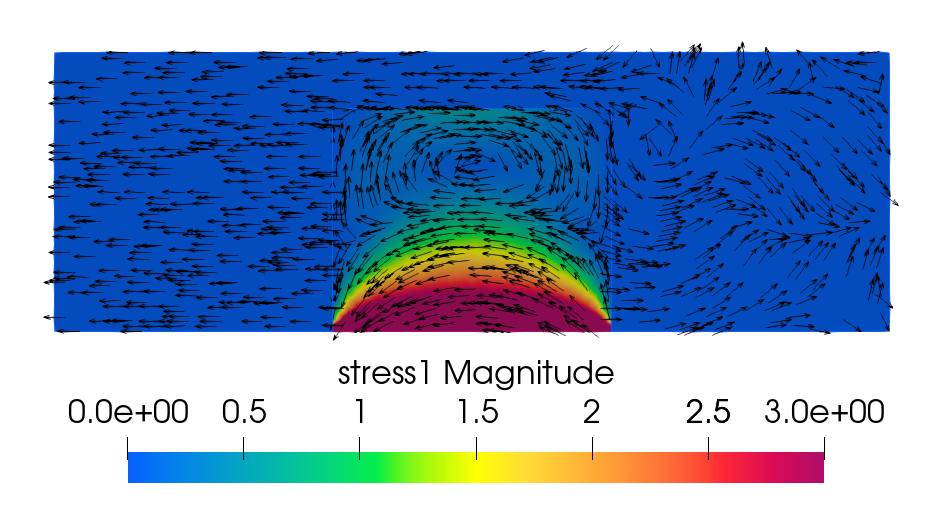}\\
\includegraphics[width=0.45\textwidth]{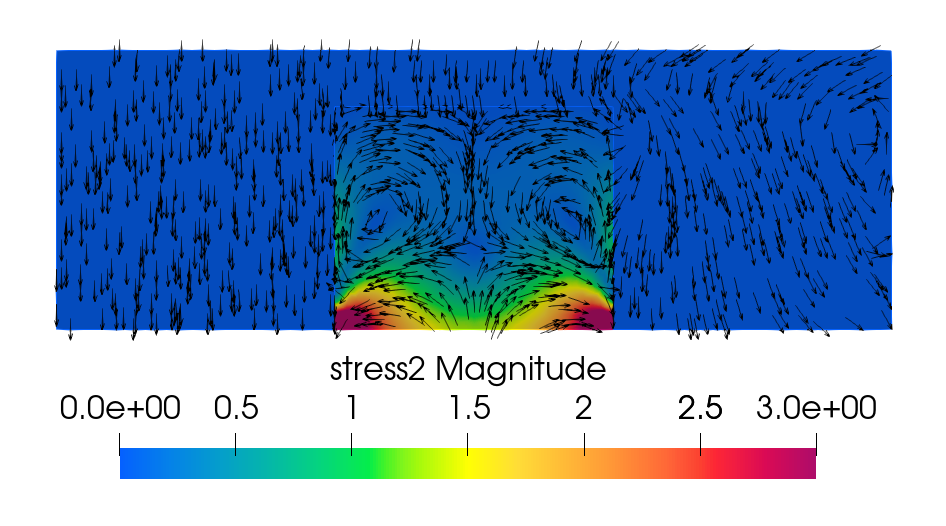}
\includegraphics[width=0.45\textwidth]{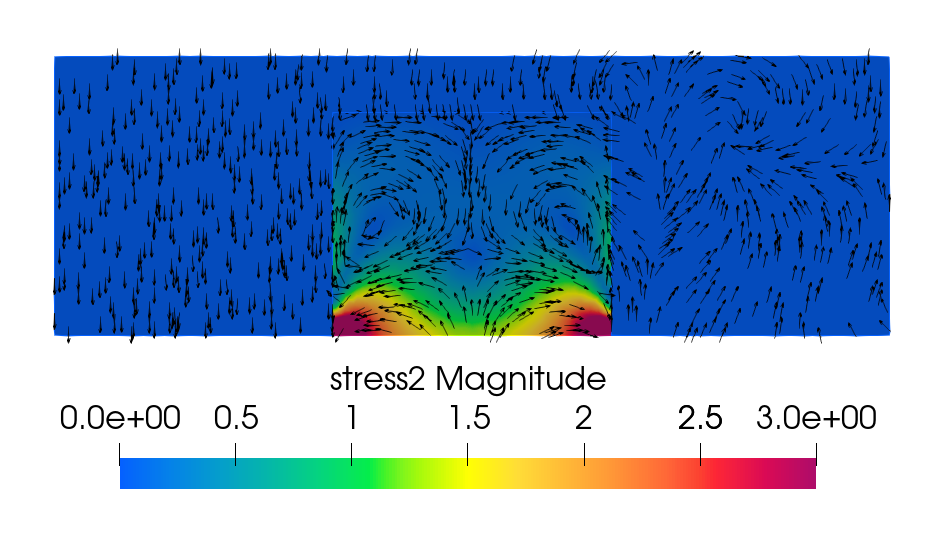}
\end{center}
\caption{Computed stress tensors $\bsi_{fh}+\alpha \,p_{ref} \bI$ and $\bsi_{ph} + \alpha \, p_{ref} \bI$ for the hard material (left) and soft material (right) at time T=80 s. Top: first rows of the stress tensors (arrows, not scaled) and their magnitudes (color); bottom: second rows of the stress tensors (arrows) and their magnitudes (color).}
\label{filter-stress}
\end{figure}

\begin{figure}[ht]
\begin{center}
\includegraphics[width=0.45\textwidth]{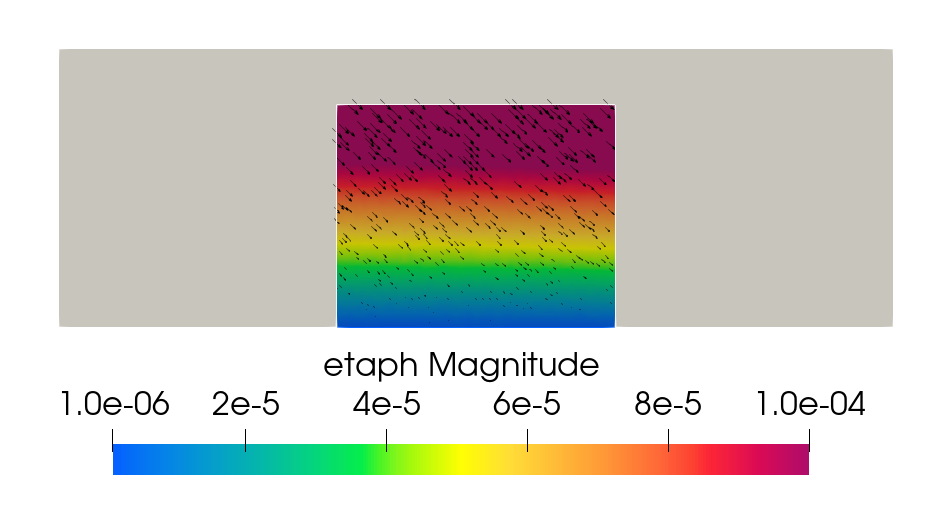}
\includegraphics[width=0.45\textwidth]{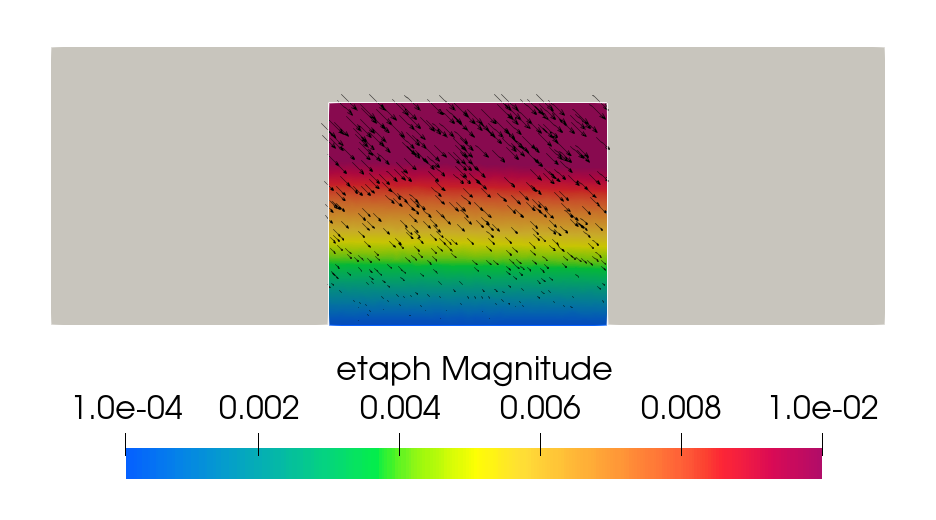}\\
\includegraphics[width=0.45\textwidth]{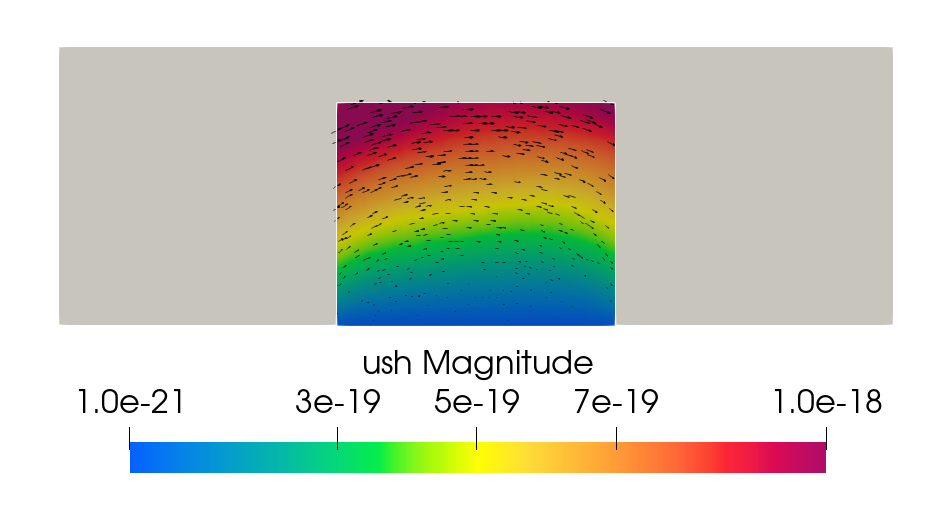}
\includegraphics[width=0.45\textwidth]{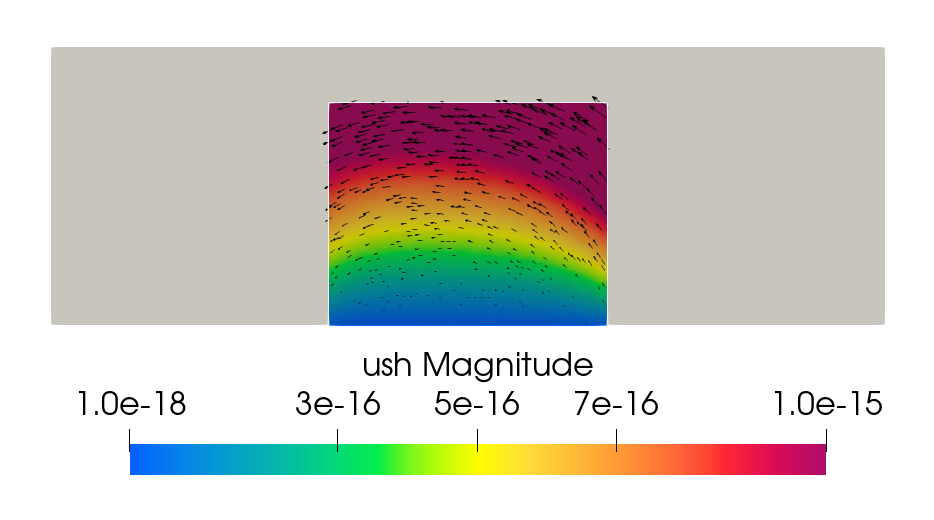}
\end{center}
\caption{Computed displacement $\bbeta_{ph}$ and structure velocity $\bu_{sh}$ for the hard material (left) and soft material (right) at time T=80 s.
Top: displacement (arrows) and its magnitude (color); bottom: structure velocity (arrows) and its magnitude (color).}
\label{filter-displacement}
\end{figure}

\section{Conclusions}

In this paper we develop an augmented fully mixed formulation for the quasistatic Navier--Stokes--Biot model and its mixed finite element approximation. The variables are pseudostress--velocity for Navier--Stokes, velocity--pressure for Darcy flow, and stress--displacement--rotation for elasticity. The traces of the structure velocity and the Darcy pressure on the interface are introduced as Lagrange multipliers to impose weakly the interface transmission conditions. In order to obtain control on the fluid pseudostress and velocity in their natural norms, the Navier--Stokes scheme is augmented with redundant Galerkin-type terms arising from the equilibrium and constitutive equations. The scheme exhibits local mass conservation for the Darcy fluid, local momentum conservation for the poroelastic stress, accurate approximations for the Darcy velocity, the poroelastic stress, and the fluid pseudostress with continuous normal components across element edges or faces, locking-free behavior, and robustness with respect to the physical parameters. We establish well-posedness of the weak formulation and its mixed finite element approximation, employing the semigroup theory for differential equations with monotone operators, combined with a fixed point argument. We further derive error estimates with rates of convergence. The presented numerical results verify the convergence rates and show the performance of the method for modeling blood flow in an arterial bifurcation and air flow through a filter using realistic parameters. The results illustrate the importance of including poroelastic and inertial effects in the model. Furthermore, we observe correct imposition of the interface conditions, accurate stress and velocity computation, and oscillation-free solution for parameters in the locking regime for poroelasticity.

\bibliographystyle{abbrv}
\bibliography{augmented-NS-Biot}

\begin{thebibliography}{10}

\bibitem{aggr2018}
M.~Alvarez, G.~N. Gatica, B.~Gomez-Vargas, and R.~Ruiz-Baier.
\newblock New mixed finite element methods for natural convection with
  phase-change in porous media.
\newblock {\em J. Sci. Comput.}, 80(1):141--174, 2019.

\bibitem{at1979}
M.~Amara and J.~M. Thomas.
\newblock Equilibrium finite elements for the linear elastic problem.
\newblock {\em Numer. Math.}, 33(4):367--383, 1979.

\bibitem{aeny2019}
I.~Ambartsumyan, V.~J. Ervin, T.~Nguyen, and I.~Yotov.
\newblock A nonlinear {S}tokes-{B}iot model for the interaction of a
  non-{N}ewtonian fluid with poroelastic media.
\newblock {\em ESAIM Math. Model. Numer. Anal.}, 53(6):1915--1955, 2019.

\bibitem{fpsi-transport}
I.~Ambartsumyan, E.~Khattatov, T.~Nguyen, and I.~Yotov.
\newblock Flow and transport in fractured poroelastic media.
\newblock {\em GEM Int. J. Geomath.}, 10(1):1--34, 2019.

\bibitem{akny2018-a}
I.~Ambartsumyan, E.~Khattatov, J.~M. Nordbotten, and I.~Yotov.
\newblock A multipoint stress mixed finite element method for elasticity on
  simplicial grids.
\newblock {\em SIAM J. Numer. Anal.}, 58(1):630--656, 2020.

\bibitem{msfmfe-Biot}
I.~Ambartsumyan, E.~Khattatov, and I.~Yotov.
\newblock A coupled multipoint stress--multipoint flux mixed finite element
  method for the {B}iot system of poroelasticity.
\newblock {\em Comput. Methods Appl. Mech. Engrg.}, 372:113407, 2020.

\bibitem{akyz2018}
I.~Ambartsumyan, E.~Khattatov, I.~Yotov, and P.~Zunino.
\newblock A {L}agrange multiplier method for a {S}tokes-{B}iot
  fluid-poroelastic structure interaction model.
\newblock {\em Numer. Math.}, 140(2):513--553, 2018.

\bibitem{abd1984}
D.~N. Arnold, F.~Brezzi, and J.~Douglas.
\newblock {PEERS}: a new mixed finite element for plane elasticity.
\newblock {\em Japan J. Appl. Math.}, 1(2):347--367, 1984.

\bibitem{afw2007}
D.~N. Arnold, R.~S. Falk, and R.~Winter.
\newblock Mixed finite element methods for linear elasticity with weakly
  imposed symmetry.
\newblock {\em Math. Comp.}, 76(260):1699--1723, 2007.

\bibitem{awanou2013}
G.~Awanou.
\newblock Rectangular mixed elements for elasticity with weakly imposed
  symmetry condition.
\newblock {\em Adv. Comput. Math.}, 38(2):351--367, 2013.

\bibitem{bqq2009}
S.~Badia, A.~Quaini, and A.~Quarteroni.
\newblock Coupling {B}iot and {N}avier-{S}tokes equations for modelling
  fluid-poroelastic media interaction.
\newblock {\em J. Comput. Phys.}, 228(21):7986--8014, 2009.

\bibitem{b1941}
M.~Biot.
\newblock General theory of three-dimensional consolidation.
\newblock {\em J. Appl. Phys.}, 12:155--164, 1941.

\bibitem{Bociu-etal-2021}
L.~Bociu, S.~Canic, B.~Muha, and J.~T. Webster.
\newblock Multilayered poroelasticity interacting with {S}tokes flow.
\newblock {\em SIAM J. Math. Anal.}, 53(6):6243--6279, 2021.

\bibitem{Brezzi-Fortin}
F.~Brezzi and M.~Fortin.
\newblock {\em Mixed and Hybrid Finite Element Methods}.
\newblock Springer Series in Computational Mathematics, 15. Springer-Verlag,
  New York, 1991.

\bibitem{bukavc2015}
M.~Buka{\v{c}}, I.~Yotov, R.~Zakerzadeh, and P.~Zunino.
\newblock Partitioning strategies for the interaction of a fluid with a
  poroelastic material based on a {N}itsche's coupling approach.
\newblock {\em Comput. Methods Appl. Mech. Eng.}, 292:138--170, 2015.

\bibitem{Bukac-JCP}
M.~Buka\v{c}.
\newblock A loosely-coupled scheme for the interaction between a fluid, elastic
  structure and poroelastic material.
\newblock {\em J. Comput. Phys.}, 313:377--399, 2016.

\bibitem{byz2015}
M.~Buka\v{c}, I.~Yotov, and P.~Zunino.
\newblock An operator splitting approach for the interaction between a fluid
  and a multilayered poroelastic structure.
\newblock {\em Numer. Methods Partial Differential Equations},
  31(4):1054--1100, 2015.

\bibitem{Buk-Yot-Zun-fracture}
M.~Buka\v{c}, I.~Yotov, and P.~Zunino.
\newblock Dimensional model reduction for flow through fractures in poroelastic
  media.
\newblock {\em ESAIM Math. Model. Numer. Anal.}, 51(4):1429--1471, 2017.

\bibitem{cort2017}
J.~Cama\~no, R.~Oyarz\'ua, R.~Ruiz-Baier, and G.~Tierra.
\newblock {Error analysis of an augmented mixed method for the
  {N}avier--{S}tokes problem with mixed boundary conditions}.
\newblock {\em IMA Journal of Numerical Analysis}, 38(3):1452--1484, 2017.

\bibitem{cgot2016}
J.~Cama{\~n}o, G.~Gatica, R.~Oyarz{\'u}a, and G.~Tierra.
\newblock An augmented mixed finite element method for the {N}avier--{S}tokes
  equations with variable viscosity.
\newblock {\em SIAM J. Numer. Anal.}, 54:1069--1092, 2016.

\bibitem{cot2016}
J.~Cama{\~n}o, R.~Oyarz{\'u}a, and G.~Tierra.
\newblock Analysis of an augmented mixed-{FEM} for the {N}avier--{S}tokes
  problem.
\newblock {\em Math. Comput.}, 86:589--615, 2017.

\bibitem{cgos2017}
S.~Caucao, G.~N. Gatica, R.~Oyarz{\'u}a, and I.~{\v S}ebestov{\'a}.
\newblock A fully-mixed finite element method for the
  {N}avier--{S}tokes/{D}arcy coupled problem with nonlinear viscosity.
\newblock {\em Journal of Numerical Mathematics}, 25(2):55 -- 88, 2017.

\bibitem{fpsi-fvca}
S.~Caucao, T.~Li, and I.~Yotov.
\newblock A cell-centered finite volume method for the
  {N}avier{\textendash}{S}tokes/{B}iot model.
\newblock In {\em Finite Volumes for Complex Applications {IX} - Methods,
  Theoretical Aspects, Examples}, pages 325--333. Springer International
  Publishing, 2020.

\bibitem{fpsi-msfmfe}
S.~Caucao, T.~Li, and I.~Yotov.
\newblock A multipoint stress-flux mixed finite element method for the
  {S}tokes-{B}iot model.
\newblock {\em Numer. Math.}, 152(2):411--473, 2022.

\bibitem{cesm2017}
A.~Cesmelioglu.
\newblock Analysis of the coupled {N}avier-{S}tokes/{B}iot problem.
\newblock {\em J. Math. Anal. Appl.}, 456(2):970--991, 2017.

\bibitem{Cesm-Chid}
A.~Cesmelioglu and P.~Chidyagwai.
\newblock Numerical analysis of the coupling of free fluid with a poroelastic
  material.
\newblock {\em Numer. Methods Partial Differential Equations}, 36(3):463--494,
  2020.

\bibitem{Cesm-etal-optim}
A.~Cesmelioglu, H.~Lee, A.~Quaini, K.~Wang, and S.-Y. Yi.
\newblock Optimization-based decoupling algorithms for a fluid-poroelastic
  system.
\newblock In {\em Topics in numerical partial differential equations and
  scientific computing}, volume 160 of {\em IMA Vol. Math. Appl.}, pages
  79--98. Springer, New York, 2016.

\bibitem{ciarlet1978}
P.~Ciarlet.
\newblock {\em The Finite Element Method for Elliptic Problems}.
\newblock Studies in Mathematics and its Applications, Vol. 4. North-Holland
  Publishing Co., Amsterdam-New York-Oxford, 1978.

\bibitem{cgg2010}
B.~Cockburn, J.~Gopalakrishnan, and J.~Guzm\'an.
\newblock A new elasticity element made for enforcing weak stress symmetry.
\newblock {\em Math. Comp.}, 79(271):1331--1349, 2010.

\bibitem{umfpack}
T.~Davis.
\newblock Algorithm 832: {UMFPACK} {V}4.3 - an unsymmetric-pattern multifrontal
  method.
\newblock {\em ACM Trans. Math. Software}, 30(2):196--199, 2004.

\bibitem{ervin2009}
V.~J. Ervin, E.~W. Jenkins, and S.~Sun.
\newblock Coupled generalized nonlinear {S}tokes flow with flow through a
  porous medium.
\newblock {\em SIAM J. Numer. Anal.}, 47(2):929--952, 2009.

\bibitem{galvis2007}
J.~Galvis and M.~Sarkis.
\newblock Non-matching mortar discretization analysis for the coupling
  {S}tokes-{D}arcy equations.
\newblock {\em Electron. Trans. Numer. Anal.}, 26:350--384, 2007.

\bibitem{Gatica}
G.~N. Gatica.
\newblock {\em A {S}imple {I}ntroduction to the {M}ixed {F}inite {E}lement
  {M}ethod. {T}heory and {A}pplications}.
\newblock Springer Briefs in Mathematics. Springer, Cham, 2014.

\bibitem{gmor2014}
G.~N. Gatica, A.~M\'arquez, R.~Oyarz\'ua, and R.~Rebolledo.
\newblock Analysis of an augmented fully-mixed approach for the coupling of
  quasi-{N}ewtonian fluids and porous media.
\newblock {\em Comput. Methods Appl. Mech. Engrg.}, 270:76--112, 2014.

\bibitem{gos2011}
G.~N. Gatica, R.~Oyarz\'ua, and F.~Sayas.
\newblock Analysis of fully-mixed finite element methods for the
  {S}tokes--{D}arcy coupled problem.
\newblock {\em Math. Comp.}, 80(276):1911--1948, 2011.

\bibitem{gov2020}
G.~N. Gatica, R.~Oyarz\'ua, and N.~Valenzuela.
\newblock A five-field augmented fully-mixed finite element method for the
  {N}avier--{S}tokes/{D}arcy coupled problem.
\newblock {\em Comput. Math. Appl.}, 80:1944--1963, 2020.

\bibitem{freefem}
F.~Hecht.
\newblock New development in {F}ree{F}em++.
\newblock {\em J. Numer. Math.}, 20(3-4):251--265, 2012.

\bibitem{Kunwar-etal}
H.~Kunwar, H.~Lee, and K.~Seelman.
\newblock Second-order time discretization for a coupled quasi-{N}ewtonian
  fluid-poroelastic system.
\newblock {\em Internat. J. Numer. Methods Fluids}, 92(7):687--702, 2020.

\bibitem{lee2016}
J.~Lee.
\newblock Robust error analysis of coupled mixed methods for {B}iot's
  consolidation model.
\newblock {\em J. Sci. Comput.}, 69(2):610--632, 2016.

\bibitem{fpsi-mixed-elast}
T.~Li and I.~Yotov.
\newblock A mixed elasticity formulation for fluid--poroelastic structure
  interaction.
\newblock {\em ESAIM Math. Model. Numer. Anal.}, 56(1):01--40, 2022.

\bibitem{Stokes-Biot-eye}
R.~Ruiz-Baier, M.~Taffetani, H.~D. Westermeyer, and I.~Yotov.
\newblock The {B}iot-{S}tokes coupling using total pressure: formulation,
  analysis and application to interfacial flow in the eye.
\newblock {\em Comput. Methods Appl. Mech. Engrg.}, 389:Paper No. 114384, 30,
  2022.

\bibitem{swgbh2020}
M.~Schneider, K.~Weishaupt, D.~Gl\"aser, W.~M. Boon, and R.~Helmig.
\newblock Coupling staggered-grid and {MPFA} finite volume methods for free
  flow/porous-medium flow problems.
\newblock {\em Journal of Computational Physics}, 401:109012, 2020.

\bibitem{sz1990}
L.~R. Scott and S.~Zhang.
\newblock Finite element interpolation of nonsmooth functions satisfying
  boundary conditions.
\newblock {\em Mathematics of Computation}, 54(190):483--493, 1990.

\bibitem{Seboldt-etal-2021}
A.~Seboldt, O.~Oyekole, J.~Tamba\v{c}a, and M.~Buka\v{c}.
\newblock Numerical modeling of the fluid-porohyperelastic structure
  interaction.
\newblock {\em SIAM J. Sci. Comput.}, 43(4):A2923--A2948, 2021.

\bibitem{Showalter}
R.~E. Showalter.
\newblock {\em Monotone {O}perators in {B}anach {S}pace and {N}onlinear
  {P}artial {D}ifferential {E}quations}.
\newblock Mathematical Surveys and Monographs, 49. American Mathematical
  Society, Providence, RI, 1997.

\bibitem{s2005}
R.~E. Showalter.
\newblock Poroelastic filtration coupled to {S}tokes flow.
\newblock {\em Control theory of partial differential equations. Lect. Notes
  Pure Appl. Math., 242, Chapman \& Hall/CRC, Boca Raton, FL}, pages 229--241,
  2005.

\bibitem{s2010}
R.~E. Showalter.
\newblock Nonlinear degenerate evolution equations in mixed formulations.
\newblock {\em SIAM J. Math. Anal.}, 42(5):2114--2131, 2010.

\bibitem{stenberg1988}
R.~Stenberg.
\newblock A family of mixed finite elements for the elasticity problem.
\newblock {\em Numer. Math.}, 53(5):513--538, 1988.

\bibitem{Wang-Yotov}
X.~Wang and I.~Yotov.
\newblock A {L}agrange multiplier method for the fully dynamic
  {N}avier--{S}tokes--{B}iot system.
\newblock Preprint.

\bibitem{Yi-Biot-mixed}
S.-Y. Yi.
\newblock Convergence analysis of a new mixed finite element method for
  {B}iot's consolidation model.
\newblock {\em Numer. Meth. Partial. Differ. Equ.}, 30(4):1189--1210, 2014.

\bibitem{Yi-Biot-locking}
S.-Y. Yi.
\newblock A study of two modes of locking in poroelasticity.
\newblock {\em SIAM J. Numer. Anal.}, 55(4):1915--1936, 2017.

\end{thebibliography}

\end{document}